%% file: the_real_main.tex
\documentclass[11pt,a4paper]{article}
\input{settings/packages_new}
\input{settings/styles_new}

\input{settings/macros_Maxi}

\input{settings/macros_fanny}

\begin{document}


\begin{center}
    {\LARGE
        \textsc{Posterior contraction\\ under misspecification and heteroscedasticity\\ in non-linear inverse problems}
    \par}

    \vspace{1cm}

 \begin{center}
{\Large\scshape
\begin{tabular}{c c c}
Fanny Seizilles\footnotemark[1] & \hspace{1cm}\&\hspace{1cm} & Maximilian Siebel\footnotemark[2] \\
[-0.3em]
{\small University of Cambridge} & & {\small Heidelberg University}
\end{tabular}
}
\end{center}

    \vspace{0.5cm}

    {\small
        \textsc{\today}
    \par}
\end{center}

\footnotetext[1]{DPMMS, University of Cambridge; e-mail: \href{mailto:fps25@cam.ac.uk}{fps25@cam.ac.uk}}
\footnotetext[2]{Institute for Mathematics, Heidelberg University; e-mail: \href{mailto:siebel@math.uni-heidelberg.de}{siebel@math.uni-heidelberg.de}}

\begin{abstract}
    In many practical and numerical inverse problems, the exact data log-likelihood is not fully accessible, motivating the use of surrogate models. We study heteroscedastic nonparametric nonlinear regression problems with Gaussian errors and establish contraction  results for posterior distributions arising from a surrogate log-likelihood constructed from proxy error variances, an approximate forward map, and an appropriate Gaussian process prior. Under general assumptions on the approximation quality, we show that the resulting surrogate posterior is statistically reliable and contracts about the true parameter at rates comparable to those of the exact posterior. The analysis leverages consistency properties of the (penalised) MLE to effectively handle heteroscedastic noise and to control the impact of likelihood approximation errors. We apply the framework to PDE-constrained inverse problems for a reaction–diffusion equation and the two-dimensional Navier–Stokes equation. In the latter case, we consider misspecified viscosity and forcing terms as well as Oseen-type linearization models, highlighting the relevance of our results for numerical analysis applications.
\end{abstract}

\tableofcontents

\input{Sections/_01_Introduction}

\input{Sections/_02_LSE_ErrorMisspecification}

\input{Sections/_03_GeneralBayes}

\input{Sections/_04_Examples}
\input{Sections/_05_Additional_proofs_General}

\section*{Funding}
 FS is funded by an ERC Advanced Grant (UKRI G116786) and MS is funded by Deutsche Forschungsgemeinschaft (DFG, German Research Foundation) under Germany’s Excellence Strategy EXC-2181/1-39090098 (the Heidelberg STRUCTURES Cluster of Excellence). 

 \section*{Acknowledgements}
The authors are grateful to Richard Nickl for suggesting this project and for many insightful discussions. They further thank Richard Nickl and Jan Johannes for facilitating reciprocal research visits to Cambridge and Heidelberg, respectively, and Robert Scheichl for proposing the inclusion of Oseen-type iterations as an illustrative example.


\printbibliography

\newpage
\appendix

\newcommand{\appendixtitlerule}{Appendix }
\titleformat{\section}
  {\normalfont\Large\bfseries}
  {\appendixtitlerule\Alph{section}:}{1em}{}

    \input{Sections/_app_Prior}
    \input{Sections/_app_MEstimation}
\input{Sections/_app_PDE}

    \input{Sections/app_miscellaneous}

\end{document}

%% file: settings/packages_new.tex
\usepackage[utf8]{inputenc}
\usepackage[T1]{fontenc}
\usepackage{lmodern}
\usepackage{amsmath, amsthm, amssymb}
\allowdisplaybreaks
\usepackage{mathtools}
\usepackage{geometry}
\usepackage{graphicx}
\usepackage{authblk}
\usepackage{xcolor}

\usepackage[
  backend=biber,
  style=authoryear,
  sortcites=true,
  sorting=ynt,
  maxcitenames=2,
  mincitenames=1,
   date=year,
  maxbibnames=99,
  uniquelist=false,
  uniquename=false
]{biblatex}

\usepackage[
  colorlinks=true,
  linkcolor=black,
  citecolor=blue
]{hyperref}

\usepackage{cleveref}
\AtBeginDocument{%
  \DeclareCiteCommand{\cite}
    {\usebibmacro{prenote}}
    {\bibhyperref{\textcolor{blue}{%
      \printnames{labelname}~%
      \mkbibparens{\usebibmacro{cite:labeldate+extradate}}%
    }}}
    {\multicitedelim}
    {\usebibmacro{postnote}}}

\usepackage{booktabs}
\usepackage{colortbl}
\usepackage{float}

\usepackage{etoolbox}
\usepackage{bbm}
\usepackage{xcolor}
\usepackage{amsfonts}
\usepackage{enumerate}   
\usepackage{enumitem}
\usepackage{stmaryrd}
\DeclareSymbolFont{stmry}{U}{stmry}{m}{n}
\SetSymbolFont{stmry}{bold}{U}{stmry}{m}{n}
\usepackage{mathrsfs} 
\usepackage{titlesec}
\usepackage{dsfont}
\usepackage{accents}
\usepackage{lipsum}



%% file: settings/styles_new.tex
\usepackage{thmtools}
\usepackage{mdframed}
\definecolor{cardinalred}{RGB}{140,21,21}
\definecolor{forestgreen}{RGB}{34,139,34}

\definecolor{titlegray}{RGB}{240,240,240}
\definecolor{firstcolgray}{RGB}{248,248,248}

\declaretheoremstyle[
  headfont=\color{cardinalred}\bfseries,
  bodyfont=\normalfont,
  qed=\raisebox{1pt}{\textcolor{cardinalred}{\openbox}}
]{theoremstyle}

\declaretheoremstyle[
  headfont=\color{black}\bfseries,
  bodyfont=\normalfont,
  qed=\raisebox{1pt}{\textcolor{black}{\openbox}}
]{conditionstyle}

\declaretheoremstyle[
  headfont=\color{forestgreen}\bfseries,
  bodyfont=\normalfont,
  qed=\raisebox{1pt}{\textcolor{forestgreen}{\openbox}}
]{examplestyle}

\declaretheorem[style=theoremstyle, numberwithin=section]{theorem}
\declaretheorem[style=theoremstyle, sibling=theorem]{lemma}
\declaretheorem[style=theoremstyle, sibling=theorem]{proposition}
\declaretheorem[style=theoremstyle, sibling=theorem]{corollary}
\declaretheorem[style=conditionstyle, sibling=theorem]{condition}
\declaretheorem[style=conditionstyle, sibling=theorem]{remark}
\declaretheorem[style=examplestyle, sibling=theorem]{example}

\theoremstyle{definition}

\theoremstyle{remark}

\newtheorem{notation}[theorem]{Notation}

\crefname{theorem}{\textsc{Theorem}}{\textsc{Theorems}}
\crefname{lemma}{\textsc{Lemma}}{\textsc{Lemmas}}
\crefname{proposition}{\textsc{Proposition}}{\textsc{Propositions}}
\crefname{corollary}{\textsc{Corollary}}{\textsc{Corollaries}}
\crefname{definition}{\textsc{Definition}}{\textsc{Definitions}}
\crefname{remark}{\textsc{Remark}}{\textsc{Remarks}}
\crefname{example}{\textsc{Example}}{\textsc{Examples}}
\crefname{equation}{\textsc{Equation}}{\textsc{Equations}}
\crefname{condition}{\textsc{Condition}}{\textsc{Conditions}}
\crefname{notation}{\textsc{Notation}}{\textsc{Notations}}
\crefname{section}{\textsc{Section}}{\textsc{Sections}}
\crefname{appendix}{\textsc{Appendix}}{\textsc{Appendix}}

\crefname{enumi}{}{}  
\Crefname{enumi}{}{}  

\crefname{equation}{\textsc{Eq.}}{\textsc{Eqs.}}
\Crefname{equation}{\textsc{Equation}}{\textsc{Equations}}

%% file: settings/macros_Maxi.tex
\newcommand{\IC}{\mathbb{C}}

\newcommand{\IE}{\mathbb{E}}

\newcommand{\IL}{L}

\newcommand{\IN}{\mathbb{N}}

\newcommand{\IP}{\mathbb{P}}
\newcommand{\IQ}{\mathbb{Q}}
\newcommand{\IR}{\mathbb{R}}

\newcommand{\IT}{\mathbb{T}}

\newcommand{\IZ}{\mathbb{Z}}

\newcommand{\rmB}{\mathrm{B}}

\newcommand{\rmD}{\mathrm{D}}

\newcommand{\rmJtilde}{\Tilde{\mathrm{J}}}

\newcommand{\rmM}{\mathrm{M}}

\newcommand{\rmU}{\mathrm{U}}

\newcommand{\rma}{\mathrm{a}}
\newcommand{\rmb}{\mathrm{b}}
\newcommand{\rmc}{\mathrm{c}}
\newcommand{\rmd}{\mathrm{d}}
\newcommand{\rme}{\mathrm{e}}

\newcommand{\rmv}{\mathrm{v}}

\newcommand{\calA}{\mathcal{A}}
\newcommand{\calB}{\mathcal{B}}

\newcommand{\calG}{\mathcal{G}}
\newcommand{\calGtilde}{\widetilde{\mathcal{G}}}
\newcommand{\calH}{\mathcal{H}}

\newcommand{\calJ}{\mathcal{J}}

\newcommand{\calL}{\mathcal{L}}
\newcommand{\calM}{\mathcal{M}}
\newcommand{\calN}{\mathcal{N}}
\newcommand{\calO}{\mathcal{O}}
\newcommand{\calP}{\mathcal{P}}

\newcommand{\calT}{\mathcal{T}}

\newcommand{\calW}{\mathcal{W}}

\newcommand{\calY}{\mathcal{Y}}
\newcommand{\calZ}{\mathcal{Z}}

\newcommand{\scrB}{\mathscr{B}}

\newcommand{\scrH}{\mathscr{H}}

\newcommand{\scrR}{\mathscr{R}}

\newcommand{\scrZ}{\mathscr{Z}}

\newcommand{\frakd}{d}
\newcommand{\frakdtilde}{\widetilde{\frakd}}

\newcommand{\frakt}{\mathfrak{t}}

\newcommand{\frakD}{D}
\newcommand{\frakDtilde}{\widetilde{\frakD}}

\makeatletter
\newcommand{\norm}[1]{%
  \| #1 \|%
  \@ifnextchar\bgroup 
    {\norm@i}
    {}
}

\newcommand{\norm@i}[1]{_{#1}}

\makeatother

\newcommand{\IRdnorm}[1]{\left|#1\right|_{\IR^d}}

\newcommand{\sprod}[3]{\left\langle #1, #2\right\rangle_{#3}}

\newcommand{\nset}[1]{\leq #1}

\newcommand{\domain}{\calO}

\newcommand{\distd}{\widetilde{\frakd}_{\delta,\ops}^2}

\newcommand{\sbar}{\Bar{s}_N}
\newcommand{\Mbar}{\Bar{M}}

\newcommand{\elltilde}{\Tilde{\ell}}

\newcommand{\abs}[1]{\left|#1\right|}
\newcommand{\pRz}{(0,\infty)}
\newcommand{\pRzz}{[0,\infty)}

\newcommand{\Vnorm}[1]{\left|#1\right|_V}
\newcommand{\Wnorm}[1]{\left|#1\right|_W}
\newcommand{\Vsprod}[2]{\sprod{#1}{#2}{V}}

\newcommand{\thetahat}{\hat{\theta}_{N}}

\newcommand{\thetatrue}{\theta_{0}}
\newcommand{\thetadiamond}{\theta^\diamond}
\newcommand{\Thetadiamond}{\Theta^\diamond}

\newcommand{\empmeasure}{\widehat{\IP}^{(i)}}

\newcommand{\br}[1]{\left(#1\right)}
\newcommand{\brK}[1]{\left\{#1\right\}}
\newcommand{\brE}[1]{\left[#1\right]}

\newcommand{\ops}{\operatorname{s}}

\DeclareMathOperator*{\argmin}{arg\,min}

\newcommand{\tikfunc}{\rmJtilde_{\delta,N,\ops}}

\newcommand{\ConstLiptwo}{C_{\operatorname{Lip},2}}
\newcommand{\ConstLipinfty}{C_{\operatorname{Lip},\infty}}
\newcommand{\Const}[1]{C_{\operatorname{#1}}}
\newcommand{\const}{c}

\newcommand{\Lfor}{\IL_\zeta^2(\calZ,V)}

\newcommand{\torus}{\IT^d}

\newcommand{\Hdot}{\dot{H}}

\newcommand{\dimW}{d_W}
\newcommand{\dimV}{d_V}

\newcommand{\Pitilde}{\widetilde{\Pi}}

\newcommand{\Hddiamond}{\dot{H}_{\diamond}}

\newcommand{\boldbeta}{\boldsymbol{\beta}}

%% file: settings/macros_fanny.tex
\newcommand{\loglikmodel}{\tilde \ell_N}

\newcommand{\Truedensityo}{\IP_{\thetatrue}^N}

\newcommand{\diffG}{\calGtilde(\theta_0)(Z_i)- \calGtilde(\theta)(Z_i)}
\newcommand{\diffnormG}{\norm{\calGtilde(\thetatrue)-\calGtilde(\theta)}_{\IL_\zeta^2}}

\newcommand{\deltanoise}{\tilde\delta_{\mathrm{noise},N}}
\newcommand{\deltamap}{\tilde\delta_{\mathrm{model},N}}

%% file: Sections/_01_Introduction.tex
\section{Introduction}\label{sec:intro}

\subsection{Statistical inference for nonlinear inverse problems}

A wide range of statistical inference problems arising in the natural sciences can be formulated
as \emph{nonlinear inverse problems}.
In such settings, an unknown parameter $\theta$ that is typically infinite-dimensional, is linked to
observable data through a nonlinear forward operator
\[
\calG \colon \Theta\ni\theta \mapsto \calG(\theta) \in \calY
\]
which models the response of a complex system to the parameter of interest.
A practically important class of nonlinear inverse problems arises when
the forward operator $\calG$ is defined implicitly as the solution map of a nonlinear
dynamical system governed by ordinary or partial differential equations; see, e.g.,
\cite{Temam_1997,Strogatz_2018}.
In this case, $\theta$ may represent an unknown initial condition, forcing term, or constitutive
parameter. Physical observations are typically indirect, noisy, and available only at finitely many `design' points $(t_i,x_i)$,
leading to regression-type models of the form
\begin{equation}\label{eq:intro:regression_general}
    Y_i = \calG(\theta)(t_i,x_i) + \varepsilon_i, \quad i=1,\dots,N,
\end{equation}
where $\varepsilon_1,\dots,\varepsilon_N$ are independent Gaussian measurement errors with noise variances $\sigma_1^2,\dots,\sigma_N^2$.
The statistical task then consists in recovering $\theta$ from noisy partial
observations as described by \cref{eq:intro:regression_general}.
From a deterministic and statistical perspective, inverse problems of this type have been studied
extensively; see, for instance, \cite{Engletal_2000, Kaipio_Somersalo_2005, Kaltenbacheretal_2008, stuart_2010, Arridgeetal_2019} and the references therein. A popular viewpoint is provided by the Bayesian approach. From this perspective, uncertainty about $\theta$ is encoded by placing a prior
distribution $\Pi$ on the parameter space $\Theta$, leading - via Bayes' theorem - to the posterior
distribution
\begin{equation}\label{eq:intro:posterior}
    \rmd\Pi(\theta\mid D_N) \propto \rme^{\ell_N(\theta)}\rmd\Pi(\theta),
\end{equation}
where $\ell_N(\theta)$ denotes the log-likelihood function with data $D_N = \brK{(Y_i,t_i,x_i)}_{i=1}^N$ associated to
\cref{eq:intro:regression_general}.
In infinite-dimensional inverse problems, Gaussian process priors are commonly employed, and
posterior contraction rates provide a first-order notion of frequentist validity.
For models, where the log-likelihood $\ell_N$ is fully accessible, posterior contraction for nonlinear inverse problems is by now
well understood in a number of settings, including problems governed by partial differential
equations. 
For a comprehensive overview of the (infinite dimensional) Bayesian methodology as well its
treatment in (non-)linear inverse problems, we refer to \cite{GhosalvanderVaart_2017} and \cite{Nickl_2023}, together with the references therein.\\


In practice, the exact log-likelihood function $\ell_N(\theta)$ may not be available.
Instead, statistical inference is then often based on \emph{approximate} or \emph{surrogate}
likelihoods, arising from numerical discretization of the forward operator, incomplete knowledge
of the noise distribution, or the use of pre-estimated noise levels.
Even in the idealized case of additive Gaussian errors, the exact evaluation of
$\ell_N(\theta)$ requires repeated solutions of the forward problem $\theta\mapsto \calG(\theta)$, which is often computationally
prohibitive in nonlinear or high-dimensional settings.
As a result, Bayesian inference is commonly carried out using a surrogate posterior $\Pitilde(\theta|D_N)$ computed as in \cref{eq:intro:posterior} but with a misspecified log-likelihood of
the form
\begin{equation}\label{eq:intro:misspec_likelihood}
    \tilde\ell_N(\theta)
    = -\frac{1}{2}\sum_{i=1}^N \frac{1}{s_i^2}\,
    \bigl|Y_i-\tilde{\calG}(\theta)(t_i,x_i)\bigr|^2,
\end{equation}
where $\tilde{\calG}$ denotes a numerical or otherwise approximate forward map replacing $\calG$ and
$s_1^2,\dots,s_N^2$ are surrogate noise variances.\\

Beyond modelling considerations, likelihood misspecification is also closely tied to
computational feasibility.
Posterior inference for Bayesian inverse problems typically relies on sampling-based algorithms
such as Markov chain or sequential Monte Carlo methods, all of which require repeated
evaluation of the (log-)likelihood; see, e.g.,
\cite{stuart_2010, Cotteretal_2013,Haireretal_2014, Nickl_Wang_2024, GiordanoWang_2025, CastreNickl_2026} and the references therein.
This observation has motivated a substantial literature on approximate Bayesian methods,
including noisy and pseudo-marginal MCMC algorithms, delayed acceptance schemes, and surrogate-based
approaches; see, for instance,
\cite{ChristenFox_2005, AndrieuRoberts_2009,Andrieuetal_2010}.\\

\subsection{Prior works and contributions}

Studying properties of posteriors in misspecified models is notoriously challenging, and obtaining quantitative contraction results often requires strong assumptions. In noise misspecification, for instance, previous approaches model an unknown noise using a Gaussian distribution: that is the case of \cite{Norets_2015} for heteroscedastic misspecified noise in nonparametric linear regression, or \cite{kleijn_2006, GhosalvanderVaart_2017} for nonparametric (nonlinear) regression. Such results however require stronger hypotheses, such as the parameter-to-observation map being uniformly bounded over the parameter space. 
The use of \textit{fractional posteriors} where the surrogate posterior is obtained by raising the likelihood to some power $\alpha\in(0,1)$, has also been investigated in this context: by reducing the weight of the data, this makes the posterior more robust to misspecification, and leads to posterior contraction results obtained in Rényi divergences (see e.g. \cite{Grunwald_2012, Bhattacharya_2019, miller_2019, Huillier_2023}). These divergences are however weaker metrics, so that in turn one requires stronger stability conditions if one aims to recover contraction at the level of the target parameter.

\bigskip

In a significant proportion of the Bayesian inference literature, misspecification arises when the data generating distribution $\IP_{\theta_0}$ does not belong to the class  $\{ \IQ_\theta, \theta\in \Theta\}$ of models under consideration.  In these situations, \cite{berk_1966, kleijn_2012} show that the posterior distribution typically concentrates around a pseudo-true parameter $\theta^*$ minimising the Kullback-Leibler divergence between the true distribution and the model class, i.e. $\theta^*=\argmin_{\theta \in \Theta} D_{KL}(\IP_{\theta_0}, \IQ_\theta)$. Contraction is established in a `Hellinger transform' testing divergence, which cannot be immediately related to less abstract metrics relevant in the case of nonlinear inverse problems, as the class of models associated to the loglikehood \cref{eq:intro:misspec_likelihood} is not convex.\\

The main contribution of this paper is to give sufficient, workable, conditions for reliable Bayesian inference in PDE-based inverse problems under likelihood misspecification. Precisely, we show posterior contraction of the resulting misspecified posterior
distribution
\[
    \rmd\Pitilde(\theta\mid D_N) \propto \rme^{\tilde\ell_N(\theta)}\rmd\Pi(\theta)
\]
directly around the true parameter $\thetatrue$ at the usual nonparametric rate for correctly specified problems, under a regime of \emph{mild misspecification} in which the `error' arising from either unknown Gaussian noise variance or an approximate PDE model decays sufficiently fast relative to the sample size $N$ and the ill-posedness of the inverse-problem. Adapting ideas in \cite{Nickl_2023} to such a situation, and employing suitable stability estimates, we then show posterior contraction for the nonlinear inverse problem, and prove convergence rates for the surrogate posterior mean $\IE^{\Pitilde}[\theta | D_N]$ towards the ground truth $\thetatrue$.
Our approach leverages the inherent robustness of the penalised MLE -- which has been shown to converge to the pseudo-true parameter even under model error (\cite{white_1982, kleijn_2012}) -- to build a sequence of test functions $\Psi_N = \Psi_N(D_N)$ whose type-I-error and type-II-error decay sufficiently fast (following the work of \cite{Siebel_2025, Nickl_vdGeer_Wang_2020}). Robustness of those point estimators is however not sufficient, and to obtain posterior contraction we further establish `change of measure' conditions, which ensure that the growth of certain likelihood terms is offset by the decay of the prior. Importantly, our proof methods for the change of measure enable us to consider PDE maps that are not uniformly bounded over the parameter space, through the use of a slicing argument. These conditions are also essential to prove convergence results of the mean of the surrogate posterior distribution. 

\medskip


\paragraph{Outline of the paper} In \cref{sec:setting} we introduce the general setting and notations. In \cref{sec:posterior-contraction} we establish the main surrogate posterior contraction theorem, with proofs given in \cref{sec:additional-proofs}. In \cref{sec:examples} we give three prototypical examples of misspecification for PDE-based inverse problems, respectively noise misspecification in the reaction-diffusion equation and model misspecification (via wrong parameter definition, and via numerical approximation) in the Navier-Stokes equation. Further results in \cref{app:MAP} revisit the known robustness of M-estimation techniques under misspecification which are needed to establish the correct convergence for the hypothesis tests.

%% file: Sections/_02_LSE_ErrorMisspecification.tex
\section{Setting} \label{sec:setting}

\subsection{Notation and preliminaries}
\begin{notation}[Preliminaries]
        We set $\IN_0\coloneqq \IN\cup\brK{0}$. 
         If $(S,\calT)$ is a topological vector space, we denote by $\scrB_S$ the Borel-$\sigma$-field on $S$ generated by the topology $\calT$. 
        The topological dual of $S$ is denoted by $S^*$ and consists of all linear and bounded functionals $L\colon S\to\IR$.
         Given two normed spaces $(S_1,\norm{\cdot}_{S_1})$ and $(S_2,\norm{\cdot}_{S_2})$, we write $S_1\hookrightarrow S_2$, if $S_1$ is continuously embedded into $S_2$. Further, for $M>0$ we write $S_1(M)\coloneqq\brK{s\in S_1:\;\norm{s}_{S_1}\leq M}$.
       Throughout, random variables are defined on a probability space $(\Omega,\calA,\IP)$ if not further mentioned. The expectation w.r.t. $\IP$ is denoted by $\IE$. Lastly, universal constants are denote by $\const$. If a constant depends on a family of objects $A$, we write $\const(A)$. Constants arising from assumptions are denoted by $\Const{}$ with some further specifications. If not further mentioned, the value of a constant can change from line to line.
\end{notation}

\subsection{Function spaces}

In the following let $d\in \IN$ be the fixed dimension. 
Let $(\calZ,\scrZ,\zeta)$ be any measurable space. We denote for $p\in[1,\infty)$ by $ \IL_{\zeta}^p(\calZ,\IC)\coloneqq\IL_{\zeta}^p(\calZ,\scrZ,\IC)$ the space of $p$-integrable functions $f\colon(\calZ,\scrZ,\zeta)\to\br{\IC,\scrB_\IC}$. In particular, the Hilbert space $\IL_\zeta^2(\calZ,\IC)$ is equipped with inner product 
\begin{equation*}
    \sprod{f}{g}{\IL_\zeta^2(\calZ,\IC)}\coloneqq \int_{\calZ}f(z)\overline{g(z)}\rmd\zeta(z),
\end{equation*}
where $\overline{w}$ denotes the complex conjugate of any $w\in\IC$. For simplicity, we write $\IL_\zeta^p(\calZ)$ for the corresponding space of real-valued functions.
Throughout, $\calM$ denotes either 
\begin{itemize}
    \item a bounded open set $\domain\subseteq\IR^d$ with smooth boundary $\partial\domain$, or 

    \item the $d$-dimensional torus $\IT^d\coloneqq [0,1]^d\setminus\sim_T$, where $\sim_T$ is the equivalence relation identifying opposite points.
\end{itemize}
Both spaces, equipped with their Borel-$\sigma$-field $\scrB_\domain$ and $\scrB_{\IT^d}$, and the Lebesgue measure $\calL^d$ on $\IR^d$ form measure spaces. We define for $m\in\IN_0$ the Banach space $C^m(\calM)$ of $m$-times differentiable functions $f\colon \calM\to\IR$ with bounded derivatives up to order $m$, equipped with the norm 
\begin{equation*}
    \norm{f}_{C^m(\calM)}\coloneqq\sum_{\boldbeta\in\IN_0^d:\;|\boldbeta| \leq m}\norm{\rmD^{\boldbeta} f}_{\infty},
\end{equation*}
where $\norm{\cdot}_\infty$ denotes the uniform norm.
For $s\in\IR_{\geq 0}$ this definition is extended by saying $f\in C^s(\calM)$, if $f\in C^{\lfloor s\rfloor}(\calM)$ and further $\rmD^{\boldbeta}u$ is $s-\lfloor s\rfloor$-Hölder continuous for $|\boldbeta|= \lfloor s\rfloor$. A norm on $C^s(\calM)$ is given by 
\begin{equation*}
    \norm{f}_{C^s(\calM)}\coloneqq \sum_{\boldbeta\in\IN_0^d:\;|\boldbeta|\leq\lfloor s\rfloor}\norm{\rmD^{\boldbeta} f}_\infty + \sum_{\boldbeta\in\IN_0^d:\; |\boldbeta| =\lfloor s\rfloor} \sup_{x,y\in \calM:\; x\neq y}\brK{\frac{|\rmD^{\boldbeta} f(x)-\rmD^{\boldbeta} f(y)|}{\IRdnorm{x-y}^{s-\lfloor s\rfloor}}}.
\end{equation*}
We call a function smooth if it belongs to $C^\infty(\calM)\coloneq\bigcap_{s>0}C^s(\calM)$. We denote by $C_c^\infty(\calM)$ the subspace of smooth functions with compact support, noting that $C_c^\infty(\torus) = C^\infty(\torus)$.
For $m\in\IN_0$, we define the usual Sobolev spaces of real-valued functions $u\colon\calM\to\IR$ with square integrable weak derivatives up to order $m$.  
 Note, $H^0(\calM) = \IL^2(\calM)\coloneqq\IL^2_{\calL^d}(\calM)$. For non-integer $s\geq0$, $H^s(\calM)$ is defined via interpolation, see \cite{Triebel_1983}. For $s<0$, we define 
  $  H^s(\calM) \coloneqq \br{H^{-s}(\calM)}^*$
as the topological dual.
For $\calM = \domain$, we further define for $s\geq 0$
            $H_c^s(\domain)\coloneqq \overline{C_c^\infty(\domain)}^{\norm{\cdot}_{H^s(\domain)}}$.
        Note, for $s\leq\frac{1}{2}$, we have $H_c^s(\domain) = H^s(\domain)$ and otherwise, if $s\in\IN$, $H^s_c(\domain)$ equals the subspace of $H^s(\domain)$ of vanishing trace on $\partial\domain$, see \cite{LionsMagenes_1972}.
    For $\calM = \IT^d$ being the $d$-dimensional torus, the periodic Laplacian $-\Delta$ diagonalises the Fourier basis, i.e. after enumerating $\IZ^d = \brK{k_j:\; j\in\IN}$ such that $j\mapsto k_j$ is non-decreasing, we have an orthonormal system
    \begin{equation}\label{eq:ONB:torus}
        e_{j}\coloneqq e_{k_j} \coloneqq \exp\br{-2\pi i\sprod{x}{k_j}{\IR^d}},\quad \lambda_0 = 0,\quad \lambda_j = 4\pi\IRdnorm{k_j}^2
    \end{equation}
    with imaginary unit $i = \sqrt{-1}\in\IC$. For $s\in\IR$, the Sobolev space $H^s(\torus)$ has a equivalent spectral norm given by 
    \begin{equation*}
        \norm{u}_{h^s(\torus)}^2 \coloneqq \sum_{j\in\IN}\br{1+\lambda_j}^s\abs{\sprod{u}{e_j}{\IL^2(\torus,\IC)}}_{\IC}^2.
    \end{equation*}
    In the subsequent, we will need \textit{homogeneous} Sobolev spaces $\Hdot^s(\torus)$, which are defined as subspace of $H^s(\torus)$ removing the zero mode $\lambda_0 = 0$. To that end, define the corresponding inner product 
    \begin{equation*}
        \sprod{f}{g}{\Hdot^s(\torus)}\coloneq\sum_{j\in\IN}\lambda_j^s\sprod{f}{e_j}{\IL^2(\torus,\IC)}\cdot\overline{\sprod{g}{e_j}{\IL^2(\torus,\IC)}},
    \end{equation*}
    noting that the iduced norm is given by 
    \begin{equation*}
        \norm{u}_{\Hdot^s(\torus)}\coloneq \sum_{j\in\IN}\lambda_j^s\abs{\sprod{u}{e_j}{\IL^2(\torus,\IC)}}_{\IC}^2 = \norm{(-\Delta)^{\frac{s}{2}}u}_{\IL^2(\torus)}^2.
    \end{equation*}
We generalize the previous definitions for real- or complex valued functions to functions with values in $W$, where $(W,\Wnorm{\cdot})$ is a finite-dimensional $\IC$-vector space with dimension $\dimW\coloneqq\operatorname{dim}_{\IC}(W)$. To that end, let $F(\calM)$ one of the the previous function spaces defined before. 
We then define the space of functions $f=(f_1,\dots,f_{\dimW})\colon \calM\to W$ as
\begin{equation*}
    F(\calM,W)\coloneqq \bigtimes_{i=1}^{\dimW}F(\calM),\quad \norm{f}_{F(\calM,W)}^2 \coloneqq \sum_{i\nset{\dimW}}\norm{f_i}_{F(\calM)}^2,
\end{equation*}
where we identify $W$ canonically with $\IC^{\dimW}$. In particular, if $F(\calM)$ is a Hilbert space with inner product $\sprod{\cdot}{\cdot}{F(\calM)}$, then $F(\calM,W)$ is an Hilbert space equipped with inner product 
\begin{equation*}
    \sprod{f}{g}{F(\calM,W)}\coloneqq \sum_{i\nset{\dimW}}\sprod{f_i}{g_i}{F(\calM)}.
\end{equation*}
Accordingly, these objects are also defined for $W$ being a finite-dimensional $\IR$-vector space.
\begin{itemize}
    \item For the analysis of the 2D-Navier-Stokes equation, we require Sobolev spaces with vanishing mean and divergence. To that end, we define
    \begin{equation*}
        \Hddiamond \coloneqq \brK{u\in\IL^2(\IT^2,\IR^2):\;\operatorname{div}(u) = 0,\;\int_{\torus}u_i(x)\rmd\calL^2(x) = 0\; \text{ for } i=1,2},
    \end{equation*}
    where $\operatorname{div}(u) \coloneqq \frac{\partial}{\partial x_1}u_1 + \frac{\partial}{\partial x_2}u_2$ denotes the divergence.
    As $( \Hddiamond,\sprod{\cdot}{\cdot}{\IL^2(\IT^2,\IR^2)})$ is a closed linear subspace of $\IL^2(\IT^2,\IR^2)$, we can define the $\IL^2$-projection operator, also called \textit{Leray}-operator
    \begin{equation} \label{eq:leray-operator}
        P\colon\IL^2(\IT^2,\IR^2)\to \Hddiamond.
    \end{equation}
    For any $s\geq 0$, we then define 
    \begin{equation*}
         \Hddiamond^s\coloneqq  \Hddiamond \cap H^s(\IT^2,\IR^2).
    \end{equation*}
\end{itemize}

Let $T>0$. If $X$ is a normed linear space, we further define the Bochner space $\IL^2([0,T],X)$ of measurable maps $h$ from $[0,T]$ to $X$, such that $\norm{h(\cdot)}_{X}$ is a map in $\IL^2\br{[0,T]}$. Analogously, we also define $C^0(\brE{0,T},X)$ as the space of continuous maps from $[0,T]$ to $X$, such that $\sup_{t\in(0,T)}\norm{h(t)}_{X}$ is finite.

\subsection{Observation model, Bayesian approach and mild misspecification}\label{subsec:ObservationModel}

Throughout, let $(V,\Vnorm{\cdot})\simeq(\IR^{d_V},\sprod{\cdot}{\cdot}{\IR^{d_V}})$ and $(W,\Wnorm{\cdot})\simeq(\IR^{d_W},\sprod{\cdot}{\cdot}{\IR^{d_W}})$ be two finite dimensional $\IR$-vector spaces with dimensions $d_V\in\IN$ and $d_W\in\IN$ respectively.  
Further, let $(\calZ,\scrZ,\zeta)$ be a probability space. In this work, we are interested in a parameter space $\Theta\subseteq\IL^2(\calM,W)$ and in a measurable and possibly non-linear forward map 
\begin{equation*}
    \calG:\Theta\to\IL^2_\zeta(\calZ,V),
\end{equation*}
such that pointwise evaluations for each $\theta\in\Theta$ of $\calG(\theta)(z)$ for $z\in\calZ$ are well-defined. We assume to have access to data coming from a random design regression model with heteroscedastic error, i.e., for a fixed sample size $N\in\IN$, we observe $D_N\coloneqq\br{Y_i,Z_i}_{i=1}^N\in\br{V\times\calZ}^N$ arising from 
 \begin{equation}\label{eq:setting:data}
    Y_i = \calG(\theta)(Z_i) + \varepsilon_i,\quad \theta\in\Theta,\quad i=1,\dots ,N.
\end{equation}
The covariates $(Z_i)_{i=1}^N$ are drawn identically and independently (i.i.d.) from the law $\zeta$ and assumed to be independent of the independent Gaussian errors $\br{\varepsilon_i}_{i=1}^N$, which satisfy $ \varepsilon_i\sim\calN\br{0,\sigma_i^2\operatorname{Id}_V}$ for $i=1,\dots,N$ with heteroscedastic variances $\sigma_1^2,\dots,\sigma_N^2 >0$. The law of the data vector $D_N$ is denoted by $\IP_\theta^N$, its corresponding expectation operator is denoted by $\IE_\theta^N$. The law and expectation of a single datum $(Y_i,Z_i)$ is denoted by $\IP_{\theta}^{(i)}$ and $\IE_{\theta}^{(i)}$, respectively. $\IP_\theta^N$ has a probability density function with respect to the measure $\bigotimes_{i=1}^N\big(\calL^{d_V}\otimes \zeta\big)$, which is for all $\theta\in\Theta,\;(y,z)\in(V\times\calZ)^N$ given by 
\begin{equation*}
    p_\theta^N(y,z)\coloneqq \prod_{i\nset{N}}p_\theta^{(i)}\coloneqq\prod_{i\nset{N}}\left\{(2\pi\sigma_i^2)^{-\frac{\dimV}{2}}\exp\left(-\frac{1}{2\sigma_i^2}\Vnorm{Y_i-\calG(\theta)(Z_i)}^2\right)\right\}.
\end{equation*}
We define a corresponding (re-scaled) log-likelihood $\ell_N(\theta)$ by 
\begin{equation}\label{eq:loglikelihood}
    \forall\theta\in\Theta:\;\ell_N(\theta) \coloneqq \sum_{i\nset{N}}\ell^{(i)}(\theta)\coloneqq-\frac{1}{2}\sum_{i\nset{N}}\sigma_i^{-2}\Vnorm{Y_i-\calG(\theta)(Z_i)}^2.
\end{equation}
The results obtained in this work hold under the frequentist assumption that the data $D_N$ is generated from the law $\IP_{\thetatrue}^N$ described by a fixed and unknown \textit{ground truth} $\thetatrue\in\Theta$, which we aim to recover.\\

For the Bayesian framework, we assume that $\Theta$ is equipped with a Borel-$\sigma$-field $\scrB_\Theta$. Let $\Pi'$ be a probability measure (called \textit{base prior}) on the measurable space $(\Theta,\scrB_\Theta)$. We use so-called \textit{re-scaled} priors, defined by
\begin{equation}\label{eq:rescaledprior}
    \Pi_N = \operatorname{Law}(\theta),\quad \theta = \frac{1}{\sqrt{N\delta_N^2}}\theta',\quad \theta'\sim\Pi',
\end{equation}
where $\delta_N>0$ is a sequence, such that $N\delta_N^2\to \infty$ as $N\to\infty$.
Assuming that the map 
\begin{equation*}
    \Theta\times\calZ\ni(\theta,z)\mapsto \calG(\theta)(z) \in V
\end{equation*}
is $\scrB_\Theta\otimes\scrZ-\scrB_V$ measurable, we introduce the associated \textit{posterior measure} $\Pi\br{\cdot|D_N}$ given by 
\begin{equation*}
    \forall B\in\scrB_\Theta:\, \Pi_N(B|D_N) = \frac{\int_B\rme^{\ell_N(\theta)}\rmd\Pi_N(\theta)}{\int_\Theta\rme^{\ell_N(\theta)}\rmd\Pi_N(\theta)}.
\end{equation*}


As discussed in \cref{sec:intro}, we investigate situations in which the error variances $\sigma_1^2,\dots,\sigma_N^2$ are not known and the forward map $\calG$ is only approximatively known. To that end, let $s_1^2,\dots,s_N^2>0$ be surrogate (proxy) variances and $\calGtilde\colon \Theta\to\IL^2_\zeta(\calZ,V)$ a surrogate (proxy) forward map, which is jointly  $\scrB_\Theta\otimes\scrZ-\scrB_V$ measurable as a map $\Theta\times\calZ\ni(\theta,z)\mapsto \calGtilde(\theta)(z) \in V$. Given the data $D_N\sim\IP_{\thetatrue}^N$, we then have a surrogate log-likelihood defined by 
\begin{equation}\label{eq:surrogateloglikelihood}
    \forall \theta\in\Theta:\;  \elltilde_N(\theta)\coloneqq  -\frac{1}{2}\sum_{i\nset{N}}s_i^{-2}\Vnorm{Y_i-\calGtilde(\theta)(Z_i)}^2,
\end{equation}
which we can evaluate numerically.
Note, $\elltilde_N$ is the (re-scaled) log-likelihood associated to the \textit{misspecified} regression model described by 
\begin{equation}\label{eq:wrongregmodel}
    Y_i = \calGtilde(\theta)(Z_i) + \Tilde{\varepsilon}_i, \quad\theta\in\Theta,\quad i=1,\dots,N,
\end{equation}
with independent $\Tilde{\varepsilon}_i\sim\operatorname{N}\br{0,s_i^2\operatorname{Id}_V}$. The law generating \cref{eq:wrongregmodel} is analogously denoted by $\IQ_\theta^N$.  Its probability density function w.r.t $\bigotimes_{i=1}^N\br{\calL^{d_V}\otimes \zeta}$ for $\theta\in\Theta,\;(y,z)\in(V\times\calZ)^N$ is given by 
\begin{equation*}
    q_\theta^N(y,z)\coloneqq \prod_{i\nset{N}}q_\theta^{(i)}(y_i,z_i)\coloneqq\prod_{i\nset{N}}\left\{(2\pi s_i^2)^{-\frac{\dimV}{2}}\exp\left(-\frac{1}{2 s_i^2}\Vnorm{Y_i-\calGtilde(\theta)(Z_i)}^2\right)\right\}.
\end{equation*} 
In other words, given the data $D_N\sim\IP_{\thetatrue}^N$ with not fully accessible log-likelihood $\ell_N$, we replace the \textit{true} log-likelihood $\ell_N$ by its surrogate $\elltilde_N$. In the Bayesian approach that means we look at the corresponding surrogate posterior distribution on $(\Theta,\scrB_\Theta)$, which we define as 
    \begin{equation}\label{eq:surrogateposterior}
        \forall B\in\scrB_\Theta:\, \Pitilde_N(B|D_N) = \frac{\int_B\rme^{ \elltilde_N(\theta)}\rmd\Pi_N(\theta)}{\int_\Theta\rme^{ \elltilde_N(\theta)}\rmd\Pi_N(\theta)}.
    \end{equation}

    The goal in this work is to show that $\Pitilde_N(\cdot|D_N)$ is statistical reliable under $\IP_{\thetatrue}^N$-probability to infer the unknown parameter of interest $\thetatrue$. To that end, we provide contraction results in \cref{sec:posterior-contraction}.
    
\subsection{Regularity conditions on the forward map}

We now impose analytical assumptions on the forward map $\calG$.

\begin{condition}[Forward Regularity]\label{ass:operatorI}
    \mbox{}\\
     Let $\Theta\subseteq\IL^2(\calM,W)$ be the parameter space. Let $(\scrR,\norm{\cdot}_{\scrR})$ be a separable normed subspace of $\Theta$ such that
    \[
        (\scrR,\norm{\cdot}_{\scrR})\hookrightarrow (B^\eta,\norm{\cdot}_{B^\eta}),
    \]
    where $B^\eta$ is either $C^\eta(\calM,W)$ or $H^\eta(\calM,W)$ for some $\eta\geq 0$.
    \begin{enumerate}[label={\textbf{[FR\arabic*]}}]
        \item\label{cond:fm:Liptwo} For all $M>0$ there exist constants $\ConstLiptwo(M)>0$ and $\kappa\geq 0$, such that for all $\theta_1,\theta_2\in\scrR(M)$
        \begin{equation*}
            \norm{\calG(\theta_1)-\calG(\theta_2)}_{\IL_\zeta^2(\calZ,V)} \leq \ConstLiptwo(M)\times\norm{\theta_1-\theta_2}_{(H^\kappa(\calM,W))^*}.
        \end{equation*}
    
        \item\label{cond:fm:bound} There exist constants $\Const{\calG,B}>0$ and $\gamma_B\geq 0$, such that for all $\theta\in\scrR$
        \begin{equation*}
            \norm{\calG(\theta)}_\infty \leq \Const{\calG,B}\times\br{1+\norm{\theta}_{\scrR}^{\gamma_B}}.
        \end{equation*}

        \item\label{cond:fm:Lipinf} For all $M>0$ there exist a constant $\ConstLipinfty(M)>0$, such that for all $\theta_1,\theta_2\in\scrR(M)$
        \begin{equation*}
            \norm{\calG(\theta_1)-\calG(\theta_2)}_{\infty} \leq \ConstLipinfty(M)\times\norm{\theta_1-\theta_2}_{B^{\eta}}.
        \end{equation*}
    \end{enumerate}
\end{condition}
\newpage
\begin{remark}[Forward Regularity]\label{rmk:forward-regularity}
    \mbox{}
    \begin{enumerate}[label = \roman*)]
        \item In the following, we refer to $(\scrR,\norm{\cdot}_{\scrR})$ as the \textit{regularization space}, which is typically the largest possible space so that the conditions in \cref{ass:operatorI} are still satisfied.
        
        \item Note, the general theory developed in \cite{Nickl_2023}
        utilizes \cref{cond:fm:Liptwo}, which implies that the induced prior is well specified on the set of admissible regression functions, and a weaker version of \cref{cond:fm:bound}, namely that $\calG$ is uniformly bounded on bounded balls of $\scrR$.
        In this work, we need to trace the dependence on $\theta$ more carefully, since later in \cref{cond:misspec-error} and \cref{cond:misspec-model} we require at most linear or quadratic (polynomial) growth in order to prove posterior contraction in misspecified models.
        Furthermore, we want to highlight that \cref{cond:fm:bound} includes uniformly bounded forward maps ($\gamma_B = 0$), such as the solution maps associated to PDE-constrained regression models driven by the Darcy problem and the time-(in)dependent Schrödinger equation (see \cite{Nickl_vdGeer_Wang_2020,Kekkonen_2022}).
        In \cref{sec:examples}, we will apply the present general theory to PDE-constrained regression models driven by non-linear reaction diffusion equation and the 2D-Navier-Stokes equation. 
        For the latter, it is shown in \cite{NicklTiti_2024} and \cite{Konen_Nickl_2025} that the corresponding solution map, mapping the initial condition $\theta$ to the solution $u_\theta$ of the dynamical system satisfies the assumptions imposed in \cref{ass:operatorI}, particularly \cref{cond:fm:bound} with some $\gamma_B >0$. Following the theory provided in \cite{Nickl_2025_BvMRDE}, where \cref{cond:fm:Liptwo} and \cref{cond:fm:Lipinf} are shown for the solution map of the non-linear reaction diffusion equation, we derive \cref{cond:fm:bound} in \cref{lemma:prop5new}.
        

    \end{enumerate}

\end{remark}







\subsection{Conditions on the Prior}

\begin{condition}[Base Prior]\label{cond:BasePrior}\label{cond:MAP:scrH}
    Under the conditions imposed in \cref{ass:operatorI}, let $\Pi'$ be a centered Gaussian measure on the linear subspace $\Theta\subseteq\IL^2(\calM,W)$ with \textit{reproducing kernel Hilbert space} (RKHS) $\scrH$, such that $\scrH\hookrightarrow \scrR$. For some $\alpha>0$ assume that 
    either 
    \begin{equation*}
        \scrH \hookrightarrow H_c^\alpha(\calM,W), \quad\text{if $\kappa\geq \frac{1}{2}$},\quad\text{or}\quad  \scrH \hookrightarrow H^\alpha(\calM,W), \quad\text{if $\kappa < \frac{1}{2}$.}
    \end{equation*}
    Further, assume that
    \begin{equation*}
        \Pi'\br{\theta\in\Theta:\; \norm{\theta}_{\scrR} < \infty} = 1.
    \end{equation*}
\end{condition}

\begin{remark}
    \mbox{}
    \begin{enumerate}[label = \roman*)]
        \item Typical choices of $\scrH$ and $\scrR$ from \cref{cond:BasePrior} we have in mind are $\scrH = H^\alpha(\calM,W)$ and $\scrR = H^\beta(\calM,W)$ for appropriate $\alpha>\beta$, which will be particularly important in \cref{sec:examples}.
    
        \item Several constructions for Gaussian priors that satisfy \cref{cond:BasePrior} have been discussed in \cite{Nickl_2023}. In \cref{sec:app:prior}, we summarize these discussions including explicit constructions that are suitable for the Reaction Diffusion Equation and the 2D-Navier-Stokes equation, see particularly \cref{example:prior:generic} and \cref{example:prior:PDE}. The results in this work are presented for Gaussian process priors, while finite-dimensional (\textit{sieve}) priors could also be used with some minor changes in the proofs, see also \cref{rem:app:prior} \cref{rem:prior:item:finitediemsnional}.

    \end{enumerate}
\end{remark}

%% file: Sections/_03_GeneralBayes.tex
\section{Posterior Contraction for misspecified models}\label{sec:posterior-contraction}

 In the rest of the section, let $\calG$ be a forward map satisfying the forward regularity conditions formulated in \cref{ass:operatorI} for some $\kappa\geq 0$, and $\Pi'$ be a base Gaussian prior satisfying \cref{cond:BasePrior} with $\alpha >0$. Let $\Pi_N$ be the corresponding sequence of rescaled priors from \cref{eq:rescaledprior} with
\begin{equation}\label{eq:minimax-rate}
    \delta_N = N^{-\frac{\alpha+\kappa}{2\alpha +2\kappa +d}}.
\end{equation}

In this section, the main result, \cref{thm:posterior-contraction} consists of a posterior contraction theorem at rate $\delta_N$ around the fixed ground truth $\thetatrue\in \Theta$ for the misspecified posterior distribution $\Pitilde_N\br{\cdot\mid D_N}$ under $\IP_{\thetatrue}^N$ probability, provided the misspecification of the noise variance ($s_i^2$ for $\sigma_i^2$) or approximate map ($\calGtilde$ for $\calG$) is sufficiently small. To that end, we need some assumptions on the noise variances $\sigma_1^2,\dots,\sigma_N^2$ as well as on the level of misspecification.\medskip

\begin{condition}[Noise variances]\label{cond:main:variances}
    \mbox{}
    \begin{enumerate}[label={\textbf{[NV]}}]
        \item\label{cond:error:variances} The error variances $\sigma_1^2,\dots,\sigma_N^2>0$ satisfy 
        \begin{equation*}
            0< \sigma_0^2 \coloneqq \min_{i\nset{N}}\sigma_i^2 \leq \max_{i\nset{N}}\sigma_i^2 \eqqcolon \sigma_\infty^2.
        \end{equation*}
    \end{enumerate}
\end{condition}

We will consider \cref{cond:error:variances} to hold implicitly in the rest of the paper.

\begin{condition}[Noise misspecification]\label{cond:misspec-error}
For $N\in\IN$, let $s_1^2,\dots,s_N^2>0$ be the sequence of proxy variances used in place of $\sigma_1^2,...,\sigma_N^2$.

\begin{enumerate}[label={\textbf{[NM\arabic*]}}]
    \item\label{item:NMI} Let $0<s_0^2 := \min_{i\nset{N}}s_i^2 \leq \max_{i\nset{N}}s_i^2 \eqqcolon s_\infty^2$ such that $\bar{s}_N^{-2}:=\frac{1}{N}\sum\limits_{i=1}^N s_i^{-2} \leq s_0^{-2}$.
    \item\label{item:NMII} We have 
        \begin{equation*}
            \max_{i\nset{N}}\left|1-\frac{\sigma_i^2}{s_i^2}\right| = \deltanoise
        \end{equation*}
        with some sequence $\deltanoise>0$, and either 
        \begin{enumerate}[label={\textbf{[NM2.\arabic*]}}]
            \item\label{item:NMII.I} the variance is consistently \textit{overestimated}, that is  $s_i^2 > \sigma_i^2$ for all $i\nset{N}$, and $\deltanoise\to 0$ as $N\to\infty$;

            \item \label{item:NMII.II}or, $\deltanoise \leq \Const{\mathrm{noise}}\times \delta_N^2$ for a sufficiently small constant $\Const{\mathrm{noise}}>0$, and the proxy forward $\calGtilde$ satisfies \cref{cond:fm:bound} with $\gamma_B\in[0,1]$.
        \end{enumerate}
\end{enumerate}

\end{condition}

\begin{condition}[Model misspecification]\label{cond:misspec-model} 
    Let $\calGtilde_\cdot$ be as in \cref{subsec:ObservationModel}.
     Further:
    \begin{enumerate}[label={\textbf{[MM\arabic*]}}]
    
        \item\label{cond:MM:I} The proxy operator $\calGtilde$ satisfies \cref{cond:fm:bound} with $\gamma_B\in[0,2]$.
        \item\label{cond:MM:III} Let $\rmM>0$. There exists a constant $\const(\rmM)>0$ and sequence of $\deltamap >0$ such that $\norm{\mathcal{G}(\theta) - \tilde{\mathcal{G}}(\theta)}_\infty \leq \const(\rmM)\times\deltamap$ for all $\theta\in\scrR(\rmM)$, with $\deltamap\leq\Const{\mathrm{model}}\times\delta_N^2$, for some constant $\Const{\mathrm{model}}>0$ sufficiently small.
    \end{enumerate}
    
\end{condition}
\newpage
\begin{remark}[Interpretation of \cref{cond:misspec-error} and \cref{cond:misspec-model}]
\mbox{}
\begin{enumerate}
    \item[i)] \cref{item:NMI} means that the surrogate variance has similar bounds as the true variance from \cref{cond:error:variances}, and prevents the use of proxies that would vastly underestimate the correct ones. Concerning \cref{item:NMII.I}, overestimating the variance can be related to the approach of fractional posteriors mentioned in \cref{sec:intro}, as indeed it amounts to using $s_i^2=\sigma_i^2/\alpha_N$ as proxy variance for some $0<\alpha_N<1$ with $\alpha_N \to 1$, effectively raising the original likelihood \cref{eq:loglikelihood} to the power $\alpha_N$.


    \item[ii)] The `smallness' condition on $C_{\mathrm{noise}}$ and $C_{\mathrm{model}}$ is in regard to the small ball exponent of the prior (as defined in \cref{eq:prior-tail}), and will play a part in the proofs of \cref{prop:change-measure-peeling} and \cref{thm:posterior-contraction-inverse}, \cref{eq:posterior-mean}. For the examples in \cref{sec:examples} we ignore this technicality, replacing these constants with $1/(\log N)$.
\end{enumerate}

\end{remark}

\subsection{Preliminary results}\label{subsec:prem:results}

We start with some preliminary results. Following the proof strategies in \cite{GhosalvanderVaart_2017} and \cite{Nickl_2023}, a standard posterior contraction proof relies on two main conditions: a \textit{small ball condition} and the \textit{existence of tests}. In presence of misspecification, we exhibit a third so called \textit{change of measure} condition which is crucial in nonlinear problems to control the behaviour of the posterior distribution. We show that under \cref{cond:misspec-error} and \cref{cond:misspec-model}, these requirements are satisfied. For conciseness we state here the key lemmas and propositions; remaining proofs can be found in detail in \cref{sec:additional-proofs}.\\

We define for any $\theta_1,\theta_2\in\Theta$ the shorthand notation
\begin{equation*}
    d_\calG(\theta_1,\theta_2)\coloneqq\norm{\calG(\theta_1)-\calG(\theta_2)}_{\IL^2_\zeta(\calZ,V)},
\end{equation*}
noting that this defines a semi-metric on the parameter-space $\Theta$. We analogously define $d_{\calGtilde}$. Given a fixed constant $\rmU>0$, we define the sets 
\begin{equation*}
    \widetilde{\calB}_N \coloneqq \brK{\theta\in\Theta:\;d_{\calGtilde}(\theta,\thetatrue)\leq\delta_N,\;\norm{\calGtilde_\theta}_\infty\leq\rmU}.
\end{equation*}

\paragraph{Small ball computations}

\begin{proposition}[Information Inequality]\label{prop:information-ineq}
     Under the misspecification assumptions \cref{cond:misspec-error} and \cref{cond:misspec-model}, we have the following properties.
\begin{enumerate}[label = \roman*)]
    \item There exists a constant $\const_1 = \const_1(\thetatrue,s_0^2)>0$, such that for all $\theta\in\widetilde{\calB}_N$
        \begin{equation*}
            -\IE_{\thetatrue}^N\brE{\log \br{\frac{q^N_\theta}{q^N_{\thetatrue}}}} \leq \frac{1}{2}N s_0^{-2}\times d_{\calGtilde}(\theta,\thetatrue)^2 + \const_1 N \deltamap \times d_{\calGtilde}(\theta,\thetatrue).
        \end{equation*}
    \item There exists a constant $\const_2 = \const_2(\rmU, s_0^2, \sigma_\infty^2)>0$, such that for all $\theta\in\widetilde{\calB}_N$ 
    \begin{equation*}
       \forall i\nset{N}:\;\IE_{\theta}^{(i)}\brE{\left(\log \br{\frac{q_{\theta}^{(i)}}{q_{\thetatrue}^{(i)}}} -\IE_{\theta}^{(i)}\log \br{\frac{q_{\theta}^{(i)}}{q_{\thetatrue}^{(i)}}}\right)^2} \leq \const_2\times d_{\calGtilde}(\theta,\thetatrue)^2 .
    \end{equation*}
\end{enumerate}
\end{proposition}

Note that the noise variance misspecification does not play a part here, beyond affecting the multiplicative constant in front of $d_{\calGtilde}(\theta,\thetatrue)$. From the proposition one can then easily derive the following auxiliary lemma, which shows that the denominator in the formula of the posterior measure is bounded away from $0$ on events of high $\IP_{\thetatrue}^N$-probability. \\

\begin{lemma}[Auxiliary contraction]\label{lemma:contraction-auxiliary} In the setting of \cref{prop:information-ineq}, let $\nu$ be a probability measure on some (measurable) subset $B_N \subseteq \widetilde{\mathcal{B}}_N$. For the surrogate log-likelihood $\elltilde_N$ from \cref{eq:surrogateloglikelihood}, we have for all $K> s_0^{-2}$
    
    \begin{equation*}
        \IP_{\thetatrue}^N\left(\int_{B_N} e^{\tilde\ell_N(\theta)-\tilde\ell_N(\thetatrue)}\rmd\nu(\theta) \leq e^{-KN\delta_N^2}\right) \xrightarrow[N\to \infty]{}0.
    \end{equation*}
\end{lemma}

\paragraph{Existence of Tests}

As already discussed above, in homoscedastic models, the existence of tests $\Psi_N$ is a standard result, which follows for instance from Theorem 7.1.4 in \cite{Gine_Nickl_2021}. Due to the heteroscedasticity of the observation scheme \cref{eq:intro:regression_general}, we need to construct tests explicitly, whose type-I and type-II errors are controlled sufficiently. \\

We define introduce the following \textit{regularization sets}. Given $m>0$, we define
\begin{equation}\label{eq:reg-set}
    \Theta_N(m) \coloneqq \brK{\theta \in \scrR :\; \theta=\theta_1+\theta_2,\; \norm{\theta_1}_{\br{H^\kappa(\calM,W)}^*} \leq m\delta_N, \;\norm{\theta_2}_\scrH \leq m,\; \norm{\theta}_\scrR\leq m}.
\end{equation}

\begin{proposition}[Existence of Tests]\label{prop:testing}
Let $\calG$ satisfy \cref{cond:fm:Liptwo}-\cref{cond:fm:Lipinf}. Let $\scrH$ be the RKHS from \cref{cond:BasePrior} with  $\alpha>\eta + d$. Let $N\in\IN$ and assume \cref{cond:error:variances}. Let $D_N\sim\IP_{\thetatrue}^N$ with fixed $\thetatrue\in\scrH$. Let $\delta_N$ as in \cref{eq:minimax-rate}. 
Given $\Bar{c}>0$, there exist a sequence of tests (indicator functions) $\Psi_N = \Psi_N(D_N)$, such that 
    \begin{equation*}
        \lim_{N\to\infty}\IE_{\thetatrue}^N\brE{\Psi_N} = 0 \quad \text{ and }\quad  \sup_{\substack{\theta\in\Theta_N(m):\;\norm{\calG(\theta)-\calG(\thetatrue)}{\IL_\zeta^2(\calZ,V)}\geq \rho\delta_N}}\IE_\theta^N[1-\Psi_N]  \lesssim \exp\br{-\Bar{\const}N\delta_N^2}
        \end{equation*}
    for all $\rho = \rho(\texttt{S}),m=m(\thetatrue)>0$ and $N$ sufficiently large, where
    \begin{equation*}
        \texttt{S}\coloneqq\brK{\Bar{c},\alpha,\gamma_B,\kappa,\eta,d,d_W,m,\Const{\calG,\rmB},\ConstLiptwo,\ConstLipinfty,\Const{var},\sigma_0,\sigma_\infty}.
    \end{equation*}
\end{proposition}

For the proof of \cref{prop:testing}, we use concentration properties of estimators as proposed in \cite{gine_2011}. We derive in \cref{co:GeneralConsitency} that the maximizer $\thetahat$ of the following Tikhonov-type-functional
\begin{equation*}
    \scrH(m)\ni\theta\mapsto - \frac{1}{2N}\sum_{i\nset{N}}\Vnorm{Y_i-\calG(\theta)(Z_i)}^2 - \frac{\delta_N^2}{2}\norm{\theta}_{\scrH}^2,\quad \text{for }\;m>0\;\text{ sufficiently large}
\end{equation*}
defined on balls $\scrH(m)$ of the RKHS $\scrH$, exists and is consistent in the sense that 
\begin{equation*}
    \IP_{\thetatrue}^N\br{d_{\delta_N}^2(\thetahat,\thetatrue)\geq c\delta_N^2} \xrightarrow[]{N\to\infty} 0,\quad d_{\delta_N}^2(\thetahat,\thetatrue) \coloneqq d_{\calG}(\thetahat,\thetatrue)+\delta_N^2\norm{\thetahat}_{\scrH}^2
\end{equation*}
for some $c>0$ sufficiently large.
Defining the events $A_N\coloneqq\brK{d_{\delta_N}^2(\thetahat,\thetatrue)\geq c\delta_N^2}$, we show in \cref{co:existence:test} that the resulting sequence of tests $\Psi_N \coloneqq \mathds{1}_{A_N}$ has the desired properties. In fact, in \cref{co:existence:test} we can abandon the Gaussian assumption on the measurement errors $\varepsilon_1,\dots,\varepsilon_N$ and require only a Bernstein condition (see \cref{ass:errorBernstein}), due to the well-known robustness of M-estimation techniques, which we again demonstrate in \cref{thm:LS_estimator} and \cref{co:GeneralConsitency}, respectively.


\paragraph{Change of measure}


In the proof of the main theorem \cref{thm:posterior-contraction}, it will become apparent that conditions are required control the effect of misspecification on the (log-)likelihood: this is the purpose of the following two propositions.

\begin{proposition}[Change of measure I: Inside of the regularization set]\label{prop:change-measure}
    \mbox{}\\
    Suppose \cref{cond:misspec-error} and \cref{cond:misspec-model} are satisfied.
    Then, for all $M>0$ and $b>0$, there exist $c_5=c_5(b, s_0^{-2}, \sigma_\infty^2 M, \Const{\calGtilde}, \Const{noise},\Const{model})>0$ such that 
    \begin{equation*}
       \forall\theta\in\scrR(M):\; \mathbb{E}_{\theta}^N\left[\left(\frac{q^N_\theta}{q^N_{\thetatrue}}\frac{p^N_{\thetatrue}}{p^N_{\theta}}\right)^b\right] \leq \exp\br{ c_5 \times N\delta_N^2}
    \end{equation*}
    %
%

\end{proposition}

\begin{proposition}[Change of measure II: Outside of the regularisation set]\label{prop:change-measure-peeling}
 \mbox{}\\
    Grant \cref{cond:misspec-error} and \cref{cond:misspec-model}, with constants $C_{\mathrm{noise}}, C_{\mathrm{model}}$ small enough compared to the small ball exponent of the prior (see \cref{eq:prior-tail}). Let $M>0$. Then there exists $c_6 = c_6(M, C_{\calGtilde, B}, \thetatrue, s_0^{-2}, \Const{\mathrm{noise}}, \Const{\mathrm{model}})>0$ such that
    \begin{equation*}
        \int_{\Theta_N(M)^c} \IE_{\thetatrue}^N\brE{e^{\loglikmodel(\theta)-\loglikmodel(\thetatrue)}}\rmd\Pi_N(\theta) \leq \exp\br{-c_6\times N\delta_N^2}
    \end{equation*}
    and where the constant $c_6$ can be made as large as desired by increasing $M$.
\end{proposition}

\begin{remark}
    \cref{prop:change-measure-peeling} is reminiscent of Equation (2.13) in \cite{kleijn_2006}. Here, working with the slicing technique enables us to cover a wider range of PDE problems; not necessarily uniformly bounded over the parameter space.
\end{remark}

\subsection{Basic contraction theorem}
A final requirement is a mass condition on the prior:

\begin{proposition}\label{prop:prior-mass-bn}
Let $\Pi_N$ be the sequence of rescaled Gaussian priors as above with $\delta_N$ as in \cref{eq:minimax-rate}, such that $N\delta_N^2\to\infty$ as $N\to\infty$. Under \cref{cond:misspec-model}, there exists some $A = A(d_W,\calGtilde,\calG,\thetatrue)>0$, such that for all $N$ large enough, we have
    \begin{equation}\label{eq:contraction-prior}
        \Pi_N(\widetilde\calB_N)\geq e^{-AN\delta_N^2} \text{ for some $A>0$.}
    \end{equation}
\end{proposition}

We are now able to state the main theorem:

\begin{theorem}[Posterior contraction]\label{thm:posterior-contraction}  Let $\scrH$ and $\scrR$ be as in \cref{cond:BasePrior} with  $\alpha>\eta + d$. Let $\calG$ satisfy \cref{ass:operatorI}. Let $D_N\sim\IP_{\thetatrue}^N$ be data arising as in \cref{eq:setting:data}, for fixed $\thetatrue\in \scrH$. Let $\Pi_N$ be the sequence of rescaled Gaussian priors as above with $\delta_N$ as in \cref{eq:minimax-rate}, such that $N\delta_N^2\to\infty$ as $N\to\infty$. Let $\Pitilde_N(\cdot | D_N)$ be the surrogate posterior distribution arising as in \cref{sec:setting}. Grant \cref{cond:misspec-error} and \cref{cond:misspec-model}, with constants $C_{\mathrm{noise}}, C_{\mathrm{model}}$ small enough compared to the small ball exponent of the prior (see \cref{eq:prior-tail}). Let $A$ be as in the setting of \cref{prop:prior-mass-bn}, and $M$ be large enough such that $c_6$ from \cref{prop:change-measure-peeling} satisfies $c_6 >A+ s_0^{-2}$.  Then for all $0< b < c_6-A- s_0^{-2}$, we can find $\rho>0$ large enough such that

    \begin{equation} \label{eq:target-proba}
    \IP_{\thetatrue}^N\br{\Pitilde_N(\theta \in \Theta_N(M)\;:\; d_{\mathcal{G}}(\theta, \thetatrue) \leq \rho \delta_N| D_N) \leq 1-e^{-bN\delta_N^2}} \xrightarrow[N\to\infty]{} 0
    \end{equation}

\end{theorem}

\begin{remark}
Our approach covers the case of strongly consistent plug-in estimators as proxy variances and proxy forward map.
In practice, it could also happen, e.g. in the case of noise misspecification, that the variance estimator converges in a weaker sense (in probability), or that in a fully Bayesian approach one prefers adopting a hierarchical approach and putting a prior on $\sigma^2$. Our contraction guarantees remain valid, holding for the conditional posterior of $\theta$ given $\sigma^2$.
\end{remark}\medskip

\begin{remark}[About the contraction rate]
    The resulting rate \cref{eq:minimax-rate} matches standard nonparametric forward convergence rates, which minimax optimality has been established in special cases, such as the Darcy problem (see \cite{Nickl_vdGeer_Wang_2020}). In our approach, we fix the contraction rate $\delta_N$ and set the desired decay of $\deltamap, \deltanoise$ accordingly. In practice one might be limited by the decay of these misspecification rates: note then that the posterior contraction results still hold, with statements on $\delta_N$ replaced with a slower contraction rate, and prior renormalisation also modified suitably -- as long as Conditions \ref{cond:misspec-error} and \ref{cond:misspec-model} connecting contraction and misspecification rates still hold. This yields the following corollary.
\end{remark}

\begin{corollary}
    Consider the setting of \cref{thm:posterior-contraction}, replacing the rate $\delta_N$ from \cref{eq:minimax-rate} everywhere by $\delta_N'=\log N(\deltanoise \vee \deltamap)^{1/2}$. Then for all $0< b < c_6-A- s_0^{-2}$, we can find $\rho>0$ large enough such that

    \begin{equation*} 
    \IP_{\thetatrue}^N\br{\Pitilde_N(\theta \in \Theta_N(M)\;:\; d_{\mathcal{G}}(\theta, \thetatrue) \leq \rho \delta_N'| D_N) \leq 1-e^{-bN\delta_N'^2}} \xrightarrow[N\to\infty]{} 0
    \end{equation*}
\end{corollary}


\subsection{Contraction result for the inverse problem}

While the last theorem proves contraction for the forward problem, we need the following inverse modulus of continuity to provide a corresponding contraction theorem on the parameter-level.

\begin{condition}[Inverse modulus of continuity]\label{cond:inverseproblem}
    
    For any $\delta,M>0$ define
    \begin{equation*}
        \Lambda_\delta\coloneqq\brK{(\theta_1,\theta_2)\in \br{\Theta\cap\scrR(M)}^2:\; d_\calG(\theta_1,\theta_2)\leq \delta}.
    \end{equation*}
    There exist constants $\tau>0$ and  $\Const{\calG,inv}(M)>0$, such that for $\rho$ small enough
    \begin{equation}
        \sup_{(\theta_1,\theta_2)\in\Lambda_\rho}\brK{\norm{\theta_1-\theta_2}_{\IL^2(\calM,W)}} \leq \Const{\calG,inv}(M)\times \delta^\tau.\tag{\textbf{IR1}}
    \end{equation}
\end{condition}

\begin{theorem}[Posterior contraction - inverse problem]\label{thm:posterior-contraction-inverse} 
Grant the assumptions of \cref{thm:posterior-contraction}. 
Assume additionally that \cref{cond:inverseproblem} holds true for some $\tau>0$.
We then have
    \begin{equation} \label{eq:target-proba-inverse}
    \IP_{\thetatrue}^N\br{\Pitilde_N(\theta \in \Theta_N(M): \norm{\theta- \thetatrue}_{\IL^2(\calM,W)} \leq \Const{\calG,inv}(M)(\rho \delta_N)^\tau| D_N) \leq 1-e^{-bN\delta_N^2}} \xrightarrow[N\to\infty]{} 0.
    \end{equation}
    Moreover, denoting by $\IE^{\Pitilde}\brE{\cdot|D_N}$ the surrogate posterior mean, we have 
    \begin{equation}\label{eq:posterior-mean}
        \norm{\IE^{\Pitilde}\brE{\theta|D_N}-\thetatrue}_{\IL^2(\calM,W)} = \calO_{\IP_{\thetatrue}^N}\br{\delta_N^\tau}.
    \end{equation}
\end{theorem}

The proof of \cref{eq:target-proba-inverse} follows easily from \cref{thm:posterior-contraction} and \cref{cond:inverseproblem}. The proof of \cref{eq:posterior-mean} requires more care, in particular the use of the change of measure conditions detailed in \cref{subsec:prem:results}: details can be found in \cref{sec:additional-proofs}.

%% file: Sections/_04_Examples.tex
\section{Examples}\label{sec:examples}
In this section, we apply the theoretical results established in \cref{sec:posterior-contraction} to three illustrative examples of nonlinear and time-dependent PDE-based inverse problems, where the goal is to infer the initial condition $\theta$  of the system at time $t=0$. 
\begin{enumerate}
    \item Noise Misspecification in Reaction-Diffusion: we first consider a time-evolution problem where the primary challenge lies not in the PDE model, but in the observation noise. We focus on a general heteroscedastic setting where sensor noise variances are unknown and must be estimated from auxiliary data.

    \item Model Misspecification via parameter uncertainty in Navier-Stokes: we study the case where other physical parameters governing the PDE are only known approximately.

    \item Model Misspecification via numerical approximation: when the PDE solution is computed approximately - in our case via an Oseen iterative scheme.
\end{enumerate}

\subsection{Example 1: Noise misspecification in the Reaction Diffusion equation}



We begin by addressing the problem of noise misspecification in a dynamical setting. In many practical data assimilation scenarios, observations are gathered from a network of sensors where the precision (noise variance) may vary from one sensor to another and is not known a priori. In this example, we ignore misspecification arising from the PDE map itself and isolate the problem of heteroscedastic misspecified noise. \\

In this setting, it makes sense to consider observations arising from a fixed design setting for the spatial covariate. The random design employed in \cref{eq:RDE:data} facilitates a simpler presentation and is essentially a technical choice: it can be shown to be asymptotically equivalent to other commonly used nonparametric regression models, see for instance \cite{reiss_2008}. In \cite{vollmer_2013}, a condition on the empirical distribution of design points is used.

\paragraph{Construction of the variance proxy.}
Let us then consider a fixed design setting with $L_X \in \IN$ sensors densely distributed across the spatial domain at locations $x_1,\dots,x_{L_X}$, measuring the solution $u_{\thetatrue}$ of \cref{eq:reaction-diffusion} over time. Each sensor $j$ is associated with a measurement error $\calN(0, \sigma_j^2)$ with unknown variance $\sigma_j^2$. To estimate these variances, we collect observations at each spatial location over a small time window, at $\frakt_1,\dots\frakt_{L_T}\in [0, \Delta T]$, $L_T\in\IN$
$$\Upsilon_{ij}=u_{\thetatrue}(\frakt_i,x_j)+\xi_{ij},\quad j=1,\dots,L_X,\quad i=1,\dots,L_T,\quad\xi_{ij}\sim \calN(0, \sigma_j^2).$$
We denote the law of $\Upsilon_{\cdot j}\coloneqq\br{\Upsilon_{ij}}_{i\nset{L_T}}$ by $\widetilde{\IP}_{\thetatrue,j}^{L_T}$. Natural estimators for each of the variances $\sigma_j^2$'s are then the sample variance estimators
\begin{equation}\label{eq:noise-estimate}
    s_j^2 = \frac{1}{L_T-1}\sum_{i=1}^{L_T} (\Upsilon_{ij} - \bar \Upsilon_j)^2, \hspace{1cm} \text{ with } \bar \Upsilon_{j} = \frac{1}{L_T}\sum_{i=1}^{L_T} \Upsilon_{ij}
\end{equation}

These plug-in estimators $s_1^2,...,s_{L_X}^2$ are subsequently used to compute the surrogate log-likelihood $\tilde \ell_N(\theta)$ and the resulting surrogate posterior distribution.\\

\paragraph{Parameter reconstruction.}

The underlying physical process governing the data generation is the Reaction Diffusion equation. Let $\calM = \IT^d$ with $d\leq{3}$. This equation models the time-evolution for a fixed time horizon $T>0$ of the the concentration $u\colon[0,T]\times\torus\to\IR \eqqcolon V$ of a substance from its initial condition $\theta$ by 
\begin{equation}\label{eq:reaction-diffusion}
    \begin{dcases}
        \frac{\partial}{\partial t}u - \Delta u& = f(u)\text{ on $(0,T]\times\torus$}\\
        \hfill u(\cdot,0) &= \theta\text{ on $\torus$}
    \end{dcases}
\end{equation}

where $f \colon\domain \to \IR$ is a nonlinear reaction term modelling potential creation or destruction of the substance (\cite{Temam_1997}). If $f\in C_c^\infty(\IR)$ and $\theta\in H^1(\torus)$, it can be shown (see e.g. \cite{Evans_2010}) that there exists a solution of \cref{eq:reaction-diffusion} that is unique in $C^0([0,T],\IL^2(\torus))$.

We now turn to the inverse problem of recovering $\thetatrue$. 
To that end, we observe another sample of data, which is drawn  independently from $(\Upsilon_{ij})_{i\leq L_T,j\leq L_X}$. 
Precisely, we have
$D_N\coloneqq(Y_i,t_i,X_i)_{i=1}^N\sim \IP_{\thetatrue}^N$ generated by 
\begin{equation}\label{eq:RDE:data}
    Y_i = u_{\thetatrue}(t_i,X_i) + \varepsilon_i,\quad (t_i,X_i)\overset{i.i.d.}{\sim}\operatorname{Unif}\br{[0,T]\times \{x_1,...,x_{L_X}\}},\quad \varepsilon_i| x_i \overset{ind.}{\sim}\operatorname{N}\br{0,\sigma_i^2},\quad i = 1,\dots,N
\end{equation}
with $\thetatrue\in H^1(\torus)$ and unknown variances $\sigma_1^2,\dots,\sigma_N^2$ satisfying \cref{cond:error:variances}. With the noise estimators and forward model defined, we can now state the contraction result for this setup. The following theorem ensures that, despite using estimated variances, the surrogate posterior contracts around the ground truth.\\

\begin{theorem}\label{thm:example-noise} 
 Let $\Pi'$ be a Gaussian process base prior satisfying \cref{cond:BasePrior} with $\Theta = \scrR = H^\beta(\torus)$, $\beta>2+d$, and RKHS $\scrH \hookrightarrow H^\alpha(\torus)$ with $\alpha > \beta+\frac{d}{2}$. Let $\Pi_N$ be the corresponding rescaled prior from \cref{eq:rescaledprior} with 
\begin{equation*}
    \delta_N = N^{-\frac{\alpha}{2\alpha +d}}.
\end{equation*}
Let $D_N\sim\IP_{\thetatrue}^N$ as in \cref{eq:RDE:data} with $\thetatrue\in \scrH$. Consider the surrogate posterior distribution $\Pitilde(\cdot | D_N)$ arising from that choice of prior and the surrogate log-likelihood $\loglikmodel(\theta)$ computed with the noise variance estimators $s_1^2,\dots,s_{L_X}^2$ from \cref{eq:noise-estimate} over a small enough time window $\Delta T$ such that the bias term $b_T$ satisfies $b_T^2 + L_T^{-1/2} \le \frac{1}{\log N}\delta_N^2$.
Then, for $N\to\infty$ and thus sufficiently many $L_T = L_T(N)$ past observations, the surrogate posterior
$\Pitilde_N(\cdot | D_N)$ contracts around the ground truth $\thetatrue$ at rate $\delta_N$, i.e. there exist $m,m'>0$ sufficiently large, such that
\begin{equation*}
    \Pitilde_N\br{\theta\in H^\beta(\torus):\;\norm{u_\theta-u_{\thetatrue}}_{\IL^2([0,T]\times\torus)} \leq m\delta_N\mid D_N} = 1 - o_{\IP_{\thetatrue}^N}(1).
\end{equation*}
Moreover, we have 
\begin{equation*}
    \Pitilde_N\br{\theta\in H^\beta(\torus):\;\norm{\theta-\thetatrue}_{\IL^2(\torus)} \leq m'\delta_N^{\frac{\beta}{\beta+1}}\mid D_N} = 1 - o_{\IP_{\thetatrue}^N}(1)
\end{equation*}
as well as 
\begin{equation*}
     \norm{\IE^{\Pitilde}\brE{\theta|D_N}-\thetatrue}_{\IL^2(\torus)} = \calO_{\IP_{\thetatrue}^N}\br{\delta_N^{\frac{\beta}{\beta+1}}}.
\end{equation*}
\end{theorem}

\begin{proof}[Proof of \cref{thm:example-noise}]
 We choose the time window $\Delta T$ used to build estimators $s_1^2,...,s_{L_X}^2$ sufficiently small so that
$$ \max_{1\le i\le L_T}|u_{\theta_0}(t_i,x_j)-u_{\theta_0}(0,x_j)| \le b_T, \qquad b_T \to 0,$$
uniformly in $1\leq j\leq L_X$ as $L_T\to\infty$.  Then the deterministic bias term in $s_j^2$, arising from the variations in time of $u_{\thetatrue}(t,x_j)$, is bounded by $b_T^2$. By the strong law of large numbers,
\begin{equation}\label{eq:slln}
    s_j^2 = \sigma_j^2 + b_T^2 + O_p(L_T^{-1/2})  
\end{equation}

under $\widetilde{\IP}_{\thetatrue,j}^{L_T}$ for every $1\leq j\leq L_X$.
Hence, \cref{cond:misspec-error} \ref{item:NMI} bounding the $s_j's$ away from $0$ and $\infty$ follows from \cref{cond:error:variances} with $\widetilde{\IP}_{\thetatrue,j}^{L_T}$-probability arbitrarily close to one for $L_T$ large enough.


\cref{item:NMII} is also satisfied with $\deltanoise = \max_{1\leq i \leq L_X}|1-\sigma_i^2/s_i^2|\leq \delta_N^2/\log N$, by \cref{eq:slln} and by the theorem assumption. In particular, \cref{item:NMII.II} follows from \cref{lemma:prop5new} with $a = 2$ and the embedding $H^2(\torus)\hookrightarrow C^0(\torus)$ yielding \cref{cond:fm:bound} with $\gamma_B = 1$. Note that $\calG\colon \theta\mapsto u_\theta$ satisfies \cref{cond:fm:Liptwo} and \cref{cond:fm:Lipinf} with $\kappa = 0$ and $\eta = 2$, which is shown in section 3.1.3 in \cite{Nickl_2025_BvMRDE}. Further, \cref{cond:inverseproblem} is shown in \cite{Nickl_2025_BvMRDE} (see equation (45)) with $\tau = \frac{\beta}{\beta+1}$. 
The claim thus follows immediately from the main theorems \cref{thm:posterior-contraction} and \cref{thm:posterior-contraction-inverse}.
\end{proof}

\subsection{Example 2: Model misspecification in the Navier-Stokes equation}

We now turn to the case of misspecification of the PDE forward map, ignoring possible noise misspecification. In practical fluid dynamics and data assimilation, the governing physical laws are often well-understood, but the specific physical parameters defining the system may only be known approximately: this is the example we address here.\\

Let $\calM = \IT^2$. The 2D Navier-Stokes equation describes the evolution of the velocity $u:[0,T]\times\IT^2\to V \coloneqq \IR^2$ of (incompressible) fluids from an initial velocity $\theta$ at time $t=0$ for a fixed time horizon $T>0$. Given a viscosity $\nu>0$, a scalar pressure $p : [0,T] \times \IT^d \to \mathbb{R}$, and some time-independent external forcing $f\colon\IT^2\to\IR^2$, $u =u_\theta^{\nu, f}$ then solves
\begin{equation}\label{eq:navier-stokes}
    \begin{dcases}
\begin{aligned}
\frac{\partial u}{\partial t}
- \nu \Delta u
+ (u \cdot \nabla) u
&= f - \nabla p
&& \text{on } (0,T) \times \mathbb{T}^2 , \\
u(0) &= \theta
&& \text{on } \mathbb{T}^2 , \\
\nabla \cdot u &= 0
&& \text{on } [0,T] \times \mathbb{T}^2  .
\end{aligned}
    \end{dcases}
\end{equation}

Among the physical parameters appearing in $\cref{eq:navier-stokes}$, the initial condition $\theta$ is the one that we want to infer; however the other parameters $\nu$ and $f$ might only be known approximately via some estimates $\tilde \nu$ and $\tilde f$.\\

As common in the literature, we consider the \textit{projected} equation \cref{eq:navier-stokes} by applying the Leray operator \cref{eq:leray-operator} on it. This leads to an equivalent formulation in functional form where the velocity field, as a map $u_\theta^{\nu,f}\colon[0,T]\to \Hddiamond$, is the solution to 
\begin{equation}\label{eq:projectednavierstokes}
    \frac{\rmd}{\rmd t}u + \nu Au + B[u,u] = f, \quad u(0) = \theta
\end{equation}
with  $A\coloneqq -P\Delta$ and $B[u,v]\coloneqq P[(u\cdot\nabla)v]$. It is well-known that for any $f\in\Hddiamond$, $\nu,T>0$ and any $\theta\in \Hddiamond^1$, there exists a solution of \cref{eq:projectednavierstokes} that is unique in $C^0\br{[0,T],\Hddiamond^1}\cap\IL^2\br{[0,T],\Hddiamond^2}$, see e.g. \cite{Robinson_2001}.
 Our observation scheme hence consists of data $D_N\coloneqq(Y_i,t_i,X_i)_{i=1}^N\sim \IP_{\thetatrue}^N$ generated by 
\begin{equation}\label{eq:NSE:data}
    Y_i = u_{\thetatrue}^{\nu, f}(t_i,X_i) + \varepsilon_i,\quad (t_i,X_i)\overset{i.i.d.}{\sim}\operatorname{Unif}\br{[0,T]\times\IT^2},\quad \varepsilon_i\overset{ind.}{\sim}\operatorname{N}\br{0,\sigma_i^2\operatorname{Id}_{\IR^2}},\quad i=1,\dots,N
\end{equation}
with $\thetatrue\in H^1(\torus)$ and variances $\sigma_1^2,\dots,\sigma_N^2$ satisfying \cref{cond:error:variances}. 

\begin{theorem}\label{thm:example-model-rde}

Let $\Pi'$ be a Gaussian process base prior satisfying \cref{cond:BasePrior} with $\Theta = \scrR = \Hddiamond^\beta$ with $\beta>4$ and RKHS $\scrH \hookrightarrow \Hddiamond^\alpha$ with $\alpha > \beta+1$. Let $\Pi_N$ be the corresponding rescaled prior from \cref{eq:rescaledprior} with 
\begin{equation*}
    \delta_N = N^{-\frac{\alpha}{2\alpha +2}}.
\end{equation*}
Let $D_N\sim\IP_{\thetatrue}^N$ as in \cref{eq:NSE:data} with $\thetatrue\in \scrH$, $f\in\Hddiamond^1$ and $\nu>0$. Consider the surrogate posterior distribution $\Pitilde(\cdot | D_N)$ arising from that choice of prior and the surrogate log-likelihood $\loglikmodel(\theta)$ computed with surrogate forward operator $\calGtilde(\theta) = u_{\theta}^{\tilde\nu,\tilde f}$, where the parameter approximations $\tilde \nu >0$ and $\tilde f\in\Hddiamond^1$ satisfy
\begin{equation}\label{eq:nse-proxys}
    |\tilde \nu -\nu| \lesssim \frac{1}{\log N}\delta_N^2 \text{ and } \| f-\tilde f\|_{L^2([0,T], \dot H^1)} \lesssim \frac{1}{\log N}\delta_N^2.
\end{equation}
Then 
the surrogate posterior
$\Pitilde_N(\cdot | D_N)$ contracts around the ground truth $\thetatrue$ at rate $\delta_N$, i.e. there exist $m,m'>0$ sufficiently large, such that
\begin{equation*}
    \Pitilde_N\br{\theta\in\Hddiamond^\beta:\;\norm{u_\theta-u_{\thetatrue}}_{\IL^2([0,T]\times\IT^2,\IR^2)} \leq m\delta_N\mid D_N} = 1 - o_{\IP_{\thetatrue}^N}(1).
\end{equation*}
Moreover, we have that 
\begin{equation*}
    \Pitilde_N\br{\theta\in\Hddiamond^\beta:\;\norm{\theta-\thetatrue}_{\IL^2(\IT^2,\IR^2)} \leq m'\delta_N^{\frac{\beta}{\beta+1}}\mid D_N} = 1 - o_{\IP_{\thetatrue}^N}(1)
\end{equation*}
as well as
\begin{equation*}
     \norm{\IE^{\Pitilde}\brE{\theta|D_N}-\thetatrue}_{\IL^2(\IT^2,\IR^2)} = \calO_{\IP_{\thetatrue}^N}\br{\delta_N^{\frac{\beta}{\beta+1}}}.
\end{equation*}
\end{theorem}

\begin{proof}[Proof of \cref{thm:example-model-rde}]
    Firstly, note that the associated forward map $\calG\colon\theta\to u_\theta^{\nu,f}$ of \cref{eq:projectednavierstokes} satisfies the conditions imposed in \cref{ass:operatorI} and \cref{ass:operatorII} with $\kappa = 0$, $\eta = 2$, $\gamma_B = 2$ and $\tau=\frac{\beta}{\beta+1}$ as derived in \cite{NicklTiti_2024} and \cite{Konen_Nickl_2025}. Analogously, $\calGtilde\colon\theta\mapsto u_\theta^{\tilde\nu,\tilde f}$ satisfies \cref{cond:MM:I} (with a change of constants). \cref{cond:MM:III} then follows from  the stability of $u_\theta^{\nu,f}$ with respect to the parameters $\nu$ and $f$ as derived in \cref{lem:NS:Misspec:regularity} as well as the Sobolev embedding $\Hdot^2(\IT^2)\hookrightarrow C^0(\IT^2)$. Thus the claim follows immediately from an application of \cref{thm:posterior-contraction} and \cref{thm:posterior-contraction-inverse}.
\end{proof}

\subsection{Example 3: Model misspecification from numerical approximation}





We now address another source of model misspecification, via numerical approximation. Looking at the 2D Navier-Stokes equation \cref{eq:navier-stokes}, computing a corresponding solution $\calG(\theta) = u_\theta^{\nu,f}$ is generally challenging, with the main difficulty coming from the non-linear convection term $(u \cdot \nabla)u$. In practice, linearization from an iterative process is often used. Starting with an appropriate initializer $u^0$, the $l^{th}$ (projected) iteration for $l\in\IN_0$ is defined via
\begin{equation}\label{eq:oseen-iteration}
    \frac{\rmd}{\rmd t}u^l + \nu Au^l + B[u^{l-1},u^l] = f, \quad u^l(0) = \theta.
\end{equation}
These are called \emph{Oseen equations} and constitute a good approximation of Navier-Stokes in viscous flow settings under some smallness condition on the Reynolds number ($\nu \gg1, Re\ll1$)(see \cite{Girault_Raviart_1986, Batchelor_1999} for details). These fixed point (Picard) iterations are known to converge linearly, so that by taking the number of iterations $L\in\IN$ sufficiently high one can then show that for $\calGtilde(\theta)=u^L_\theta$ the surrogate operator is \textit{good enough} for the contraction of the surrogate posterior as the following results show.\\

\begin{proposition} \label{prop:oseen-satisfies-condition}
Let $\nu>0$, $f\in \IL^2([0,T],\Hddiamond^1)$, $\theta\in\Hddiamond^1$, and choose an initilizer $u^0\in\IL^2\br{[0,T],\Hddiamond^2}$.
Beginning with $u^0$, for all $l\in\IN$, the $l^{th}$-iteration step of \cref{eq:oseen-iteration} has a solution 
\begin{equation*}
    u^l\in C^{0}([0,T],\Hddiamond^1)\cap\IL^2([0,T],\Hddiamond^{2}),\quad\frac{\rmd u^l}{\rmd t}\in \IL^2([0,T],\Hddiamond^{0}).
\end{equation*}
Now let $L=L_N\in\IN$ be chosen sufficiently large, such that 
\begin{equation*}
        \forall r>0:\; \sup_{\theta\in\Hddiamond^2(r)}\sup_{t\in[0,T]}\norm{u_\theta^{L}(t)-u_\theta^{L-1}(t)}_{\Hdot^2(\IT^2)} \leq \Const{\mathrm{model}}(r)\times\frac{1}{\log N}\delta_N^2. 
\end{equation*}
Assume further, that this last iterate $u_\theta^L$ satisfies
\begin{equation*}
    \forall\theta\in\Hddiamond^2:\;\sup_{t\in[0,T]}\norm{u_\theta^{L}(t)}_{\Hdot^2(\IT^2)} \leq \Const{Oseen,B}\times\br{1+\norm{\theta}_{\Hdot^2}^2}.
\end{equation*}
Defining than the surrogate forward map $\calGtilde\colon \Hddiamond^2\ni\theta\mapsto u_\theta^{L}\in \Hddiamond^2$ it satisfies \cref{cond:misspec-model}.
\end{proposition}
The proof of \cref{prop:oseen-satisfies-condition} follows standard regularity estimates as presented in \cite{Konen_Nickl_2025} and can be found in \cref{sec:app:PDE}. 

\begin{theorem}\label{thm:example-model-oseen}

Let $\Pi'$ be a Gaussian process base prior satisfying \cref{cond:BasePrior} with $\Theta = \scrR = \Hddiamond^\beta$ with $\beta>4$ and RKHS $\scrH \hookrightarrow \Hddiamond^\alpha$ with $\alpha > \beta+1$. Let $\Pi_N$ be the corresponding rescaled prior from \cref{eq:rescaledprior} with 
\begin{equation*}
    \delta_N = N^{-\frac{\alpha}{2\alpha +2}}.
\end{equation*}
Let $D_N\sim\IP_{\thetatrue}^N$ as in \cref{eq:NSE:data} with $\thetatrue\in \scrH$, $f\in\Hddiamond^1$ and $\nu>0$. Consider the surrogate posterior distribution $\Pitilde(\cdot | D_N)$ arising from that choice of prior and the surrogate log-likelihood $\loglikmodel(\theta)$ computed with surrogate forward operator $\calGtilde(\theta) = u_\theta^{L}$ as described in \cref{prop:oseen-satisfies-condition}.
Then 
the surrogate posterior
$\Pitilde_N(\cdot | D_N)$ contracts around the ground truth $\thetatrue$ at rate $\delta_N$, i.e. there exist $m,m'>0$ sufficiently large, such that
\begin{equation*}
    \Pitilde_N\br{\theta\in\Hddiamond^\beta:\;\norm{u_\theta-u_{\thetatrue}}_{\IL^2([0,T]\times\IT^2,\IR^2)} \leq m\delta_N\mid D_N} = 1 - o_{\IP_{\thetatrue}^N}(1).
\end{equation*}
Moreover, we have that 
\begin{equation*}
    \Pitilde_N\br{\theta\in\Hddiamond^\beta:\;\norm{\theta-\thetatrue}_{\IL^2(\IT^2,\IR^2)} \leq m'\delta_N^{\frac{\beta}{\beta+1}}\mid D_N} = 1 - o_{\IP_{\thetatrue}^N}(1)
\end{equation*}
as well as
\begin{equation*}
     \norm{\IE^{\Pitilde}\brE{\theta|D_N}-\thetatrue}_{\IL^2(\IT^2,\IR^2)} = \calO_{\IP_{\thetatrue}^N}\br{\delta_N^{\frac{\beta}{\beta+1}}}.
\end{equation*}
\end{theorem}

The proof of \cref{thm:example-model-oseen} follows the arguments of \cref{thm:example-model-rde} using \cref{prop:oseen-satisfies-condition} and applying \cref{thm:posterior-contraction} as well as \cref{thm:posterior-contraction-inverse}, and is thus omitted.

\begin{remark}
   \cref{prop:oseen-satisfies-condition} provides a guideline for when the iteration can be considered sufficiently converged: when two successive iterates are sufficiently close in the $\dot H^2$-norm and when $u^l$ exhibits the qualitative analytic properties of the exact solution $u_\theta$ (relative to \cref{cond:fm:bound}).
\end{remark}
   It should be noted that in this setting, recovering the initial condition from a linearized Navier-Stokes becomes an approximately simpler \textit{linear} inverse problem, but our analysis still applies.

%% file: Sections/_05_Additional_proofs_General.tex
\section{\texorpdfstring
  {Proof of \cref{sec:posterior-contraction}}
  {Proof of Theorem X}
}\label{sec:additional-proofs}


In this section we prove the results of \cref{sec:posterior-contraction}. We will use repeatedly the following inequality, which can easily be derived by Cauchy-Schwarz.
\begin{equation}\label{eq:sum-squares}
    \forall j\nset{N},\;(a_j)_{j=1}^J\subseteq \IR:\;\sum_{j\nset{J}}a_j\leq \sqrt{J} \times \sqrt{\sum_{j\nset{J}} a_j^2}
\end{equation}
Further, in what follows we will write shorthand $\tilde w_\theta := \calG(\theta)-\calGtilde(\theta)$.

\subsection{Small ball computations}

\begin{proof}[Proof of \cref{prop:information-ineq}] 
    First, under the ground truth probability $\Truedensityo$, 

$$\forall i=1,\dots,N:\;Y_i = \calG(\thetatrue)(Z_i)+\varepsilon_i = \calGtilde(\thetatrue) (Z_i)+\tilde \varepsilon_i ~~~~\text{where } \tilde \varepsilon_i = \tilde w_{\thetatrue}(Z_i) + \varepsilon_i,$$
so
\begin{align}\label{eq:pseudo-loglik-comput}
    \loglikmodel(\theta)-\loglikmodel(\thetatrue) &= -\frac{1}{2}\sum_{i=1}^N\frac{1}{s_i^2}\Vnorm{\calGtilde(\thetatrue)(Z_i)-\calGtilde(\theta)(Z_i)+\tilde\varepsilon_i}^2 +\frac{1}{2}\sum_{i=1}^N \frac{1}{s_i^2}\Vnorm{\tilde \varepsilon_i}^2 \notag\\
    &=-\frac{1}{2}\sum_{i=1}^N\frac{1}{s_i^2}\Vnorm{\calGtilde(\theta)(Z_i)-\calGtilde(\thetatrue)(Z_i)}^2 -\sum_{i=1}^N \frac{1}{s_i^2}\Vsprod{\tilde w_{\thetatrue}(Z_i)}{\calGtilde(\thetatrue)(Z_i)-\calGtilde(\theta)(Z_i)}\notag\\
    &\hspace{1cm}-\sum_{i=1}^N \frac{1}{s_i^2}\Vsprod{\varepsilon_i}{\calGtilde(\thetatrue)(Z_i)-\calGtilde(\theta)(Z_i)}.
\end{align}

\begin{itemize}
    \item \textbf{Proof of i):} 
        Taking the expectation under $\Truedensityo$, the last term cancels and we obtain
        \begin{align*}
            -\IE_{\thetatrue}^N[\loglikmodel(\theta)-\loglikmodel(\thetatrue)] = \frac{N}{2} \bar{s}_N^{-2}\|\calGtilde(\theta) - \calGtilde(\thetatrue) \|_{\IL_\zeta^2(\calZ,V)}^2 + \sum_{i=1}^N \frac{1}{s_i^2} \IE\brE{ \Vsprod{\tilde w_{\thetatrue} (Z_i)}{\calGtilde(\thetatrue)(Z_i)-\calGtilde(\theta)(Z_i)}}.
        \end{align*}
        Applying \cref{cond:MM:III} on $\tilde w_{\thetatrue}$ and by Cauchy-Schwarz, we further obtain
        \begin{align*}
        \IE\brE{\Vsprod{\tilde w_{\thetatrue} (Z_i)}{\calGtilde(\thetatrue)(Z_i)-\calGtilde(\theta)(Z_i)}} 
        &\leq c(\thetatrue) \deltamap\times \| \calGtilde(\thetatrue)-\calGtilde(\theta)\|_{\IL_\zeta^2(\calZ,V)}
        \end{align*}
        
        
        So that in the end
        \begin{align*}
        -\IE_{\thetatrue}^N[\loglikmodel(\theta)-\loglikmodel(\thetatrue)] &\leq \frac{1}{2}N s_0^{-2}\times d_{\calGtilde}(\theta,\thetatrue)^2 + C(\thetatrue, s_0^2, s_\infty^2) N \deltamap \times d_{\calGtilde}(\theta,\thetatrue).
        \end{align*}

        \item \textbf{Proof of ii):}  We now compute $T_i\coloneqq\IE_{\thetatrue}^{(i)}\left[\left(\log\frac{q_{\theta}^{(i)}}{q_{\thetatrue}^{(i)}}- \IE_{\thetatrue}^{(i)} \log\frac{q_{\theta}^{(i)}}{q_{\thetatrue}^{(i)}}\right)^2\right]$ for $i = 1,\dots,N$. By similar computations as above, under $\Truedensityo$, we have 
\begin{align*}
   \log\frac{q_{\theta}^{(i)}}{q_{\thetatrue}^{(i)}}- \IE_{\thetatrue}^N \log\frac{q_{\theta}^{(i)}}{q_{\thetatrue}^{(i)}} 
   &= -\frac{1}{2s_i^2}\Vnorm{\diffG}^2 - \frac{1}{s_i^2}\Vsprod{\tilde\varepsilon_i}{\diffG} \\
    &\quad + \frac{1}{2s_i^2}\diffnormG^2 + \frac{1}{s_i^2}\IE_{\thetatrue}^{(i)}\Vsprod{\tilde w_{\thetatrue}(Z_i)}{\diffG}
\end{align*}

Using \cref{eq:sum-squares}, we upper bound $T_i$ by the sum of squares: then by Cauchy-Schwarz
\begin{align*}
    T_i
    &\leq \frac{4}{s_i^4}\IE_{\thetatrue}^{(i)}\Bigg[\frac{1}{4}\Vnorm{\diffG}^4 + \Vnorm{\tilde \varepsilon_i}^2\Vnorm{\diffG}^2  \\
    &~~~~~~~~~~~~~~~~~~~~~~~~~+\frac{1}{4}\diffnormG^4 + \IE_{\thetatrue}^{(i)}\left[\Vsprod{\tilde w_{\thetatrue} (Z_i)}{\calGtilde(\thetatrue)(Z_i)-\calGtilde(\theta)(Z_i)}\right]^2\Bigg]\\
    &= \frac{1}{s_i^4}\IE_{\thetatrue}^{(i)}\left[\Vnorm{\diffG}^4\right]+\frac{4}{s_i^4}\IE_{\thetatrue}^{(i)}\left[\Vnorm{\tilde w_{\thetatrue}(Z_i)}^2 \Vnorm{\diffG}^2\right] \\
    &+ \frac{4\sigma_i^2}{s_i^4}\diffnormG^2 + \frac{1}{s_i^4}\diffnormG^4 + 4 \left(\IE_{\thetatrue}^{(i)}\brE{\Vsprod{\tilde w_{\thetatrue} (Z_i)}{\calGtilde(\thetatrue)(Z_i)-\calGtilde(\theta)(Z_i)}}\right)^2
\end{align*}

where we used that $\tilde \varepsilon_i = \varepsilon_i + \tilde w_{\thetatrue}(Z_i)$ and $\varepsilon_i \perp Z_i, \IE\brE{ \varepsilon_i} =0$ and $\IE\brE{\varepsilon_i^2}=\sigma_i^2$ to go from the first to the second line.
Then, by Jensen, Cauchy-Schwarz and the uniform bound on $\tilde w_{\thetatrue}$ by \cref{cond:misspec-model}

$$\left(\IE\brE{\Vsprod{\tilde w_{\thetatrue} (Z_i)}{\calGtilde(\thetatrue)(Z_i)-\calGtilde(\theta)(Z_i)}}\right)^2 
\lesssim \deltamap^2\times \diffnormG^2$$

and finally using the uniform bounds on $\calGtilde$ on the set $\tilde\calB_N$
\begin{align*}
    \forall i=1,\dots,N:\; T_i&\leq \frac{1}{s_i^4}\diffnormG^2 \left[4U^2+ 4\deltamap^2 + 4\sigma_i^2 + 4U^2 + 4\deltamap^2 \right]\\
    &\leq 4 \diffnormG^2 \left(4\frac{U^2 + \sigma_i^2}{s_i^4} + 2\frac{\deltamap^2}{s_i^4}\right)
\end{align*}

As $\deltamap \to 0$ as $N\to \infty$, for $N$ large enough there exists a constant $\const_2 = \const_2(U, s_0^2, \sigma_\infty^2)>0$  such that
\begin{equation*}
    \IE_{\thetatrue}^{(i)}\left[\left(\log\frac{q_{\theta}^{(i)}}{q_{\thetatrue}^{(i)}}- \IE_{\thetatrue}^{(i)} \log\frac{q_{\theta}^{(i)}}{q_{\thetatrue}^{(i)}}\right)^2\right] \leq \const_2 \diffnormG^2,
\end{equation*}
which shows the claim.
\end{itemize}

\end{proof}

Following standard literature, if the small ball condition is satisfied then \cref{lemma:contraction-auxiliary} holds.

 \begin{proof}[Proof of \cref{lemma:contraction-auxiliary}] By Jensen's inequality,
        \begin{align*}
            \log \int_{B_N} \frac{q_\theta^N}{q_{\thetatrue}^N}(D_N)d\nu(\theta) \ge \int_{B_N} \log \frac{q_\theta^N}{q_{\thetatrue}^N}(D_N)d\nu(\theta)
        \end{align*}
        
So using \cref{prop:information-ineq}, and the fact that on the support $B_N$ of $\nu$ $\norm{\calGtilde(\theta) - \calGtilde(\thetatrue)}_{\IL^2_\zeta(\calZ,V)}^2\leq \delta_N^2$, the probability in question is bounded by

    \begin{align}
       \IP_{\thetatrue}^N &\left(\int_{B_N} \log \frac{q_{\theta}^N}{q_{\thetatrue}^N}(D_N)d\nu(\theta) \leq -  KN\delta_N^2  \right) \nonumber\\
        &= \IP_{\thetatrue}^N\Big(\int_{B_N} \big[\log \frac{q_{\theta}^N}{q_{\thetatrue}^N}(D_N)-\IE_{\thetatrue}^N\log \frac{q_{\theta}^N}{q_{\thetatrue}^N}(D_N)\big]d\nu(\theta) \leq  -KN\delta_N^2 - \int_{B_N} \IE_{\thetatrue}^N\log \frac{q_{\theta}^N}{q_{\thetatrue}^N}(D_N) d\nu(\theta) \Big) \nonumber\\
     &\le \IP_{\thetatrue}^N\left(\int_{B_N} \big[\log \frac{q_{\theta}^N}{q_{\thetatrue}^N}(D_N)-\IE_{\thetatrue}^N\log \frac{q_{\theta}^N}{q_{\thetatrue}^N}(D_N)\big]d\nu(\theta)   \leq  -(K-s_0^{-2}/2) N\delta_N^2 +c_1  N\deltamap\delta_N \right)\nonumber\\
     &= \IP_{\thetatrue}^N\left(\mathbb{E}_\nu \log\frac{q_\theta^N}{q_{\thetatrue}^N} - \mathbb{E}_\nu \IE_{\thetatrue}^N\log \frac{q_{\theta}^N}{q_{\thetatrue}^N}(D_N) \leq  -(K-s_0^{-2}/2) N\delta_N^2 +c_1 N\deltamap\delta_N \right)\nonumber\\
    &= \IP_{\thetatrue}^N\left(\sum\limits_{i=1}^N \left(\mathbb{E}_\nu \log\frac{q_{\theta}^{(i)}}{q_{\thetatrue}^{(i)}}(Y_i, Z_i) - \mathbb{E}_\nu \IE\log \frac{q_{\theta}^{(i)}}{q_{\thetatrue}^{(i)}}(Y_i, Z_i)\right) \leq  -(K-s_0^{-2}/2) N\delta_N^2 + c_1 N\deltamap\delta_N \right)\nonumber\\
    &\overset{\text{Fubini}}{=}\IP_{\thetatrue}^N\left(\sum\limits_{i=1}^N \left(\mathbb{E}_\nu \log\frac{q_{\theta}^{(i)}}{q_{\thetatrue}^{(i)}}(Y_i, Z_i) -  \IE\mathbb{E}_\nu \log \frac{q_{\theta}^{(i)}}{q_{\thetatrue}^{(i)}}(Y_i, Z_i)\right) \leq  -(K-s_0^{-2}/2) N\delta_N^2 + c_1 N\deltamap\delta_N \right) \nonumber\\
    &\overset{\text{Chebyshev}}{\leq} \IP_{\thetatrue}^N\left( \left| \sum\limits_{i=1}^N W_i \right| \geq  K'N \delta_N^2 \right)
    \leq \frac{\operatorname{Var}\left(\sum\limits_{i=1}^N W_i \right)}{K'^2 N^2 \delta_N^4} 
    = \frac{\sum\limits_{i=1}^N \operatorname{Var}\left(W_i \right)}{K'^2 N^2\delta_N^4} \label{eq:var-lemma}
\end{align}
by independence of the $W_i$, for some $K'>0$ since $N\deltamap \delta_N=o(N\delta_N^2)$ by \cref{cond:misspec-model}, having defined for $i\nset{N}$ the centered variables $W_i = \mathbb{E}_\nu\left[\log \frac{q_{\theta}^{(i)}}{q_{\thetatrue}^{(i)}}(Y_i, Z_i)\right] - \IE_{\thetatrue} ^N \mathbb{E}_\nu\left[\log \frac{q_{\theta}^{(i)}}{q_{\thetatrue}^{(i)}}(Y_i, Z_i)\right]$.\\

Now, for every $i=1,...,N$
\begin{align*}
\operatorname{Var}(W_i) = \IE_{\thetatrue}^N[W_i^2] 
    &= \IE_{\thetatrue}^N \left(\mathbb{E}_\nu\left[\log \frac{q_{\theta}^{(i)}}{q_{\thetatrue}^{(i)}}(Y_i, Z_i) - \IE_{\thetatrue}^N\log \frac{q_{\theta}^{(i)}}{q_{\thetatrue}^{(i)}}(Y_i, Z_i)\right]\right)^2 \text{ by Fubini}\\
    &\leq \IE_{\thetatrue}^N\mathbb{E}_\nu\left[\left(\log \frac{q_{\theta}^{(i)}}{q_{\thetatrue}^{(i)}}(Y_i, Z_i) - \IE_{\thetatrue}^N\log \frac{q_{\theta}^{(i)}}{q_{\thetatrue}^{(i)}}(Y_i, Z_i)\right)^2\right] \text{ by Jensen on $\mathbb{E}_\nu$}\\
    &= \mathbb{E}_\nu \IE_{\thetatrue}^N\left[\left(\log \frac{q_{\theta}^{(i)}}{q_{\thetatrue}^{(i)}}(Y_i, Z_i) - \IE_{\thetatrue}^N\log \frac{q_{\theta}^{(i)}}{q_{\thetatrue}^{(i)}}(Y_i, Z_i)\right)^2\right] \text{ by Fubini.}
\end{align*}
\cref{prop:information-ineq} (ii) and the assumption on the support of $\nu$ yields
$$\operatorname{Var}(W_i)\leq c_2 ~\mathbb{E}_\nu[d_{\calGtilde}(\theta,\thetatrue)^2]\leq c_2\delta_N^2$$

and coming back to \cref{eq:var-lemma}, this leads to 
\begin{align*}
            \IP_{\thetatrue}^N\left(\int_{B_N} e^{\tilde\ell_N(\theta)-\tilde\ell_N(\thetatrue)}d\nu(\theta) \leq e^{-KN\delta_N^2}\right) 
            \leq\frac{c_2 N\delta_N^2}{K'^2N^2\delta_N^4}=O\left(\frac{1}{N\delta_N^2}\right)\xrightarrow[N\to \infty]{} 0 .
\end{align*}

\end{proof}

\subsection{Proofs for the change of measure}

\begin{proof}[Proof of \cref{prop:change-measure}]
Recall the notation $\tilde w_\theta = \calG(\theta)-\calGtilde(\theta)$, for $\theta \in \Theta$. Under $\IP_{\theta}^N$, the generated data satisfies, $\forall 1\leq i \leq N:$
\begin{align*}
    Y_i = \calG(\theta)(Z_i)+\varepsilon_i = \calGtilde(\theta) (Z_i)+\tilde \varepsilon_i, \text{ in terms of the surrogate model}
\end{align*}
where $\tilde  \varepsilon_i\mid Z_i = \tilde w_\theta(Z_i) +\varepsilon_i\mid Z_i \sim N(\tilde w_\theta(Z_i),\sigma_i^2)$ are mutually independent. Following similar computations as in the proof of \cref{prop:information-ineq}, this time under $\IP_{\theta}^N$, we find:
\begin{align*}
    \ell_N(\theta)-\ell_N(\thetatrue) = \frac{1}{2}\sum_{i=1}^N \frac{1}{\sigma_i^2} \Vnorm{\calG(\theta)(Z_i)-\calG(\thetatrue)(Z_i)}^2 + \sum_{i=1}^N \frac{1}{\sigma_i^2} \Vsprod{\calG(\theta)(Z_i)-\calG(\thetatrue)(Z_i)}{\varepsilon_i}
\end{align*}
and
\begin{align*}
    \tilde \ell_N(\theta)- \tilde \ell_N(\thetatrue) 
    &= \frac{1}{2}\sum_{i=1}^N \frac{1}{s_i^2}\Vnorm{\calGtilde(\theta)(Z_i)-\calGtilde(\thetatrue)(Z_i)}^2 + \sum_{i=1}^N \frac{1}{s_i^2}\Vsprod{\calGtilde(\theta)(Z_i)-\calGtilde(\thetatrue)(Z_i)}{\tilde \varepsilon_i}\\
    &= \frac{1}{2}\sum_{i=1}^N\frac{1}{s_i^2} \Vnorm{\calGtilde(\theta)(Z_i)-\calGtilde(\thetatrue)(Z_i)}^2 + \sum_{i=1}^N \frac{1}{s_i^2}\Vsprod{\calGtilde(\theta)(Z_i)-\calGtilde(\thetatrue)(Z_i)}{\varepsilon_i } \\
    &~~~~~~~+ \sum_{i=1}^N \frac{1}{s_i^2} \Vsprod{\calGtilde(\theta)(Z_i)-\calGtilde(\thetatrue)(Z_i)}{ \tilde w_\theta(Z_i)}.
\end{align*}


This yields, by independence of observations
\begin{align} \label{eq:target-com}
    \IE_{\theta}^N\left[\left(\frac{q_\theta^N}{q_{\thetatrue}^N}\frac{p_{\thetatrue}^N}{p_{\theta}^N}\right)^b\right]
    = \IE_{\theta}^N\brE{\exp\left(b(\tilde \ell_N(\theta)-\tilde\ell_N(\thetatrue))-b(\ell_N(\theta)-\ell_N(\thetatrue))\right)}
    =\prod_{i=1}^N\IE_{\theta}^N\left[\exp\left(b W_i\right)\right],
\end{align}

with 
\begin{equation*}
    W_i= \underbrace{\frac{1}{2} \left[\frac{1}{s_i^2}\Vnorm{\Delta \calGtilde_{i}}^2-\frac{1}{\sigma_i^2}\Vnorm{\Delta \calG_{i}}^2\right]
+  \frac{1}{s_i^2}\Vsprod{\Delta \calGtilde_{i}}{\tilde w_\theta(Z_i)}}_{A_i}
+\underbrace{ \Vsprod{\frac{1}{s_i^2}\Delta \calGtilde_{i}-\frac{1}{\sigma_i^2}\Delta \calG_{i}}{\varepsilon_i}}_{B_i}
\end{equation*}

where we write $\Delta \calGtilde_{i}:= \calGtilde(\theta)(Z_i)-\calGtilde(\thetatrue)(Z_i)$ and $\Delta \calG_{i}:= \calG(\theta)(Z_i)-\calG(\thetatrue)(Z_i)$ for ease of notation. Now, knowing that $Z_i \perp \varepsilon_i$, by the tower property and applying the Gaussian MGF for $\varepsilon_i$:

\begin{align*}
    \log&\IE_{\theta}^N\left[\exp\left(b W_i\right)\right]
    =\log\IE_{\zeta}^N \left[\IE_{\varepsilon}^N \left[\exp\left(b W_i\right)\mid Z_i\right]\right] = \log\IE_{\zeta}^N\left[\exp\left(b A_i\right) \IE_{\varepsilon}^N \left[\exp\left(b B_i\right) | Z_i\right] \right]\\
    &=\log\IE_\zeta^N\left[\exp\left(\frac{b}{2} \left[\frac{1}{s_i^2}\Vnorm{\Delta \calGtilde_{i}}^2-\frac{1}{\sigma_i^2}\Vnorm{\Delta \calG_{i}}^2\right]
+  \frac{b}{s_i^2}\Vsprod{\Delta \calGtilde_{i}}{\tilde w_\theta(Z_i)}\right)\times \exp\left(\frac{b^2 \sigma_i^2}{2}\Vnorm{\frac{1}{s_i^2}\Delta \calGtilde_{i}-\frac{1}{\sigma_i^2}\Delta \calG_{i} }^2\right)\right]\\
    &=\log\IE_\zeta^N \left[\exp\left(\frac{b}{2} \left[\left(\frac{1}{s_i^2}-\frac{1}{\sigma_i^2}\right)\Vnorm{\Delta \calGtilde_i}^2-\frac{1}{\sigma_i^2}\Vnorm{\tilde{w}_\theta(Z_i) - \tilde{w}_{\thetatrue}(Z_i)}^2\right]-\frac{b}{\sigma_i^2}\Vsprod{\Delta \calGtilde_{i}}{\tilde w_\theta(Z_i)-w_{\thetatrue}(Z_i)}\right. \right.\\
&\qquad\qquad\qquad + \left.\left. \frac{b}{s_i^2}\Vsprod{\Delta \calGtilde_{i}}{\tilde w_\theta(Z_i)}
+\frac{b^2 \sigma_i^2}{2}\Vnorm{\frac{1}{s_i^2}\Delta \calGtilde_{i}-\frac{1}{\sigma_i^2}\Delta \calG_{i} }^2\right) \right]\\
&\leq \log\IE_\zeta^N \left[\exp\left({\frac{b}{2} \left( \frac{1}{s_i^2}-\frac{1}{\sigma_i^2}\right)\Vnorm{\Delta \calGtilde_i}^2 }
+ {\Vnorm{\Delta\calGtilde_i}\left(\frac{b}{\sigma_i^2}\Vnorm{w_\theta(Z_i)-w_{\thetatrue}(Z_i)} + \frac{b}{s_i^2}\Vnorm{\tilde w_\theta(Z_i)}\right)} \right.\right.\\
&\qquad\qquad\qquad \left.\left. +{\frac{b^2 \sigma_i^2}{2}\Vnorm{\frac{1}{s_i^2}\Delta \calGtilde_{i}-\frac{1}{\sigma_i^2}\Delta \calG_{i} }^2}\right)   \right]\\
&\leq \log\IE_\zeta^N \left[\exp\left(\underbrace{\left(-\frac{b\sigma_i^2}{2}\left( 1 -\frac{\sigma_i^2}{s_i^2}\right)+\frac{b^2}{\sigma_i^2}\left|1-\frac{\sigma_i^2}{s_i^2}\right|^2 \right)\Vnorm{\Delta \calGtilde_i}^2 }_{(I)} +\underbrace{\frac{b^2}{\sigma_i^2}\Vnorm{\tilde{w}_{\thetatrue}(Z_i)-\tilde{w}_\theta(Z_i)}^2}_{(II)}\right.\right.\\
&\qquad\qquad\qquad  \left.\left. + \underbrace{\Vnorm{\Delta\calGtilde_i}\left(\frac{b}{\sigma_i^2}\Vnorm{\tilde w_\theta(Z_i)-w_{\thetatrue}(Z_i)} + \frac{b}{s_i^2}\Vnorm{\tilde w_\theta(Z_i)}\right)}_{(III)} \right)   \right]
\end{align*}

by Cauchy-Schwarz, and where to simplify we used computations
$$\Delta\calGtilde_i=\left(\calGtilde(\theta)-\calG(\theta)+\calG(\theta)-\calG(\thetatrue)+\calG(\thetatrue)-\calGtilde(\thetatrue)\right)(Z_i)=\Delta\calG_i + \tilde{w}_{\thetatrue}(Z_i)-\tilde{w}_{\theta}(Z_i)$$
to go from the first to the second line, and
\begin{align*}
\frac{b^2 \sigma_i^2}{2} \Vnorm{\frac{1}{s_i^2}\Delta \calGtilde_{i}-\frac{1}{\sigma_i^2}\Delta \calG_{i} }^2 
 &=\frac{b^2 \sigma_i^2}{2} \Vnorm{\left(\frac{1}{s_i^2}-\frac{1}{\sigma_i^2}\right)\Delta \calGtilde_i+\frac{1}{\sigma_i^2}\tilde{w}_{\thetatrue}(Z_i) - \frac{1}{\sigma_i^2}\tilde{w}_\theta(Z_i)}^2\\
 &\leq \frac{b^2}{\sigma_i^2}\left|1-\frac{\sigma_i^2}{s_i^2}\right|^2 \Vnorm{\Delta \calGtilde_{i}}^2 + \frac{b^2}{\sigma_i^2}\Vnorm{\tilde{w}_{\thetatrue}(Z_i)-\tilde{w}_\theta(Z_i)}^2.
\end{align*}

for the last line. Let us now control each term (I), (II), (III), using to Conditions \ref{cond:misspec-error} and \ref{cond:misspec-model}. We first remark that by Conditions \ref{cond:fm:bound}, for $\theta$ with $\|\theta\|_\scrR \leq M$:
\begin{align*}
    \Vnorm{\Delta\calGtilde_i}\leq \| \calGtilde(\theta)-\calGtilde(\thetatrue)\|_\infty 
    &\leq  \| \calGtilde(\theta)\|_\infty + \|\calGtilde(\thetatrue)\|_\infty \leq \Const{\calGtilde,B}(1+\| \theta\|_\scrR^{\gamma_B})+\Const{\calGtilde,B}(1+\| \thetatrue\|_\scrR^{\gamma_B}) \leq \Const{M, \calGtilde, \gamma_B}.
\end{align*}

\textbf{Controlling (I).} (I) is the standalone contribution of the noise misspecification. 
\begin{itemize}
    \item Under \cref{cond:misspec-error}~\ref{item:NMII.I}, when the variance is overestimated (i.e. $s_i^2>\sigma_i^2$) with $1-\sigma_i^2/s_i^2=\deltanoise=o(1)$, then the term in parenthesis is of order $-b\sigma_i^2\left(1-\frac{\sigma_i^2}{s_i^2}\right)$ and the contribution of (I) inside the exponential is negative: we can therefore ignore (I) as it will yield an upper bound by 1.
    \item Under \cref{cond:misspec-error}~\ref{item:NMII.II}, the term in parenthesis is of order 
    $$b\sigma_i^2\left|1-\frac{\sigma_i^2}{s_i^2}\right|\leq b\sigma_\infty^2 \deltanoise \leq b\sigma_\infty^2 C_{noise} \delta_N^2$$
\end{itemize}

yielding then $(I)\leq c\delta_N^2$ for some constant $c=c(M, \calGtilde, \gamma_B, b, \sigma_\infty^2)$.\\

\textbf{Controlling (II) and (III).} 
By \cref{cond:misspec-model}, $ \|\tilde{w}_\theta\|_\infty \leq \|\calG(\theta)-\calGtilde(\theta)\|_\infty \leq c(M)\deltamap$ with $\deltamap \leq C_{model}\delta_N^2$, yielding
\begin{align*}
    (II)\lesssim \|\tilde{w}_\theta\|_\infty^2 + \|\tilde{w}_{\thetatrue}\|_\infty^2 \lesssim {\deltamap}^2 \lesssim \delta_N^4
\end{align*}
and 
\begin{align*}
    (III)\lesssim \|\Delta\calGtilde_i\|_\infty \left( \|\tilde{w}_{\theta}\|_\infty + \|\tilde{w}_{\thetatrue}\|_\infty \right) \lesssim {\deltamap} \lesssim \delta_N^2
\end{align*}
for constants depending on $\{M, \calGtilde, b, s_0^2,\sigma_0^2,\sigma_\infty^2, C_{model}\}$.\\

Finally, going back to \cref{eq:target-com}, we obtain 
\begin{align*} 
    \IE_{\theta}^N\left[\left(\frac{q_\theta^N}{q_{\thetatrue}^N}\frac{p_{\thetatrue}^N}{p_{\theta}^N}\right)^b\right] \leq \exp(c_5\times N \delta_N^2)
\end{align*}
for some $c_5=c(b, s_0^{2}, \sigma_\infty^2, M, \Const{\calGtilde}, \Const{noise},\Const{model})$.
\end{proof}

\bigskip

\begin{proof}[Proof of \cref{prop:change-measure-peeling}]
    
We want to show that there exists a constant $c_6>0$ such that

\begin{equation}\label{eq:integral-prior}
\int_{\Theta_N(M)^c} \IE_{\thetatrue}^N[e^{\tilde\ell_N(\theta)-\tilde\ell_N(\thetatrue)}]d\Pi_N(\theta) = e^{-c_6N\delta_N^2}
\end{equation}

Using computation (\cref{eq:pseudo-loglik-comput}), and recalling $\tilde w_{\thetatrue} = \calG(\thetatrue) - \calGtilde(\thetatrue)$, under $\IP^N_{\thetatrue}$
\begin{align*}
\loglikmodel(\theta)-\loglikmodel(\thetatrue) 
&=-\frac{1}{2}\sum_{i=1}^N\frac{1}{s_i^2}\Vnorm{\calGtilde(\theta)(Z_i)-\calGtilde(\thetatrue)(Z_i)}^2 
-\sum_{i=1}^N \frac{1}{s_i^2} \Vsprod{\tilde w_{\thetatrue}(Z_i)}{\calGtilde(\thetatrue)(Z_i)-\calGtilde(\theta)(Z_i)}\\
&~~~~~~~~~~~~~~~~~-\sum_{i=1}^N \frac{1}{s_i^2}\Vsprod{\varepsilon_i}{\calGtilde(\thetatrue)(Z_i)-\calGtilde(\theta)(Z_i)}
\end{align*}   

we find that conditioned on the covariates $Z_1,...,Z_N$, and using the MGF of $\varepsilon_i$ which under $P_{\thetatrue}$ is a $N(0, \sigma_i^2)$ random variable:
\begin{align}
    \IE_{\thetatrue}^N&[e^{\loglikmodel(\theta)-\loglikmodel(\thetatrue)}|Z_1,...,Z_N]\nonumber\\
    &= \exp\left(-\frac{1}{2}\sum_{i=1}^N \frac{1}{s_i^2}\left(1-\frac{\sigma_i^2}{s_i^2}\right)\Vnorm{\calGtilde(\theta)(Z_i)-\calGtilde(\thetatrue)(Z_i)}^2 \right) \label{eq:variance-contribution} \\ 
    &~~~~~~~~~~~~~~~~~~~~~~\times\exp\left(-\sum_{i=1}^N \frac{1}{s_i^2}\Vsprod{\calG(\thetatrue)(Z_i)-\calGtilde(\thetatrue)(Z_i)}{\calGtilde(\thetatrue)(Z_i)-\calGtilde(\theta)(Z_i)}\right)\nonumber\\
    & = \exp\left(\frac{1}{2}s_0^{-2} N \deltanoise \|\calGtilde(\theta)-\calGtilde(\thetatrue)\|_\infty^2 \right)  \exp\left(N C \deltamap \|\calGtilde(\thetatrue)-\calGtilde(\theta)\|_\infty\right) \label{eq:exp-for-slicing}
\end{align}
for some $C=C(s_0^{-2}, \|\thetatrue\|_\scrR)$. 
The quantity $\|\calGtilde(\theta)-\calGtilde(\thetatrue)\|_\infty$ cannot be controlled globally on $\Theta_N(M)^c$; however we can control it on ``slices'' where the parameter has bounded $\scrR$-norm, and use tail decay of the prior via a slicing argument. Note that the integral \cref{eq:integral-prior} is $0$ outside of the support $\scrR$ of the prior; hence we can restrict computations to the integral over $\Theta_N(M)^c \cap \scrR$. Write
\begin{align}\label{eq:peeling-slices}
\Theta_N(M)^c \cap\scrR &= \bigcup_{\ell=\ell_0}^\infty \br{\Theta_N(M)^c \cap \{\theta\in \scrR : 2^\ell S\leq\|\theta\|_\scrR < 2^{\ell+1}S\}} := \bigcup_{\ell=\ell_0}^\infty \calP_\ell
\end{align}

for $S> \|\thetatrue\|_\scrR\vee 1$ and $\ell_0 = \lfloor \log_2(M/S) \rfloor$. On every slice, for every $\theta \in \calP_\ell~ (\ell \geq \ell_0)$:
\begin{align*}
    \| \calGtilde(\theta)-\calGtilde(\thetatrue)\|_\infty 
    &\leq  \| \calGtilde(\theta)\|_\infty + \|\calGtilde(\thetatrue)\|_\infty \leq \Const{\calGtilde,B}(1+\| \theta\|_\scrR^{\gamma_B})+\Const{\calGtilde,B}(1+\| \thetatrue\|_\scrR^{\gamma_B}) \leq 4 \Const{\calGtilde,B} (2^{\ell+1}S)^{\gamma_B}
\end{align*}

Coming back to \cref{eq:exp-for-slicing} and using that (see e.g. (2.21) in \cite{Nickl_2023})
\begin{equation}\label{eq:prior-tail}
\Pi_N(\calP_\ell) \leq \Pi_N(\|\theta\|_\scrR > 2^\ell S)\leq \rme^{-c_{prior}2^{2\ell}S^2N\delta_N^2},
\end{equation}
we see
\begin{align}
    \sum_{\ell=\ell_0}^\infty \int_{\calP_\ell} &\IE_{\thetatrue}^N[\rme^{\loglikmodel(\theta)-\loglikmodel(\thetatrue)}]d\Pi_N(\theta) \notag\\
    &\leq
    \sum_{\ell=\ell_0}^\infty \exp\left(8s_0^{-2} N\deltanoise \Const{\calGtilde,B}^2 (2^{\ell+1}S)^{2\gamma_B}+4NC\Const{\calGtilde,B} \deltamap N (2^{\ell+1}S)^{\gamma_B}\right)\exp(-{c_{prior}}2^{2\ell}S^2N\delta_N^2) \notag\\
    &\leq \sum_{\ell=\ell_0}^\infty \exp\left[-N\delta_N^2 2^{2\ell} S^2 \left(c_{prior} - c_a \frac{\deltanoise}{\delta_N^2} 2^{(2\gamma_B-2)\ell}S^{(2\gamma_B-2)} -c_b \frac{\deltamap}{\delta_N^2}2^{(\gamma_B-2)\ell}S^{(\gamma_B-2)}\right)\right] \notag \\
        &\leq \sum_{\ell=\ell_0}^\infty \exp\left[-N\delta_N^2 2^{2\ell}S^2\underbrace{ \left(c_{prior} - c_a \Const{noise}2^{(2\gamma_B-2)\ell} -c_b \Const{model}2^{(\gamma_B-2)\ell}\right)}_{(K(\ell))}\right] \label{eq:last-one-peeling}
\end{align}

under conditions \ref{cond:misspec-error} and \ref{cond:misspec-model}, for some constants $c_a, c_b$ depending on $\{ S, \Const{\calGtilde,B}, s_0^2, \|\thetatrue\|_\scrR \}$. We show that under our conditions, $(K(\ell))$ can be lower bounded by some strictly positive constant $\tau>0$ for every $\ell \geq \ell_0$. Let us look at contributions from each source of misspecification in $(K(\ell))$:
\begin{itemize}
\item Under \ref{item:NMII.I} where the variance gets overestimated, the 2nd term in $(K(\ell))$ accounting for noise misspecification actually disappears (since the term inside the exponential in \cref{eq:variance-contribution} is negative when $s_i^2> \sigma_i^2$). Only the contribution from model misspecification remains, and by Condition $\gamma_B \leq 2$ in \cref{cond:misspec-model} and for $\Const{model}$ small enough we obtain the desired result.
    \item Otherwise under \ref{item:NMII.II}, $\gamma_B \leq 1$. The contribution from noise misspecification decays with $\ell$, and the dominating contribution is the one from noise misspecification. For $\Const{noise}$  small enough the term $K(\ell)$ remains positive.  
\end{itemize}

Therefore, for $\Const{\mathrm{noise}}$ and $\Const{\mathrm{model}}$ sufficiently small, the term $(K(\ell))$ can be lower bounded by some constant $\tau>0$ for every $\ell \geq \ell_0$. The summands decay superexponentially with $\ell$ and we can upper bound the sum by the order of the first term, $O(\exp(-c\tau 2^{2\ell_0}N\delta_N^2))$. 


In the end, we obtain the desired $\cref{eq:integral-prior}$, for some constant $c_6=c_6(M, C_{\calGtilde, B}, \thetatrue, s_0^{2}, \Const{\mathrm{noise}}, \Const{\mathrm{model}})$. In particular, $c_6$ can be made as large as desired by increasing the radius $M$ in the definition of the regularisation set $\Theta_N(M)$.
\end{proof}

\subsection{Proof of \cref{prop:prior-mass-bn}}

\begin{proof}[Proof of \cref{prop:prior-mass-bn}]
    The proof copies that of Theorem 2.2.2 (step 2) in \cite{Nickl_2023} replacing $\calG$ with $\calGtilde$; with a simple perturbation argument. Precisely: let $M\in\pRz$, and assume that $\norm{\theta-\thetatrue}_{\scrR} \leq M$. As $\thetatrue\in\scrR$, we thus have $\norm{\theta}_{\scrR} \leq M + \norm{\thetatrue}_{\scrR} \eqqcolon\bar{M}\in\pRz$. Then, by \cref{cond:MM:III} and \cref{cond:fm:bound}:
        \begin{equation*}
        \norm{\calGtilde(\theta)}_\infty \leq \Const{\calGtilde,\rmB}\times(1+\bar{M}^{\gamma_B}) \leq c_7(\bar{M},\Const{\calGtilde,\rmB})\eqqcolon \rmU.
    \end{equation*}
    Thus, 
    \begin{align*}
         \Pi_N(\widetilde\calB_N) &= \Pi_N\left(\theta \in\Theta: d_{\calGtilde}(\theta,\thetatrue) \le \delta_N,\; \|\calGtilde(\theta) \|_{\infty} \le U \right)\\
        &\ge \Pi_N\left(\theta\in\Theta :d_{\calGtilde}(\theta,\thetatrue)  \le \delta_N,\; \|\theta - \thetatrue\|_{\scrR} \le M
        \right).
    \end{align*}
    Now with the Lipschitz-condition \cref{cond:fm:Liptwo}, the triange inequality and \cref{cond:MM:III} (noting that $(\calZ,\scrZ,\zeta)$ is a probability space), we obtain 
    \begin{align*}
        d_{\calGtilde}(\theta,\thetatrue) &\leq d_\calG(\theta,\thetatrue) + \norm{\calGtilde(\theta)-\calG(\theta)}_{\IL_\zeta^2(\calZ,V)} +\norm{\calGtilde(\thetatrue)-\calG(\thetatrue)}_{\IL_\zeta^2(\calZ,V)}\\
        & \leq d_\calG(\theta,\thetatrue) + c(\bar{M})\times\deltamap + c(\thetatrue)\deltamap\\
        & \leq  d_\calG(\theta,\thetatrue) + \const\br{\Mbar,\thetatrue,C_{model}}\delta_N^2\\
        & \leq \ConstLiptwo\times\br{1+\Mbar^{\gamma_2}}\times\norm{\theta-\thetatrue}_{(H^\kappa(\calM,W))^*}+\const\br{\Mbar,\thetatrue,C_{model}}\delta_N^2.
    \end{align*}
    The last line is upper bounded by $\delta_N$, if $\norm{\theta-\thetatrue}_{(H^\kappa(\calM,W))^*}\leq \frac{\delta_N}{c_8(\bar{M},\ConstLiptwo,\gamma_2)}$ with $c_8(\bar{M},\ConstLiptwo,\gamma_2) > \ConstLiptwo\times\br{1+\Mbar^{\gamma_2}}$ and $N = N(\Mbar,\thetatrue,C_{model})$ sufficently large. Then, the above probability can be further upper bounded by 
    \begin{align*}
        \Pi_N(\widetilde\calB_N) & \geq \Pi_N\left(
  \theta\in\Theta\; :\;
  \|\theta - \thetatrue\|_{(H^{\kappa}(\calM,W))^{*}} \leq  \frac{\delta_N}{c_8(\bar{M},\ConstLiptwo,\gamma_2)},\; \|\theta - \thetatrue\|_{\scrR} \le M
\right).
    \end{align*}
    Applying Corollary 2.6.18 in \cite{Gine_Nickl_2021} as well as the Gaussian correlation inequality, Theorem B.1.2 in \cite{Nickl_2023}, we obtain
    \begin{align*}
        \Pi_N(\widetilde\calB_N) & \geq \rme^{-\frac{1}{2} N \delta_N^2 \|\thetatrue\|_{\scrH}^{2}}\times
\Pi_N\left(
  \theta\in\Theta\; :\;\|\theta\|_{(H^{\kappa}(\calM,W))^{*}} \le \frac{\delta_N}{c_8(\bar{M},\ConstLiptwo,\gamma_2)}\right)
\times
\Pi_N\left(\|\theta\|_{\scrR} \le M\right).
    \end{align*}

The proof now follows the same steps as in Step 2 of Theorem 2.2.2 in \cite{Nickl_2023} by noting that any ball in $H^\alpha(\calM,W)$ (and $H^\alpha_c(\calM,W)$) can be covered in terms of the $\|\cdot\|_{(H^{\kappa}(\calM,W))^{*}}$-norm taking care of the additional dimension $d_W$, see also Lemma 4.9 in \cite{Siebel_2025}.
\end{proof}

\subsection{\texorpdfstring
  {Proof of \cref{thm:posterior-contraction}}
  {Proof of Theorem X}
}


\begin{proof}[Proof of \cref{thm:posterior-contraction}]
    The beginning of the proof follows the steps of Theorem 1.3.2. in \cite{Nickl_2023}. From \cref{prop:prior-mass-bn} and \cref{lemma:contraction-auxiliary} applied to  $\nu = \Pi_N(\cdot)/\Pi_N(B_N)$ for $B_N = \tilde\calB_N, K=A+ s_0^{-2}$, it follows that the events
    \begin{equation}\label{eq:setsAN}
        A_N = \brK{ \int_\Theta e^{\tilde\ell_N(\theta)-\tilde\ell_N(\thetatrue)}\rmd\Pi_N(\theta) \geq \Pi_N(B_N)e^{-s_0^{-2}N\delta_N^2} \geq e^{-(A+ s_0^{-2})N\delta_N^2}}
    \end{equation}
    satisfy $\lim_{N\to\infty}\IP_{\thetatrue}^N(A_N)= 1$. Given $M,\rho \in\pRz$, we now introduce the sets
    \begin{equation*}
        \bar{\Theta}_N(M)\coloneqq \Theta_N(M)\cap\brK{\theta\in\Theta\;:\;d_\calG(\theta,\thetatrue)\leq \rho\delta_N}
    \end{equation*}
    and denote by $ \bar{\Theta}_N(M)^c$ their complements in $\Theta$. Now controlling the events $A_N$ as above and using the tests $\Psi_N$ from \cref{prop:testing}, the target probability \cref{eq:target-proba} can be written for $N\to \infty:$
    \begin{align*}
        \IP_{\thetatrue}^N&\br{\Pitilde_N(\bar\Theta_N(M)^c\mid D_N) \geq e^{-bN\delta_N^2}}  \\
        &= \IP_{\thetatrue}^N\br{\Pitilde_N(\bar\Theta_N(M)^c\mid D_N) \geq e^{-bN\delta_N^2}, A_N} +\IP_{\thetatrue}^N\br{\Pitilde_N(\bar\Theta^c_N) \geq e^{-bN\delta_N^2}, A_N^c} \\
        &=\IP_{\thetatrue}^N\left( \frac{\int_{\bar\Theta_N(M)^c} e^{\tilde\ell_N(\theta)-\tilde\ell_N(\thetatrue)}\rmd\Pi_N(\theta)}{\int_{\Theta} e^{\tilde\ell_N(\theta)-\tilde\ell_N(\thetatrue)}\rmd\Pi_N(\theta)} \geq e^{-bN\delta_N^2},\Psi_N=0,~ A_N\right) + o(1)\\
        &\leq \IP_{\thetatrue}^N\left( \int_{\bar\Theta_N(M)^c} e^{\tilde\ell_N(\theta)-\tilde\ell_N(\thetatrue)}(1-\Psi_N)\rmd\Pi_N(\theta)\geq e^{-(b+A+s_0^{-2})N\delta_N^2}\right) + o(1).
    \end{align*}
    It remains to upper bound the last probability. Markov's inequality and Fubini's theorem yield
    \begin{align*}
        e&^{(b+A+s_0^{-2})N\delta_N^2}\times\IE_{\thetatrue}^N\brE{\int_{\bar\Theta_N(M)^c} e^{\tilde\ell_N(\theta)-\tilde\ell_N(\thetatrue)}(1-\Psi_N)d\Pi_N(\theta)}\\
        &= e^{(b+A+s_0^{-2})N\delta_N^2}\Bigg[~\int_{\Theta_N(M)^c} \IE_{\thetatrue}^N[e^{\tilde\ell_N(\theta)-\tilde\ell_N(\thetatrue)}(1-\Psi_N)]d\Pi_N(\theta) \\
        &~~~~~~~~~~~~~~~~~~~~~~~~~~~~~~~~~~~~~+ \int\limits_{\{d_{\mathcal{G}}(\theta,\thetatrue)\geq \rho\delta_N\}} \!\IE_{\thetatrue}^N[e^{\tilde\ell_N(\theta)-\tilde\ell_N(\thetatrue)}(1-\Psi_N)]d\Pi_N(\theta)\Bigg] \\
        &\leq e^{(b+A+s_0^{-2})N\delta_N^2}\left[~2\int_{\Theta_N(M)^c} \IE_{\thetatrue}^N[e^{\tilde\ell_N(\theta)-\tilde\ell_N(\thetatrue)}(1-\Psi_N)]d\Pi_N(\theta) \right.\\
        &~~~~~~~~~~~~~~~~~~~~~~~~~~~~~~~~~~~~~+\left. \int\limits_{\{d_{\mathcal{G}}(\theta,\thetatrue)\geq \rho\delta_N\}\cap \Theta_N(M)} \!\IE_{\thetatrue}^N[e^{\tilde\ell_N(\theta)-\tilde\ell_N(\thetatrue)}(1-\Psi_N)]d\Pi_N(\theta)\right]  \tag{$\dagger$}\label{eq:proba-to-bound}
    \end{align*}




where we used that $\bar\Theta^c_N = \Theta_N(M)^c \cup \{d_{\mathcal{G}(\theta,\thetatrue)}\geq \rho\delta_N\}$. Let us now control each the integral.\\

\textbf{Step 1:} We first control 
$$\int_{\brK{\theta \in \Theta_N(M) : d_\calG(\theta,\thetatrue)> \rho\delta_N} }\IE_{\thetatrue}^N\left[\frac{q_\theta^N}{q_{\thetatrue}^N}(1-\Psi_N)\right]d\Pi_N(\theta)$$

Note that
\begin{align*}
    \IE_{\thetatrue}^N\left[\frac{q_\theta^N}{q_{\thetatrue}^N}(1-\Psi_N)\right] 
    &= \IE_{\thetatrue}^N\left[(1-\Psi_N)\frac{q_\theta^N}{q_{\thetatrue}^N}\frac{p_{\thetatrue}^N}{p_{\theta}^N}\frac{p_\theta^N}{p_{\thetatrue}^N}\right]
    =\IE_{\theta}^N\left[(1-\Psi_N)\frac{q_\theta^N}{q_{\thetatrue}^N}\frac{p_{\thetatrue}^N}{p_{\theta}^N}\right]\\
    &\leq \left(\IE_{\theta}^N\left[(1-\Psi_N)^2\right]\right)^{1/2} \left(\IE_{\theta}^N\left[\left(\frac{q_\theta^N}{q_{\thetatrue}^N}\frac{p_{\thetatrue}^N}{p_{\theta}^N}\right)^2\right]\right)^{1/2} \text{ by Cauchy-Schwarz }\\
    &\leq \left(\IE_{\theta}^N\left[(1-\Psi_N)\right]\right)^{1/2} \left(\IE_{\theta}^N\left[\left(\frac{q_\theta^N}{q_{\thetatrue}^N}\frac{p_{\thetatrue}^N}{p_{\theta}^N}\right)^2\right]\right)^{1/2} \text{as $|1-\Psi_N| \leq 1$}\\
    &\leq \exp{(- \bar c/2 \times N\delta_N^2)} \exp{(c_5 N\delta_N^2)},
\end{align*}

where we used the assumptions on the tests (\cref{prop:testing}) to control the factor on the left, and \cref{prop:change-measure}, to control the factor on the right, since $\Theta_N(M) \subset \{ \theta: \|\theta\|_\scrR \leq M \}$. \\

\textbf{Step 2:} On $\Theta^c_N$, notice that $\IE_{\thetatrue}^N[e^{\tilde\ell_N(\theta)-\tilde\ell_N(\thetatrue)}(1-\Psi_N)] \leq \IE_{\thetatrue}^N[e^{\tilde\ell_N(\theta)-\tilde\ell_N(\thetatrue)}]$ so that 
$$\int_{\Theta_N(M)^c} \IE_{\thetatrue}^N[e^{\tilde\ell_N(\theta)-\tilde\ell_N(\thetatrue)}(1-\Psi_N)]d\Pi(\theta)=\exp(-c_6\times N \delta_N^2)$$ by virtue of \cref{prop:change-measure-peeling}. \bigskip


In the end, the target probability can be bounded by above by
\begin{align*}
    e^{(b+A+\bar{s}_N^{-2})N\delta_N^2}\left(e^{-(\bar c/2 -c_5)N\delta_N^2} + e^{-c_6N\delta_N^2}\right)\to 0, \text{ as } N\to\infty
\end{align*}

for $\rho$ large enough such that $\bar c > 2(b+A+ s_0^{-2}+c_5)$, and also by assumption $c_6 > b + A + s_0^{-2}$.
\end{proof}

\subsection{\texorpdfstring
  {Proof of \cref{thm:posterior-contraction-inverse}}
  {Proof of Theorem X}
}

\begin{proof}[Proof of \cref{thm:posterior-contraction-inverse}]

As in Theorem 2.3.1 of \cite{Nickl_2023}, the proof of \cref{eq:target-proba-inverse} follows from\cref{thm:posterior-contraction} and \cref{cond:inverseproblem} yielding the set inclusion
$$\{\theta \in \Theta_N(M) : d_\calG(\theta,\thetatrue)\leq \rho \delta_N \} \subset \{\theta \in \Theta_N(M) : d_\calG(\theta,\thetatrue)\leq \Const{\calG,inv}(M)\} (\rho \delta_N)^\tau \}.$$

For the convergence of the posterior mean \cref{eq:posterior-mean}, following the proof of Theorem 2.3.2 in \cite{Nickl_2023}, we set $\eta_N = \Const{\calG,inv}(M)(\rho\delta_N)^{\eta}$. Then, by Jensen and Cauchy--Schwarz,
\begin{align*}
\left\| \mathbb{E}^{\Pitilde}[\theta \mid D_N] - \thetatrue\right\|_{L^2}
&\le \mathbb{E}^{\Pitilde}\!\left[\|\theta - \thetatrue\|_{L^2} \mid D_N \right] \\
&\le \eta_N + \mathbb{E}^{\Pitilde}\!\left[\|\theta - \thetatrue\|_{L^2}\mathds{1}_{\!\left\{\|\theta - \thetatrue\|_{L^2} > \eta_N\right\}}\mid D_N \right] \\
&\le \eta_N + \mathbb{E}^{\Pitilde}\!\left[\|\theta - \thetatrue\|_{L^2}^2 \mid D_N \right]^{1/2}\Pitilde\!\left(\|\theta - \thetatrue\|_{L^2} > \eta_N \mid D_N \right)^{1/2},
\end{align*}

and we now show that the last term is $O_{P^N_{\thetatrue}}(\eta_N)$ to prove the theorem. We recall the sets $A_N$ from \cref{eq:setsAN} which satisfy $\IP^N_{\thetatrue}(A_N) \to 1$ as $N \to \infty$. Then using \cref{eq:target-proba-inverse}, Markov's inequality, Fubini's theorem, 

\begin{align*}
& \IP^N_{\thetatrue}\!\Big(\mathbb{E}^{\Pitilde}\!\left[\|\theta - \thetatrue\|_{L^2}^2 \mid D_N\right]\Pitilde\!\left(\|\theta - \thetatrue\|_{L^2} > \eta_N \mid D_N\right)> \eta_N^2 \Big)\\
&\le \IP^N_{\thetatrue}\!\Big(\mathbb{E}^{\Pitilde}\!\left[\|\theta - \thetatrue\|_{L^2}^2 \mid D_N\right] e^{-bN\delta_N^2} > \eta_N^2 \Big) + o(1)\\
&\le \IP^N_{\thetatrue}\!\Big(e^{-bN\delta_N^2} \frac{\int \|\theta - \thetatrue\|_{L^2}^2 e^{\tilde\ell_N(\theta)-\tilde\ell_N(\thetatrue)}\, d\Pi(\theta)}{\int e^{\tilde\ell_N(\theta)-\tilde\ell_N(\thetatrue)}\, d\Pi(\theta)}> \eta_N^2,\; A_N \Big) + o(1)\\
&\le e^{(A+s_0^{-2}-b)N\delta_N^2}\,\eta_N^{-2} \int \|\theta - \thetatrue\|_{L^2}^2 \IE^N_{\thetatrue}\!\left[e^{\tilde\ell_N(\theta)-\tilde\ell_N(\thetatrue)}\right] d\Pi(\theta)
+ o(1)
\end{align*}

As in the proof of \cref{thm:posterior-contraction}, we cannot directly upper bound $\IE^N_{\thetatrue}\!\left[e^{\tilde\ell_N(\theta)-\tilde\ell_N(\thetatrue)}\right]$ by $1$ because of the misspecification. We again make use of the change of measures propositions. Note that we can decompose the integral:
\begin{align*}
    &\int \|\theta - \thetatrue\|_{L^2}^2 \IE^N_{\thetatrue}\!\left[e^{\tilde\ell_N(\theta)-\tilde\ell_N(\thetatrue)}\right] d\Pi(\theta)\\
    &= \int_{\Theta_{N}(M)} \|\theta - \thetatrue\|_{L^2}^2 \IE^N_{\thetatrue}\!\left[e^{\tilde\ell_N(\theta)-\tilde\ell_N(\thetatrue)}\right] d\Pi(\theta) 
    + \int_{\scrR\cap\Theta_{N}(M)^c} \|\theta - \thetatrue\|_{L^2}^2 \IE^N_{\thetatrue}\!\left[e^{\tilde\ell_N(\theta)-\tilde\ell_N(\thetatrue)}\right] d\Pi(\theta)
\end{align*}

For the first term, notice that $\Theta_{N}(M) \subset \scrR(M)$, and that
$\IE^N_{\thetatrue}\!\left[e^{\tilde\ell_N(\theta)-\tilde\ell_N(\thetatrue)}\right]=\IE^N_{\theta}\!\left[\frac{q^N_\theta}{q^N_{\thetatrue}}\frac{p^N_{\thetatrue}}{p^N_\theta}\right]$, therefore \cref{prop:change-measure} applies with $b=1$, yielding
\begin{align*}
    \int_{\Theta_{N}(M)} \|\theta - \thetatrue\|_{L^2}^2 \IE^N_{\thetatrue}\!\left[e^{\tilde\ell_N(\theta)-\tilde\ell_N(\thetatrue)}\right] d\Pi(\theta) \leq e^{c_5 N\delta_N^2} \int_{\Theta_{N}(M)} \|\theta - \thetatrue\|_{L^2}^2 d\Pi(\theta) \leq C  e^{c_5 N\delta_N^2}
\end{align*}

where we used that the Gaussian measure $\Pi'$ used as the base prior is supported in $L^2$ and integrates $\|\cdot\|_{L^2}^2$ to a finite constant.\\

To deal with the second term, we adapt the slicing argument from \cref{prop:change-measure-peeling}, noting that on every slice $\calP_\ell$, because $\scrR$ embeds continuously into $L^2$:

$$ \|\theta - \thetatrue \|_{L^2} \leq \|\theta - \thetatrue \|_{\scrR} \leq \| \theta\|_\scrR + \| \thetatrue \|_\scrR \leq 2 (2^{\ell+1}S)$$

and plugging into \cref{eq:last-one-peeling} we get
\begin{align*}
\int_{\scrR\cap\Theta_{N}(M)^c} \|\theta - \thetatrue\|_{L^2}^2 \IE^N_{\thetatrue}\!\left[e^{\tilde\ell_N(\theta)-\tilde\ell_N(\thetatrue)}\right] d\Pi(\theta)
&\leq \sum_{\ell= \ell_0}^\infty \int_{\calP_\ell} \|\theta - \thetatrue\|_{\scrR}^2 \IE^N_{\thetatrue}\!\left[e^{\tilde\ell_N(\theta)-\tilde\ell_N(\thetatrue)}\right] d\Pi(\theta)\\
&\leq 4 \sum_{\ell= \ell_0}^\infty \int_{\calP_\ell} (2^{\ell+1}S)^2 \exp(-N\delta_N^2 2^{2\ell}K(\ell))
\end{align*}

which converges to a constant $C'$ by \cref{lem:ExpoSum}. Therefore eventually: 

\begin{align*}
&e^{(A+s_0^{-2}-b)N\delta_N^2}\,\eta_N^{-2} \int \|\theta - \thetatrue\|_{L^2}^2 \IE^N_{\thetatrue}\!\left[e^{\tilde\ell_N(\theta)-\tilde\ell_N(\thetatrue)}\right] d\Pi(\theta)\\
&\leq e^{(A+s_0^{-2}-b)N\delta_N^2}\,\eta_N^{-2} \left( Ce^{c_5N\delta_N^2} + C' \right)
\end{align*}

We conclude for $b$ large enough such that $c_6-A-s_0^{-2}>b> c_5 + A + s_0^{-2}$. To ensure this is possible, we look at the growth of constants $c_6$ and $c_5$ with $M$. $c_6$ is of order $2^{2\ell_0}=2^{2\log_2(M/S)}=O(M^2)$ by \cref{eq:last-one-peeling}. $c_5$ grows at the rate of $O(s_0^{-2}C_{noise}(1+M^{\gamma_B})^2 \vee C_{model}(1+M^{\gamma_B})) \lesssim (C_{noise}\vee C_{model})M^2$ by \cref{eq:exp-for-slicing}, under Conditions \ref{cond:misspec-error} and \ref{cond:misspec-model}. Therefore for constants $C_{noise}$ and $C_{model}$ sufficiently small, by choosing $M$ and then $\rho$ large enough, convergence of the posterior mean to the ground truth holds at the desired rate. 

\end{proof}

%% file: Sections/_app_Prior.tex
\section{Choice of Priors}\label{sec:app:prior}

In this section we review the construction of Gaussian base prior distributions $\Pi'$ as they are needed in \cref{cond:BasePrior}.
The following examples and remarks summarize the constructions considered in \cite{Nickl_2023}, \cite{NicklTiti_2024}, \cite{Konen_Nickl_2025}, and \cite{Nickl_2025_BvMRDE}.

\begin{example}[General prior]\label{example:prior:generic}
Let $\br{\lambda_j,e_j}_{j\in\IN}\subseteq(0,\infty)\times\IL^2(\calM,W)$ be an orthonormal basis of the linear subspace $\Theta\subseteq \IL^2(\calM,W)$ such that there exist constants $\Const{1},\Const{2}>0$ with
\begin{equation*}
    \forall j\in\IN:\quad \Const{1} j^{\frac{2}{d}} \leq \lambda_j \leq \Const{2} j^{\frac{2}{d}}.
\end{equation*}
Let $\alpha>0$ be a fixed smoothness parameter. Given an i.i.d.\ sequence $\br{g_j}_{j\in\IN}\sim\operatorname{N}(0,1)$, we define a centred Gaussian process $\brK{\calW(z)\;:\; z\in\calM}$ by
\begin{equation}\label{eq:prior:process}
    \forall z\in\calM:\quad \calW(z)\coloneqq \sum_{j\in\IN}\lambda_j^{-\frac{\alpha}{2}} g_j e_j(z).
\end{equation}
If $(\lambda_j^\alpha)_{j\in\IN}\in\ell^2(\IN)$, it follows from \cite[Example 2.6.15]{Gine_Nickl_2021} that its RKHS $\scrH$ is given by
\begin{equation*}
    \scrH \coloneqq \brK{\sum_{j\in\IN}\lambda_j^{\frac{\alpha}{2}} h_j e_j \;:\; (h_j)_{j\in\IN}\in\ell^2(\IN)}
    = \Hdot^\alpha(\calM,W).
\end{equation*}
Moreover, the process $\brK{\calW(z)\;:\; z\in\calM}$ converges in $\Hdot^\beta(\calM,W)$ since
\begin{equation*}
    \sum_{j\in\IN}\lambda_j^{\beta}\IE\brE{\sprod{\calW}{e_j}{\IL^2(\calM)}^2}
    = \sum_{j\in\IN}\lambda_j^{\beta-\alpha}
    \lesssim \sum_{j\in\IN} j^{\frac{2(\beta-\alpha)}{d}} < \infty,
\end{equation*}
provided that $\beta < \alpha - \frac{d}{2}$. We thus define the prior $\Pi'$ on $\Theta$ as the resulting Borel law $\Pi'=\operatorname{Law}(\calW)$, whose support is given by $\scrR\coloneqq \Hdot^\beta(\calM,W)$. If $\beta>\frac{d}{2}$, the Sobolev embedding $\Hdot^\beta(\calM,W)\hookrightarrow C^0(\calM,W)$ implies that $\Pi'$ also defines a law on $C^0(\calM,W)$.
\end{example}

\begin{remark}\label{rem:app:prior}
\mbox{}
\begin{enumerate}[label=\roman*)]
\item As an alternative to the Gaussian random series approach, \emph{Whittle--Matérn prior measures} are commonly used for the construction of infinite-dimensional priors; see \cite[Theorem~B.1.3]{Nickl_2023} for further details.

\item If $\calM=\domain\subseteq\IR^d$ is a bounded domain with smooth boundary $\partial\domain$, one may wish to encode boundary behaviour into the prior $\Pi'$. This is particularly relevant for the Darcy problem; see, for instance, \cite{Giordano_Nickl_2020}. To this end, consider the centred Gaussian process $\brK{\calW(z)\;:\; z\in\domain}$ and a smooth cut-off function $\varphi\in C_c^\infty(\domain)$. Defining
\begin{equation*}
    \brK{\calW_\varphi(z)\coloneqq \calW(z)\varphi(z)\;:\; z\in\domain}
\end{equation*}
yields a centred Gaussian process with RKHS
\begin{equation*}
    \scrH_\varphi\coloneqq \brK{h\varphi\;:\; h\in\scrH} \subseteq H_c^\alpha(\domain,W),
\end{equation*}
whose support is continuously embedded into $H_c^\beta(\domain,W)$ for $\beta<\alpha-\frac{d}{2}$.

\item\label{rem:prior:item:finitediemsnional}
High-dimensional \emph{sieved} priors: For computational reasons, it can be attractive to consider finite- (but high-) dimensional priors. Fix a dimension $D\in\IN$ and consider the truncated Gaussian random series
\begin{equation*}
    \forall z\in\calM:\quad \calW^D(z)\coloneqq \sum_{j\nset{D}}\lambda_j^{-\frac{\alpha}{2}} g_j e_j(z),
\end{equation*}
which law defines a Gaussian prior distribution $\Pi'=\operatorname{Law}(\calW^D)$ on the finite-dimensional space
\begin{equation*}
    E_D \coloneqq \operatorname{span}\brK{e_j\;:\; j\nset{D}} \simeq \IR^D.
\end{equation*}
More precisely,
\begin{equation*}
    \Pi' = \operatorname{N}\br{0,\operatorname{diag}\br{\lambda_j^{-\alpha}\;:\; j\nset{D}}}.
\end{equation*}
In this work, we restrict ourselves to priors constructed as in \cref{example:prior:generic}, noting that analogous results can be obtained with only minor modifications of the proofs. We refer to \cite{Nickl_2023} (in particular Exercise~2.4.3), as well as to \cite{Bohr_Nickl_2024} and \cite{Giordano_Nickl_2020}.
\end{enumerate}
\end{remark}

The following example summarizes explicit choices for priors $\Pi'$ when studying the Reaction Diffusion Equation and the 2D-Navier-Stokes Equation in \cref{sec:examples}.
\begin{example}[Gaussian process priors for RDE and NSE]\label{example:prior:PDE}
\mbox{}
\begin{enumerate}[label=\roman*)]
\item\label{item:prior:RDE}
\textbf{Reaction--diffusion equation:} Based on \cref{eq:prior:process}, choose the orthonormal system from \cref{eq:ONB:torus}, which yields a prior supported on $\scrR=\Hdot^\beta(\torus)$ for all $\beta<\alpha-\frac{d}{2}$. See also \cite{Nickl_2025_BvMRDE}.

\item\label{item:prior:NSE}
\textbf{2D Navier--Stokes equation:} Based on \cref{eq:prior:process}, let $k=(k_1,k_2)\in\IZ^2\setminus\brK{(0,0)}$ and define
\begin{equation*}
    c_k(x) \propto (k_2,-k_1)^{T}\cos(2\pi k\cdot x),
    \qquad
    s_k(x) \propto (k_2,-k_1)^{T}\sin(2\pi k\cdot x),
\end{equation*}
which are eigenfunctions of the Stokes operator $A=-P\Delta$. After enumerating $k=k_j$, $j\in\IN$, we define the eigenpairs $(\lambda_j,e_j)_{j\in\IN}$ from \cref{example:prior:generic} by
\begin{equation*}
    e_{2j-1}\coloneqq c_{k_j},\quad
    e_{2j}\coloneqq s_{k_j},\quad
    Ae_j=\lambda_j e_j,\quad
    0<\lambda_j\simeq |k_j|.
\end{equation*}
The final identity follows from Weyl's law for the eigenvalues $\lambda_j$; see Proposition~4.14 in \cite{ConstantinFoias_1989}. Choosing $\Pi'=\operatorname{Law}(\calW)$ then induces a law on $\Hddiamond^\beta$ for all $\beta<\alpha-\frac{d}{2}$. See also \cite{Konen_Nickl_2025}.
\end{enumerate}
\end{example}

%% file: Sections/_app_MEstimation.tex
\section{Mild Misspecification in M-Estimation}\label{app:MAP}

While posterior contraction appears to be delicate with respect to (mild) misspecification, it is well known that frequentist M-estimation techniques are quite robust; see \cite{Geer_2000a,Vaart_1998}. In this section, we revisit the results of \cite{Siebel_2025} 
while generalizing the techniques of \cite{Nickl_vdGeer_Wang_2020} to allow heteroscedasticity and (mild) misspecification,
and unbounded forward maps, such as the solution maps to the Reaction Diffusion equation and the 2D-Navier-Stokes Equations.
While allowing a direct comparison between Bayesian and frequentist approaches, the favourable asymptotic behaviour of the MAP estimator is used to prove \cref{prop:testing}.\\

Recall the observation scheme presented in \cref{subsec:ObservationModel}. Instead of assuming heteroscedastic Gaussian errors, M-estimation techniques allow us to consider more general distributions.

\begin{condition}[Bernstein Condition]\label{ass:errorBernstein}
    The error terms $\varepsilon_1,\dots,\varepsilon_N$ are independent, heteroskedastic, and centred $V$-valued random variables that satisfy a Bernstein-type condition. That is, there exists a family $(\sigma_i^2)_{i=1}^N\subseteq(0,\infty)$ and a finite constant $\rmB\in(0,\infty)$ such that for all $i=1,\dots,N$
    \begin{equation*}
        \forall v\in V:\quad \IE\brE{\Vsprod{\varepsilon_i}{v}} = 0
        \quad\text{and}\quad
        \IE\brE{\abs{\Vsprod{\varepsilon_i}{v}}^2} \leq \sigma_i^2\,\Vnorm{v}^2,
    \end{equation*}
    as well as
    \begin{equation*}
        \forall v\in V,\; k\in\IN_{\geq 2}:\quad
        \IE\brE{|\Vsprod{\varepsilon_i}{v}|^k}
        \leq \frac{k!}{2}\,\sigma_i^2\,\rmB^{k-2}\,\Vnorm{v}^k.
    \end{equation*}
    In addition, there exist $\sigma_0,\sigma_\infty\in(0,\infty)$, such that 
                \begin{equation*}
                    \sigma_0^2\leq\min_{i\in\nset{N}}\sigma_i^2 \leq \max_{i\in\nset{N}}\sigma_i^2 \leq \sigma_\infty^2.
                \end{equation*}
\end{condition}

\begin{example}\label{ex:bernstein}
    In the following, let $N\in\IN$ be fixed.
    \begin{enumerate}[label=\roman*)]
        \item Let $V=\IR^d$ with $d\in\IN$. Assume that $\varepsilon_i\sim\operatorname{N}(0,\Sigma_i)$ independently, with positive semi-definite covariance matrices $\Sigma_i\in\IR^{d\times d}$ for $i=1,\dots,N$, such that the largest eigenvalues are uniformly bounded, i.e.,
        \[
            \max_{i\nset{N}}\lambda_{\max}(\Sigma_i)\leq \Sigma_\infty\in(0,\infty).
        \]
        Then, for all $v\in\IR^d$, we have
        \begin{equation*}
            \IE\brE{\sprod{\varepsilon_i}{v}{\IR^d}} = 0
            \quad\text{and}\quad
             \IE\brE{\abs{\sprod{\varepsilon_i}{v}{\IR^d}}^2}
            \leq \lambda_{\max}(\Sigma_i)\,\IRdnorm{v}^2.
        \end{equation*}
        Moreover,
        \begin{equation*}
            \forall v\in\IR^d,\; k\in\IN_{\geq 2}:\quad
            \IE\brE{|\sprod{\varepsilon_i}{v}{\IR^d}|^k}
            \leq \frac{k!}{2}\,(v^\top\Sigma_i v)^{k/2}\,\IRdnorm{v}^k
            \leq \frac{k!}{2}\,\lambda_{\max}(\Sigma_i)\,\Sigma_\infty^{\frac{k-2}{2}}\,\IRdnorm{v}^k.
        \end{equation*}
        Thus, $\varepsilon_1,\dots,\varepsilon_N$ satisfy \cref{ass:errorBernstein} with $\sigma_i^2=\lambda_{\max}(\Sigma_i)$, $i\leq N$, and $\rmB=\sqrt{\Sigma_\infty}$.
        In particular, if $\Sigma_i = v_i^2\operatorname{diag}(1,\dots,1)\in\IR^{d\times d}$ with $(v_i^2)_{i\in\nset{N}}\subseteq(0,v_\infty^2)$ for some $v_\infty^2>0$, we have $(\sigma_i^2)_{i\in\nset{N}}=(v_i^2)_{i\in\nset{N}}$ and $\rmB = v_{\infty}$.

        \item Let $\varepsilon_1,\dots,\varepsilon_N$ be independent, centred, $V$-valued random variables such that $\Vnorm{\varepsilon_i}\leq \rmB_i\in(0,\infty)$ for all $i=1,\dots,N$. Further assume that $\max_{i\nset{N}}\rmB_i\leq \rmB_\infty\in(0,\infty)$. Then, for all $v\in V$,
        \begin{equation*}
            \IE\brE{\Vsprod{\varepsilon_i}{v}} = 0
            \quad\text{and}\quad
            \IE\brE{\abs{\Vsprod{\varepsilon_i}{v}}^2}
            \leq \rmB_i^2\,\Vnorm{v}^2,
        \end{equation*}
        as well as
        \begin{equation*}
            \forall v\in V,\; k\in\IN_{\geq 2}:\quad
            \IE\brE{|\Vsprod{\varepsilon_i}{v}|^k}
            \leq \rmB_i^k\,\Vnorm{v}^k
            \leq \frac{k!}{2}\,\rmB_i^2\,\rmB_\infty^{k-2}\,\Vnorm{v}^k.
        \end{equation*}
        Thus, $\varepsilon_1,\dots,\varepsilon_N$ satisfy \cref{ass:errorBernstein} with $\sigma_i^2=\rmB_i^2$ and $\rmB=\rmB_\infty$.
    \end{enumerate}
\end{example}

In the Bayesian approach, we have imposed in \cref{ass:operatorI} forward regularity assumptions, that hold locally on bounded balls of $\scrR$. For M-estimation techniques, however, to obtain statistical guarantees for \textit{global optimizer} we generally need 
assumptions on the forward map that quantify its regularity more globally.

\begin{condition}[Forward Regularity -- Global]\label{ass:operatorII}
    Let $\Theta\subseteq\IL^2(\calM,W)$ be the parameter space. Let $(\scrR,\norm{\cdot}_{\scrR})$ be a separable normed subspace of $\Theta$ such that
    \[
        (\scrR,\norm{\cdot}_{\scrR})\hookrightarrow (B^\eta,\norm{\cdot}_{B^\eta}),
    \]
    where $B^\eta$ is either $C^\eta(\calW,W)$ or $H^\eta(\calW,W)$ for some $\eta\geq 0$. Furthermore, let $\Thetadiamond\subseteq\Theta$.
    \begin{enumerate}[label=\textbf{[gFR\arabic*]}]
        \item\label{cond:fm:Liptwoglobal}
        There exist constants $\ConstLiptwo'>0$ and $\gamma_2,\kappa\geq 0$ such that for all $\theta_1,\theta_2\in\Thetadiamond\cap\scrR$,
        \begin{equation*}
            \norm{\calG(\theta_1)-\calG(\theta_2)}_{\IL_\zeta^2(\calZ,V)}
            \leq \ConstLiptwo'\,
            \bigl(1+\norm{\theta_1}_{\scrR}^{\gamma_2}\lor\norm{\theta_2}_{\scrR}^{\gamma_2}\bigr)\,
            \norm{\theta_1-\theta_2}_{(H^\kappa(\domain,W))^*}.
        \end{equation*}

        \item\label{cond:fm:boundglobal}
        There exist constants $\Const{\calG,B}'>0$ and $\gamma_B\geq 0$ such that for all $\theta\in\Thetadiamond\cap\scrR$,
        \begin{equation*}
            \norm{\calG(\theta)}_\infty
            \leq \Const{\calG,B}'\,
            \bigl(1+\norm{\theta}_{\scrR}^{\gamma_B}\bigr).
        \end{equation*}

        \item\label{cond:fm:Lipinfglobal}
        There exist constants $\ConstLipinfty'>0$ and $\gamma_\infty\geq 0$ such that for all $\theta_1,\theta_2\in\Thetadiamond\cap\scrR$,
        \begin{equation*}
            \norm{\calG(\theta_1)-\calG(\theta_2)}_\infty
            \leq \ConstLipinfty'\,
            \bigl(1+\norm{\theta_1}_{\scrR}^{\gamma_\infty}\lor\norm{\theta_2}_{\scrR}^{\gamma_\infty}\bigr)\,
            \norm{\theta_1-\theta_2}_{B^\eta}.
        \end{equation*}
    \end{enumerate}
\end{condition}

We see that the conditions imposed in \cref{ass:operatorII} are stronger and hence imply the conditions imposed in \cref{ass:operatorI}.

\subsection{Tikhonov-regularized estimator}

Let $N\in\IN$ and $\thetatrue\in\Theta$. Given data $D_N\sim\IP_{\thetatrue}^N$, any $\delta>0$, and $\alpha>0$, we
we use the proxy variances $(s_i^2)_{i\leq N}$ and the proxy forward map $\calGtilde$ to define 
a Tikhonov-regularized functional
\begin{equation*}
    \rmJtilde_{\delta,N,\ops}\colon\Theta\to[-\infty,0),\qquad
    \rmJtilde_{\delta,N,\ops}[\theta]
    \coloneqq
    -\frac{1}{2N}\sum_{i\nset{N}} s_i^{-2}\,\Vnorm{Y_i-\calGtilde_\theta(Z_i)}^2
    -\frac{\delta^2}{2}\,\norm{\theta}_{\scrH}^2,
\end{equation*}
where $\scrH\subseteq\IL^2(\calM,W)$ is a separable Hilbert space.
We call any element $\thetahat\in\Thetadiamond$ a maximizer of $\rmJtilde_{\delta,N,\ops}$ over $\Thetadiamond\subseteq\Theta$ if it satisfies
\begin{equation}\label{eq:maximizerproperty}
    \rmJtilde_{\delta,N,\ops}[\thetahat]
    = \sup_{\theta\in\Thetadiamond}\rmJtilde_{\delta,N,\ops}[\theta].
\end{equation}


\begin{proposition}[Existence of $\thetahat$]\label{propo:existenceofthethat}
    \mbox{}
    Let \cref{ass:errorBernstein} be satisfied. Let $\thetatrue\in\Theta$ be fixed. Assume $D_N\sim\IP_{\thetatrue}^N$ for a fixed sample size $N\in\IN$. Assume that the surrogate operator $\calGtilde$ satisfies \cref{cond:fm:boundglobal} and \cref{cond:fm:Lipinfglobal} of \cref{ass:operatorII} as well as \cref{cond:MM:III} of \cref{cond:misspec-model}. Let $\scrH\subseteq\IL^2(\calM,W)$ satisfy \cref{cond:MAP:scrH} with $\alpha>\eta + \frac{d}{2}$.
    Let either 
    \begin{equation}\label{eq:existence:embedding}
        \Thetadiamond\subseteq\scrH,\quad\text{if $\gamma_B = 0$}\quad\text{or}\quad\Thetadiamond\subseteq\scrR(M)\cap\scrH\quad\text{for some $M\in(0,\infty)$, if $\gamma_B \neq 0$,}
    \end{equation}
    be weakly closed in $\scrH$. Then, for all $\delta\in(0,\infty)$ and proxy variances $(s_i^2)_{i=1}^N\subseteq(0,\infty)$, such that $\min_{i\nset{N}}s_i^2 \geq s_0^2\in(0,\infty)$, almost surely under $\IP_{\thetatrue}^N$, there exists a maximizer $\thetahat$ of $\tikfunc$ over $\Thetadiamond$, that is 
    \begin{equation*}
            \tikfunc[\thetahat] = \sup_{\theta\in\Thetadiamond}\tikfunc[\theta] \quad \IP_{\thetatrue}^N \;a.s.
    \end{equation*}
\end{proposition}

\begin{proof}[Proof of \cref{propo:existenceofthethat}]
The proof follows the lines of \cite{Siebel_2025}, where existence is shown for the homoscedastic and correctly specified model, i.e. when $s_i^2 = \sigma_i^2 = \sigma^2\in(0,\infty)$ for all $i\leq N$.
As the proof does not change substantially, the details are left for the interested reader.

\end{proof}

\begin{example}\label{ex:thetadiamond}
    \mbox{}
    In \cref{propo:existenceofthethat}, it is required to chose $\Thetadiamond$ accordingly, such that it is weakly closed in $\scrH$. Standard arguments show that typical choices as
    \begin{equation*}
        \Thetadiamond\coloneqq\scrH, \quad\text{ if }\quad \gamma_B = 0,
    \end{equation*}
    or 
    \begin{equation*}
        \Thetadiamond\coloneqq\scrH\cap\scrR(M)\quad\text{and}\quad\Thetadiamond\coloneqq\scrH(M) \quad\text{ for some }\quad M\in(0,\infty),\quad\text{ if }\quad \gamma_B > 0
    \end{equation*}
    satisfy the requirements, if $\scrH\hookrightarrow\scrR$.

\end{example}

\begin{theorem}\label{thm:LS_estimator}
    Grant \cref{ass:errorBernstein} and assume that the surrogate operator $\calGtilde$ satisfies \cref{ass:operatorII} as well as \cref{cond:MM:III} of \cref{cond:misspec-model}. Let $N\in\IN$, $\thetatrue\in\Theta$ be fixed. Let $D_N\sim\IP_{\thetatrue}^N$.  Let $\scrH\subseteq\IL^2(\calM,W)$ satisfy \cref{cond:MAP:scrH} with $\alpha>\max\brK{\frac{d}{2}\gamma_2-\kappa,\eta + d\lor \frac{d}{2}(1+\gamma_\infty)}$.
    Let either 
    \begin{equation*}
        \Thetadiamond\subseteq\scrH,\quad\text{if $\gamma_B = 0$}\quad\text{or}\quad\Thetadiamond\subseteq\scrR(M)\cap\scrH\quad\text{for some $M\in(0,\infty)$, if $\gamma_B \neq 0$,}
    \end{equation*}
    be weakly closed in $\scrH$.
    Assume that 
    the proxy variances $(s_i^2)_{i=1}^N\subseteq(0,\infty)$ satisfy $0< s_0^2\coloneqq\min_{i\nset{N}}s_i^2\leq\max_{i\nset{N}}s_i^2\eqqcolon s_\infty^2$
     Given any $\Bar{\const}\in(0,\infty)$, we define the family of symbols $$\texttt{S}\coloneqq\brK{\Bar{c},\alpha,\gamma_2,\gamma_\infty,\gamma_B,\kappa,\eta,d,d_W,M,\Const{\calGtilde,\rmB},\ConstLiptwo',\ConstLipinfty',\sigma_0,\sigma_\infty,s_0,s_\infty,\rmB}.$$  Then, we can choose $\Const{Rate}=\Const{Rate}(\texttt{S})\in(0,\infty)$ sufficiently large, such that for all $\delta\in(0,\infty)$ and $R\in(0,\infty)$
        with $R\geq\delta \geq N^{-\frac{1}{2}}$ that fulfill 
        \begin{equation}\label{eq:ratecondI}
            \delta^{-1-\frac{d}{2(\alpha+\kappa)}}\lesssim N^\frac{1}{2},
        \end{equation}
        any maximizer $\thetahat$ of $\tikfunc$ over $\Thetadiamond$ satisfies
        \begin{equation*}
            \IP_{\thetatrue}^N\br{\frakdtilde_{\delta,\ops}^2(\thetahat,\thetatrue) \geq \Const{Rate}\br{\frakdtilde_{\delta,\ops}^2(\thetadiamond,\thetatrue)+R^2}} \leq \frac{3}{\ln(2)}\exp\br{-\Bar{\const}NR^2}
        \end{equation*}
        for any $\thetadiamond\in\Thetadiamond$.
        Here, $\frakdtilde_{\delta,\ops}^2$ is defined via 
        \begin{equation*}
            \forall \theta_1\in\Theta\cap\scrH,\theta_2\in\Theta:\; \frakdtilde_{\delta,\ops}^2(\theta_1,\theta_2)\coloneqq \sbar^{-2}\norm{\calGtilde(\theta_1)-\calG(\theta_2)}_{\IL_\zeta^2(\calZ,V)}^2+\delta^2\norm{\theta_1}_{\scrH}^2.
        \end{equation*}
\end{theorem}



\subsection{Consistency and Tests}  

\begin{corollary}\label{co:GeneralConsitency}
     Grant \cref{ass:errorBernstein} and assume that the surrogate operator $\calGtilde$ satisfies \cref{ass:operatorII} as well as \cref{cond:MM:III} of \cref{cond:misspec-model}. Let $\scrH\subseteq\IL^2(\calM,W)$ satisfy \cref{cond:MAP:scrH} with $\alpha>\max\brK{\frac{d}{2}\gamma_2-\kappa,\eta + d\lor \frac{d}{2}(1+\gamma_\infty)}$.
     Let either 
    \begin{equation*}
        \Thetadiamond = \scrH,\quad\text{if $\gamma_B = 0$}\quad\text{or}\quad\Thetadiamond = \scrR(M)\cap\scrH\quad\text{for some $M\in(0,\infty)$, if $\gamma_B \neq 0$.}
    \end{equation*}
     Let $\delta\in(0,\infty)$, such that $ \Bar{\delta}\coloneqq(\delta\lor\deltamap)\geq N^{-\frac{1}{2}}$ and $ \Bar{\delta}^{-1-\frac{d}{2(\alpha+\kappa)}}\lesssim N^{\frac{1}{2}}$.
     Given any $\Bar{\const},M\in(0,\infty)$, we can choose $\Const{Rate}=\Const{Rate}(\texttt{S})\in(0,\infty)$ and $m = m(\ConstLiptwo',M,\gamma_2)$ sufficiently large, such that 
     any maximizer $\thetahat$ of $\rmJtilde_{\delta,N,\boldsymbol{1}}$ over $\Thetadiamond$ satisfies
     \begin{equation*}
         \sup_{\thetatrue\in\Theta_{\Bar{\delta}}(M)}\IP_{\thetatrue}^N\br{\norm{\calGtilde(\thetahat)-\calG(\thetatrue)}_{\IL_\zeta^2(\calZ,V)}^2+\delta^2\norm{\thetahat}_{\scrH}^2 \geq \Const{Rate} m^2  \Bar{\delta}^2} \lesssim \exp\br{-\Bar{c}m^2 N \Bar{\delta}^2}.
     \end{equation*}
     Here, $\Theta_{ \Bar{\delta}}(M)$ denotes 
     \begin{equation*}
         \Theta_{ \Bar{\delta}}(M) \coloneqq \brK{\theta \in \scrR :\; \theta=\theta_1+\theta_2,\; \norm{\theta_1}_{\br{H^\kappa(\calM,W)}^*} \leq M \Bar{\delta} , \;\norm{\theta_2}_\scrH \leq M,\; \norm{\theta}_\scrR\leq M}.
     \end{equation*}
\end{corollary}

\begin{remark}\label{rem:consistency}
    Note, if $\delta$ is chosen in an admissible way that $\Bar{\delta}N\to\infty$, we obtain consistency.
    This is particularly in the situation where $\delta = \delta_N = N^{-\frac{\alpha+\kappa}{2(\alpha+\kappa)+d}}$ and $\deltamap \leq \delta_N$, such that $\Bar{\delta} = \delta_N$.
\end{remark}

\begin{proof}[Proof of \cref{co:GeneralConsitency}]
    Fix $\thetatrue\in\Theta_{\delta}(M)$ and set $\boldsymbol{1}\coloneqq (1)_{i=1}^N \eqqcolon (s_i^2)_{i=1}^N$.
    Then, following \cref{propo:existenceofthethat} and \cref{ex:thetadiamond}, there exists a maximizer $\thetahat$
    of $\rmJtilde_{\delta,N,\boldsymbol{1}}$ over $\Thetadiamond$ at least $\IP_{\thetatrue}^N$-almost surely.
    As $\calGtilde$ satisfies \cref{cond:MM:III} of \cref{cond:misspec-model}, we have
    for all $\thetadiamond\in\Thetadiamond$
    \begin{align*}
        \norm{\calGtilde(\thetadiamond)-\calG(\thetatrue)}_{\Lfor}^2 & \leq 2 \norm{\calGtilde(\thetadiamond)-\calGtilde(\thetatrue)}_{\Lfor}^2+2 \norm{\calGtilde(\thetatrue)-\calG(\thetatrue)}_{\Lfor}^2.
    \end{align*}
    As $\thetatrue\in\Theta_{\delta}(M)\subseteq\scrR(M)$, we have 
    \begin{equation*}
         \norm{\calGtilde(\thetatrue)-\calG(\thetatrue)}_{\Lfor}^2 \leq \const_1(M)\times\deltamap^2 \leq \const_1(M)\times  \Bar{\delta}^2.
    \end{equation*}
   As $\calGtilde$ satisfies \cref{ass:operatorII}, we further have
   \begin{align*}
        \norm{\calGtilde(\thetadiamond)-\calGtilde(\thetatrue)}_{\Lfor} \leq \ConstLiptwo'\br{1+\norm{\thetadiamond}_{\scrR}^{\gamma_2}\lor \norm{\thetatrue}_{\scrR}^{\gamma_2}}\norm{\thetadiamond-\thetatrue}_{(H^\kappa(\calM,W))^*}.
   \end{align*}
   Now as $\thetatrue\in\Theta_{\delta}(M)$, we can choose $\theta_{0,1},\theta_{0,2}$, such that $\theta_0 =\theta_{0,1}+\theta_{0,2}$
   and $\norm{\theta_{0,1}}_{\br{H^\kappa(\calM,W)}^*} \leq M \Bar{\delta}$ and $\norm{\theta_{0,2}}_\scrH \leq M$.
   Now we can choose $\thetadiamond = \theta_{0,2}\in\scrH(M)\subseteq \scrH\cap\scrR(M)$, such that the last display reads as
   \begin{equation*}
        \norm{\calGtilde(\thetadiamond)-\calGtilde(\thetatrue)}_{\Lfor}^2 \leq c_2(\ConstLiptwo')\times \br{1+M^{2\gamma_2}} \Bar{\delta}^2 M^2.
   \end{equation*}
   Overall we conclude 
   \begin{align*}
         \norm{\calGtilde(\thetadiamond)-\calG(\thetatrue)}_{\Lfor}^2+\delta^2\norm{\thetadiamond}_{\scrH}\leq  \const_3(\ConstLiptwo',M,\gamma_2)\times \Bar{\delta}^2,
   \end{align*}
   where the right-hand-side does not depend on the choice of $\thetatrue$.
   Thus, choosing $R = \frac{m}{\sqrt{2}} \times \Bar{\delta}$ with $m^2\geq 2\lor 2\const_3(\ConstLiptwo',M,\gamma_2)$,
   the claim follows by a direct consequence of \cref{thm:LS_estimator}.
\end{proof}

\begin{corollary}[Tests]\label{co:existence:test}
    Let $N\in\IN$. Assume \cref{ass:errorBernstein}. Assume that the true operator $\calG$ satisfies \cref{ass:operatorI}. Let $\thetatrue\in\scrH$ be fixed with $\alpha>\eta+d$. Let $D_N\sim\IP_{\thetatrue}^N$. Given $\Bar{c}\in(0,\infty)$, there exists a test (indicator function) $\Psi_N = \Psi_N(D_N)$, such that 
    \begin{equation*}
        \lim_{N\to\infty}\IE_{\thetatrue}^N\brE{\Psi_N} = 0 \quad \text{ and }\quad  \sup_{\substack{\theta\in\Theta_N(M):\;\norm{\calG(\theta)-\calG(\thetatrue)}{\IL_\zeta^2(\calZ,V)}\geq L\delta_N}}\brK{\IE_\theta^N[1-\Psi_N] } \lesssim \exp\br{-\Bar{\const}N\delta_N^2}
        \end{equation*}
    for all $L = L(\texttt{S}),M=M(\thetatrue)>0$ and $N$ sufficiently large.
\end{corollary}

\begin{proof}[Proof of \cref{co:existence:test}]
    The proof follows with similar arguments as presented in \cref{co:GeneralConsitency} with the simplification $\calGtilde = \calG$ and hence $\deltamap = 0$. Given $M\in(0,\infty)$ consider $\Thetadiamond\coloneqq \scrH(M)$.
    From the proof of \cref{thm:LS_estimator} it becomes evident, that we can weaken the global assumptions on $\calG$ as formulated in \cref{ass:operatorII} to the local ones as stated in \cref{ass:operatorI}, if the $\Thetadiamond$ is a bounded ball in $\scrH$.
    Choosing $\delta = \delta_N$ as in \cref{rem:consistency} and utilizing \cref{propo:existenceofthethat} and \cref{ex:thetadiamond}, there exists a maximizer $\thetahat$ of $\tikfunc$ over $\scrH(M)$ at least $\IP_{\thetatrue}^N$-almost surely. We then define the test statistic
    \begin{equation*}
        \hat{S}_N \coloneqq \norm{\calG(\thetahat)-\calG(\thetatrue)}_{\IL^2_\zeta(\calZ,V)}^2 + \delta_N^2\norm{\thetahat}_{\scrH}^2
    \end{equation*}
    and a corresponding test statistic $\Psi_N = \Psi_N(D_N)$ by 
    \begin{equation*}
        \Psi_N\coloneqq\mathds{1}_{\brK{\hat{S}_N \geq \const{1}\delta_N^2}}
    \end{equation*}
    for some constant $c_1\in(0,\infty)$ choosen later. Applying \cref{thm:LS_estimator} with $R=\delta_N$, we get 
    \begin{equation*}
        \IP_{\thetatrue}^N\br{\frakd_{\delta_N}^2(\thetahat,\thetatrue) \geq \Const{Rate}\br{\frakd_{\delta_N}^2(\thetadiamond,\thetatrue)+\delta_N^2}} \lesssim\exp\br{-\Bar{\const}N\delta_N^2}
    \end{equation*}
    for any $\thetadiamond\in\scrH(M)$ and $\Const{rate} = \Const{rate}(\texttt{S})\in(0,\infty)$ sufficiently large.
    Here, $\frakd_{\delta_N}^2$ is defined via 
        \begin{equation*}
            \forall \theta_1\in\Theta\cap\scrH,\theta_2\in\Theta:\; \frakd_{\delta_N}^2(\theta_1,\theta_2)\coloneqq \norm{\calG(\theta_1)-\calG(\theta_2)}_{\IL_\zeta^2(\calZ,V)}^2+\delta_N^2\norm{\theta_1}_{\scrH}^2.
        \end{equation*}
    With $M>\norm{\thetatrue}_{\scrH}$, such that $\thetatrue\in\Thetadiamond$, we can choose $\thetadiamond = \thetatrue$ and obtain 
    \begin{equation*}
        \Const{Rate}\br{\frakd_{\delta_N}^2(\thetadiamond,\thetatrue)+\delta_N^2} \leq \Const{Rate}\br{\delta_N^2\norm{\thetatrue}_{\scrH}^2 + \delta_N^2} \leq \Const{rate}'(\texttt{S},\norm{\thetatrue}_{\scrH}^2)\times\delta_N^2.
    \end{equation*}
    Thus,
    \begin{equation*}
        \IE_{\thetatrue}^N\brE{\Psi_N} = \IP_{\thetatrue}^N\br{\hat{S}_N\geq \Const{rate}'(\texttt{S},\norm{\thetatrue}_{\scrH}^2)\times\delta_N^2} \lesssim \exp\br{-\Bar{c}N\delta_N^2},
    \end{equation*}
    which controls the type-I-error. For the type-II-error, let $\theta\in\Theta_N(M)$, such that $\norm{\calG(\theta)-\calG(\thetatrue)}{\IL_\zeta^2(\calZ,V)}\geq L\delta_N$ be arbitrary but fixed. We then have 
        \begin{equation*}
            \norm{\calG(\theta)-\calG(\thetatrue)}{\IL_\zeta^2(\calZ,V)}^2 \leq 2\norm{\calG(\thetahat)-\calG(\theta)}{\Lfor}^2+2\norm{\calG(\thetahat)-\calG(\thetatrue)}{\Lfor}^2,
        \end{equation*}
        where we have applied the triangle inequality and $(a+b)^2 \leq 2a^2+2b^2$ for any reals $a,b\in(0,\infty)$.
        Thus, we have 
        \begin{align*}
            \frakd_{\delta_N}^2(\thetahat,\thetatrue) & = \norm{\calG(\thetahat)-\calG(\thetatrue)}{\Lfor}^2+\delta_N^2\norm{\thetahat}{\scrH}^2\\
            & \geq \frac{1}{2}\norm{\calG(\theta)-\calG(\thetatrue)}{\Lfor}^2 -\norm{\calG(\thetahat)-\calG(\theta)}^2 +\delta_N^2\norm{\thetahat}_{\scrH}^2\\
            & \geq \frac{1}{2}\norm{\calG(\theta)-\calG(\thetatrue)}{\Lfor}^2 -\frakd_{\delta_N}^2(\thetahat,\theta).
        \end{align*}
        Thus, we have 
        \begin{align*}
            \IE_{\theta}^N\brE{1-\Psi_N} & =\IP_\theta^N\br{\frakd_{\delta_N}^2(\thetahat,\thetatrue)\leq c_1\delta_N^2} \\& \leq\IP_\theta^N\br{\frakd_{\delta_N}^2(\thetahat,\theta)\geq \frac{1}{2}\norm{\calG(\theta)-\calG(\thetatrue)}{\Lfor}^2-c_1\delta_N^2}\\
            & \leq \IP_\theta^N\br{\frakd_{\delta_N}^2(\thetahat,\theta)\geq \br{\frac{1}{2}L^2-c_1}\delta_N^2}\\
            & \leq\IP_\theta^N\br{\frakd_{\delta_N}^2(\thetahat,\theta)\geq \frac{1}{4}L^2\delta_N^2}
        \end{align*}
        where we have used that $L\geq 2\sqrt{c_1}$. To upper bound the last probability, let $\thetahat$ be any maximizer of $\tikfunc$ over $\Thetadiamond = \scrH(M)$ under $\IP_\theta^N$-probability with $\theta\in\Theta_N(M)$. Thus, by \cref{thm:LS_estimator}, for every $\Bar{c}\in(0,\infty)$, we can find $\Const{rate} = \Const{rate}(\texttt{S})$ sufficiently large, such that 
        \begin{equation*}
            \IP_{\theta}^N\br{\frakd_{\delta_N}^2(\thetahat,\theta) \geq \Const{rate}\br{\frakd_{\delta_N}^2(\thetadiamond,\theta)+\delta_N^2}} \leq 6\times\exp\br{-\Bar{\const}N\delta_N^2}.
        \end{equation*}
        As $\theta\in\Theta_N(M)$ there exists a decomposition $\theta  = \theta_1+\theta_2$, such that $\norm{\theta_1}_{(H^\kappa(\calM,W))^*}\leq M\delta_N$ and $\norm{\theta_2}_{\scrH} \leq M$ as well as $\norm{\theta}_{\scrR}\leq M$. Thus, choosing $\thetadiamond = \theta_2$, we obtain with \cref{cond:fm:Liptwo}
        \begin{align*}
            \frakd_{\delta_N}^2(\thetadiamond,\theta)  \lesssim_M\norm{\theta_1}_{(H^\kappa(\calM,W))^*}^2 \lesssim_M \delta_N^2,
        \end{align*}
        where $\lesssim_M$ means that the multiplicative constant depends on $M$.
        Thus, if we choose $L$ sufficiently large, depending on $M$, we obtain
        \begin{equation*}
             \IE_{\theta}^N\brE{1-\Psi_N} \lesssim \times \exp\br{-\Bar{\const}N\delta_N^2}
        \end{equation*}
        uniformly over $\theta\in\Theta_N(M)$, which shows the claim. 
\end{proof}

\subsection{\texorpdfstring
  {Proof of \cref{thm:LS_estimator}}
  {Proof of Theorem X}
}\label{sec:ProofMainIneq}

In this section we are going to proof the general concentration inequality, \cref{thm:LS_estimator}. To that end we need some notations and key arguments from M-estimation theory we adapt to the present setting.
\begin{notation}
    \mbox{}
    \begin{enumerate}
        \item For $N\in\IN$, let $D_N \sim (Y_i,Z_i)_{i\nset{N}} \subseteq (V\times\calZ)^N$ be a finite family of random variables with law $\IP_{\thetatrue}^N$, where $\thetatrue\in\Theta$. We denote by $\IP_{\thetatrue}^{(i)}$ the law of any datum $(Y_i,Z_i)$, $i\nset{N}$.  
        For any $(y,z)\in V\times \calZ$, let $\delta_{(y,z)}(\cdot)$ denote the Dirac measure on $(V\times\calZ,\scrB_{V\times\calZ})$.  
        For each $i\nset{N}$, we define
        \[
            \forall B\in\scrB_{V\times\calZ}:\quad
            \empmeasure(B)\coloneqq\delta_{(Y_i,Z_i)}(B).
        \]
        Thus, for any measurable function $h:V\times\calZ\to\IR$ and each $i\nset{N}$,
        \[
            \int_{V\times\calZ} h(y,z)\,\rmd\empmeasure(y,z)
            = h(Y_i,Z_i).
        \]

        \item For any $\theta\in\Theta$, we define
        \[
            \forall (y,z)\in V\times\calZ:\quad
            \Tilde{L}_\theta(y,z)\coloneqq
            \exp\br{\Vnorm{y-\calGtilde_\theta(z)}^2}.
        \]
    \end{enumerate}
\end{notation}

\begin{lemma}[Key Argument]\label{lem:KeyInequality}
    Let $\Thetadiamond\subseteq\Theta\cap\scrH$. Given $\thetatrue\in\Theta$ and $D_N\sim\IP_{\thetatrue}^N$, 
    we define for all $\theta\in\Thetadiamond$ the empirical processes 
    \begin{equation*}
        \widetilde{\calW}_{N}(\theta) \coloneqq \frac{1}{2N}\sum_{i\nset{N}}s_i^{-2}\int_{V\times\calZ}\log\br{\frac{\Tilde{L}_{\thetadiamond}(y_i,z_i)}{\Tilde{L}_{\theta}(y_i,z_i)}}\rmd(\empmeasure-\IP_{\thetatrue}^{(i)})(y_i,z_i)
    \end{equation*}
    with fixed $\thetadiamond\in\Thetadiamond$.
    Given any fixed $\delta\in(0,\infty)$ and $(s_i^2)_{i=1}^N\subseteq(0,\infty)$, let $\thetahat$ any maximizer of $\tikfunc$ over $\Thetadiamond$, i.e. that satisfies
    \begin{equation*}
        \tikfunc[\thetahat] = \sup_{\theta\in\Thetadiamond}\tikfunc[\theta].
    \end{equation*}
    We then have
    \begin{equation*}
        2\widetilde{\calW}_N(\thetahat) + \distd(\thetadiamond,\thetatrue) \geq \distd(\thetahat,\thetatrue),
    \end{equation*}
    where the inequality is understood $\IP_{\thetatrue}^N$-almost surely.
\end{lemma}

\begin{proof}[Proof of \cref{lem:KeyInequality}]
In the following, we fix an arbitrary maximizer $\thetahat$ of $\tikfunc$ under $\IP_{\thetatrue}^N$-probability. Thus, all (in-)equalities hold $\IP_{\thetatrue}^N$-almost surely.
By definition of $\thetahat$, we have
\[
    \tikfunc[\thetahat] \geq \tikfunc[\thetadiamond],
\]
and hence
\[
    \frac{1}{2N}\sum_{i\nset{N}} s_i^{-2}
    \left\{
        \Vnorm{Y_i-\calGtilde(\thetadiamond)(Z_i)}^2
        -\Vnorm{Y_i-\calGtilde(\thetahat)(Z_i)}^2
    \right\}
    +\frac{\delta^2}{2}\norm{\thetadiamond}_{\scrH}^2
    \geq
    \frac{\delta^2}{2}\norm{\thetahat}_{\scrH}^2.
\]
This is equivalent to
\[
    \frac{1}{2N}\sum_{i\nset{N}} s_i^{-2}
    \log\br{\frac{\Tilde{L}_{\thetadiamond}(Y_i,Z_i)}{\Tilde{L}_{\thetahat}(Y_i,Z_i)}}
    +\frac{\delta^2}{2}\norm{\thetadiamond}_{\scrH}^2
    \geq
    \frac{\delta^2}{2}\norm{\thetahat}_{\scrH}^2.
\]
Subtracting on both sides
\[
    \frac{1}{2N}\sum_{i\nset{N}} s_i^{-2}
    \int_{V\times\calZ}
    \log\br{\frac{\Tilde{L}_{\thetadiamond}(y_i,z_i)}{\Tilde{L}_{\thetahat}(y_i,z_i)}}
    \,\rmd\IP_{\thetatrue}^{(i)}(y_i,z_i),
\]
the previous inequality becomes
\[
    \widetilde{\calW}_N(\thetahat)
    +\frac{\delta^2}{2}\norm{\thetadiamond}_{\scrH}^2
    \geq
    \frac{\delta^2}{2}\norm{\thetahat}_{\scrH}^2
    - \frac{1}{2N}\sum_{i\nset{N}} s_i^{-2}
    \int_{V\times\calZ}
    \log\br{\frac{\Tilde{L}_{\thetadiamond}(y_i,z_i)}{\Tilde{L}_{\thetahat}(y_i,z_i)}}
    \,\rmd\IP_{\thetatrue}^{(i)}(y_i,z_i).
\]
Since $Z_i\sim\zeta$ is stochastically independent of $\varepsilon_i$, we observe that
\[
    \IP_{\thetatrue}^{(i)}
    =
    \br{\zeta\otimes \IP^{\varepsilon_i}}\circ T_{\thetatrue}^{-1},
    \quad
    T_{\thetatrue}\colon
    \calZ\times V \ni (z,e)
    \mapsto
    \br{z,\calG(\thetatrue)(z)+e}
    \in \calZ\times V,
\]
such that
\begin{align*}\label{eq:WNExp_Decomposition}
    &\int_{V\times\calZ}
    \log\br{\frac{\Tilde{L}_{\thetadiamond}(y_i,z_i)}{\Tilde{L}_{\thetahat}(y_i,z_i)}}
    \,\rmd\IP_{\thetatrue}^{(i)}(y_i,z_i) \\
    &\hspace{1cm}=
    \int_{\calZ\times V}
    \Vnorm{\calG(\thetatrue)(z_i)-\calGtilde(\thetadiamond)(z_i)}^2
    -\Vnorm{\calG(\thetatrue)(z_i)-\calGtilde(\thetahat)(z_i)}^2\rmd\br{\zeta\otimes\IP^{\varepsilon_i}}(z_i,e_i)\\
    & \hspace{2cm}+
    2\int_{\calZ\times V}\Vsprod{e_i}{\calGtilde(\thetahat)(z_i)-\calGtilde(\thetadiamond)(z_i)}
    \,\rmd\br{\zeta\otimes\IP^{\varepsilon_i}}(z_i,e_i).
\end{align*}
By independence of $Z_i$ and $\varepsilon_i$, and since $\varepsilon_i$ is a centered $V$-valued random variable, we obtain
\[
    \int_{V\times\calZ}
    \log\br{\frac{\Tilde{L}_{\thetadiamond}(y_i,z_i)}{\Tilde{L}_{\thetahat}(y_i,z_i)}}
    \,\rmd\IP_{\thetatrue}^{(i)}(y_i,z_i)
    =
    \norm{\calG(\thetatrue)-\calGtilde(\thetadiamond)}_{\IL^2_\zeta(\calZ,V)}^2
    -
    \norm{\calG(\thetatrue)-\calGtilde(\thetahat)}_{\IL^2_\zeta(\calZ,V)}^2.
\]
Combining these computations yields
\[
    \widetilde{\calW}_N(\thetahat)
    +\frac{1}{2}\sbar^{-2}
    \norm{\calG(\thetatrue)-\calGtilde(\thetadiamond)}_{\IL^2_\zeta(\calZ,V)}^2
    +\frac{\delta^2}{2}\norm{\thetadiamond}_{\scrH}^2
    \geq
    \frac{1}{2}\sbar^{-2}
    \norm{\calG(\thetatrue)-\calGtilde(\thetahat)}_{\IL^2_\zeta(\calZ,V)}^2
    +\frac{\delta^2}{2}\norm{\thetahat}_{\scrH}^2.
\]
This proves the claim.
\end{proof}

\begin{lemma}[M-Estimation]\label{lem:M_estimation}
    Define for each $\thetatrue\in\Theta$, $\thetadiamond\in\Thetadiamond$ and real numbers $\delta,R\in(0,\infty)$ the event
    \begin{equation*}
        \Xi_{\delta,R}\coloneqq\left\{\frakdtilde_{\delta,\ops}^2(\thetahat,\thetatrue) \geq c_1\br{\frakdtilde_{\delta,\ops}^2(\thetadiamond,\thetatrue)+R^2}\right\},
    \end{equation*}
    where $c_1\in\IR_{\geq 2}$ is a fixed positive constant. 
    On $\Xi_{\delta,R}$, we have 
    \begin{enumerate}[label = \roman*)]
        \item $\sbar^{-2}\norm{\calGtilde(\thetahat)-\calGtilde(\thetadiamond)}_{\IL_\zeta^2(\calZ,V)}^2 + \delta^2\norm{\thetahat}_{\scrH}^2\geq c_2R^2$ for some constant $c_2\geq 1$.
        \item  $\frakdtilde_{\delta,\ops}^2(\thetahat,\thetatrue)-\frakdtilde_{\delta,\ops}^2(\thetadiamond,\thetatrue)\geq \frac{1}{6}\br{\sbar^{-2}\norm{\calGtilde(\thetahat)-\calGtilde(\thetadiamond)}_{\IL_\zeta^2(\calZ,V)}^2 + \delta^2\norm{\thetahat}_{\scrH}^2}$.
    \end{enumerate}
\end{lemma}

\begin{proof}[Proof of \cref{lem:M_estimation}]
    The proof follows from standard computations in M-estimation theory as they can be found in \cite{Geer_2001}. Hence, the derivation is omitted.
\end{proof}

We are now able to proof the general concentration inequality presented in \cref{thm:LS_estimator}.

\begin{proof}[Proof of \cref{thm:LS_estimator}]
    Noting that the event in question is the event $\Xi_{\delta,R}$ as presented in \cref{lem:KeyInequality}, \cref{lem:KeyInequality} implies that it suffices to establish the correct exponential decay of the probability
    \begin{equation*}
        \IP_{\thetatrue}^N\br{\brK{2\widetilde{\calW}_N(\thetahat) +\frakdtilde_{\delta,\ops}^2(\thetadiamond,\thetatrue) \geq\frakdtilde_{\delta,\ops}^2(\thetahat,\thetatrue)}\cap \Xi_{\delta,R} }.
    \end{equation*}
    Defining the shorthand notation
    \begin{equation*}
        \frakDtilde^2(\thetahat,\thetadiamond) \coloneqq \sbar^{-2}\norm{\calGtilde(\thetahat)-\calGtilde(\thetadiamond)}_{\IL_\zeta^2(\calZ,V)}^2 + \delta^2\norm{\thetahat}_{\scrH}^2,
    \end{equation*}
    according to \cref{lem:M_estimation}, the last probability can be upper bounded by
    \begin{equation}\label{eq:LeastSquaresProof_1}
        \IP_{\thetatrue}^N\br{|\widetilde{\calW}_N(\thetahat)|\geq \frac{1}{12}\frakDtilde^2(\thetahat,\thetadiamond) ,\;\frakDtilde^2(\thetahat,\thetadiamond) \geq \frac{1}{2}c_1R^2}.
    \end{equation}
    To further upper bound the last probability, we apply a peeling device (see \cite{Geer_2000}). For this, we define slices of the set $\Thetadiamond$, i.e., for each $l\in\IN$ we set
    \begin{equation*}
        \Thetadiamond_l \coloneqq \brK{\theta\in\Thetadiamond:\,\frakDtilde^2(\theta,\thetadiamond)\in\left[\frac{1}{2}c_1R^2\cdot2^{2l-2},\frac{1}{2}c_1R^2\cdot 2^{2l}\right)}
    \end{equation*}
    and observe that
    \begin{equation*}
        \brK{\theta\in\Theta:\,\frakDtilde^2(\theta,\thetadiamond)\geq \frac{1}{2}c_1R^2} = \bigcup_{l\in\IN}\Thetadiamond_l.
    \end{equation*}
    Moreover, by definition we have for each $l\in\IN$
    \begin{equation*}
        \Thetadiamond_l\subseteq H^\alpha\br{\calM,W,2^l\frac{\sqrt{c_1}R}{\sqrt{2}\delta}}
        \quad\text{and for each $\theta\in\Thetadiamond_l$}\quad
        \norm{\calGtilde(\theta)-\calGtilde(\thetadiamond)}_{\IL_\zeta^2(\calZ,V)}^2 \leq \frac{2^{2l}c_1R^2}{2\sbar^{-2}},
    \end{equation*}
    where we have used \cref{cond:MAP:scrH}.
    Then, the peeling device yields that \cref{eq:LeastSquaresProof_1} can be upper bounded by
    \begin{equation}\label{eq:LeastSquaresProof_2}
        \sum_{l\in\IN}\IP_{\thetatrue}^N\br{\sup_{\theta\in\Thetadiamond_l}|\widetilde{\calW}_N(\theta)| \geq \frac{2^{2l}c_1R^2}{192}}.
    \end{equation}
    Now, analogously to the proof of \cref{lem:KeyInequality}, we decompose $\widetilde{\calW}_N(\theta)$ into two centered empirical processes.
    Indeed, under $\IP_{\thetatrue}^N$, we have for each $i\nset{N}$
    \begin{align*}
        \log\br{\frac{\Tilde{L}_{\thetadiamond}(Y_i,Z_i)}{\Tilde{L}_{\theta}(Y_i,Z_i)}}
        &= \Vnorm{\calG(\thetatrue)(Z_i)-\calGtilde(\thetadiamond)(Z_i)}^2-\Vnorm{\calG(\thetatrue)(Z_i)-\calGtilde_\theta(Z_i)}^2
        + 2\Vsprod{\calGtilde_\theta(Z_i)-\calGtilde(\thetadiamond)(Z_i)}{\varepsilon_i},
    \end{align*}
    such that under $\IP_{\thetatrue}^N$ we have
    \begin{align*}
        \widetilde{\calW}_{N}(\theta)
        &= \frac{1}{\sqrt{N}}T_{N,1}(\theta) + \frac{1}{\sqrt{N}}\br{T_{N,2}(\theta)-\IE\brE{T_{N,2}(\theta)}},
    \end{align*}
    where we have defined the empirical processes
    \begin{equation*}
        T_{N,1}(\theta) \coloneqq \frac{1}{\sqrt{N}}\sum_{i\nset{N}}s_i^{-2}\Vsprod{\varepsilon_i}{\calGtilde_\theta(Z_i)-\calGtilde(\thetadiamond)(Z_i)}
    \end{equation*}
    and
    \begin{equation*}
        T_{N,2}(\theta) \coloneqq \frac{1}{2\sqrt{N}}\sum_{i\nset{N}}s_i^{-2}\brK{\Vnorm{\calG(\thetatrue)(Z_i)-\calGtilde(\thetadiamond)(Z_i)}^2-\Vnorm{\calG(\thetatrue)(Z_i)-\calGtilde(\theta)(Z_i)}^2}.
    \end{equation*}
    Thus, \cref{eq:LeastSquaresProof_2} can be upper bounded by
    \begin{multline*}
        \sum_{l\in\IN}\IP\br{\sup_{\theta\in\Thetadiamond_l}|T_{N,1}(\theta)|\geq\sqrt{N}\frac{2^{2l}c_1R^2}{384}}
        +\sum_{l\in\IN}\IP\br{\sup_{\theta\in\Thetadiamond_l}|T_{N,2}(\theta)-\IE[T_{N,2}(\theta)]|\geq\sqrt{N}\frac{2^{2l}c_1R^2}{384}}.
    \end{multline*}
    We aim to control both remaining probabilities using \cref{lem:chaining}. Starting with $\br{T_{N,1}(\theta)}_{\theta\in\Theta}$, we define for each $l\in\IN$ the class of measurable functions
    \begin{equation*}
        \calH_{l}\coloneqq\brK{h_\theta\coloneqq\calGtilde_\theta-\calGtilde(\thetadiamond)\colon\calZ\to V:\;\theta\in\Thetadiamond_l}.
    \end{equation*}
    Now, with \cref{cond:fm:boundglobal}, we obtain
    \begin{equation*}
        \forall \theta\in\Thetadiamond:\;
        \norm{h_\theta}_\infty
        \leq \norm{\calGtilde_\theta-\calGtilde(\thetadiamond)}_\infty
        \leq 2\Const{\calGtilde,\rmB}\br{1+\norm{\thetadiamond}_{\scrR}^{\gamma_B}\lor\norm{\theta}_{\scrR}^{\gamma_B}}
        \leq 2\Const{\calGtilde,\rmB}\times\max\brK{2,1+M^{\gamma_B}}.
    \end{equation*}
    Thus,
    \begin{equation*}
        \sup_{\theta\in\Thetadiamond_l}\norm{h_\theta}_\infty \leq c_2\br{\Const{\calGtilde,\rmB},M,\gamma_B}\in(0,\infty).
    \end{equation*}
    Further, for any $Z\sim\zeta$, we have
    \begin{equation*}
        \sup_{\theta\in\Thetadiamond_l}\IE\brE{h_\theta(Z)^2}
        = \sup_{\theta\in\Thetadiamond_l}\norm{\calGtilde_\theta-\calGtilde(\thetadiamond)}_{\IL^2_\zeta(\calZ,V)}^2
        \leq \frac{2^{2l}c_1R^2}{2\sbar^{-2}}
        =: \rmv_l^2.
    \end{equation*}
    Thus, all conditions of \cref{lem:chaining} are satisfied, and we can apply the result to conclude that there exists a universal constant $L\in(0,\infty)$ such that for all $x\geq 1$,
    \begin{multline*}
        \IP\left(\sup_{\theta\in\Thetadiamond_l}|T_{N,1}(\theta)|\geq L\left(\sqrt{\frac{1}{N}\sum_{i\nset{N}}\frac{\sigma_i^2}{s_i^4}}\left(\calJ_2(\calH_{l})+\rmv_l\sqrt{x}\right)+\frac{\rmB}{s_0^2\sqrt{N}}\left(\calJ_\infty(\calH_{l})+2\Const{\calG,B} x\right)\right)\right)\\
        \leq 3\exp(-x),
    \end{multline*}
    We now show that for suitable $x\geq 1$ and $c_1\in(0,\infty)$ sufficently large,
    \begin{equation*}
        L\left(\sqrt{\frac{1}{N}\sum_{i\nset{N}}\frac{\sigma_i^2}{s_i^4}}\left(\calJ_2(\calH_{l})+\rmv_l\sqrt{x}\right)+\frac{\rmB}{s_0^2\sqrt{N}}\left(\calJ_\infty(\calH_{l})+2\Const{\calG,B} x\right)\right)
        \leq \sqrt{N}\frac{2^{2l}c_1R^2}{384}.
    \end{equation*}
    Therefore, we need to upper bound the entropy integrals $\calJ_2(\calH_{l})$ and $\calJ_\infty(\calH_{l})$, respectively.
    In \cite{Siebel_2025} it is shown that there exist finite constants $c_3,c_4 \in(0,\infty)$, both depending only on $\alpha,\kappa,d,\gamma_2,\gamma_\infty,\eta$ and $d_W$, such that for any $\rho\in(0,\infty)$
    \begin{equation*}
        \log N(\calH_l,d_2,\rho) \leq c_3\times\br{\frac{2^{l}\sqrt{c_1}R\Bar{m}_2}{\delta\rho}}^{\frac{d}{\alpha+\kappa}}
    \end{equation*}
    and
    \begin{equation*}
        \log N(\calH_l,d_\infty,\rho) \leq c_4\times\br{\frac{2^{l}\sqrt{c_1}R\Bar{m}_\infty}{\delta\rho}}^{\frac{d}{\alpha-\eta}}
    \end{equation*}
    with $\Bar{m}_i \coloneqq \max\brK{\ConstLiptwo,\ConstLipinfty}\times\br{1+\br{\frac{2^l\sqrt{c_1}R}{\delta}}^{\gamma_i}}$, $i\in\brK{2,\infty}$.
    With similar arguments as in \cite{Siebel_2025} we can choose $\alpha>\frac{d}{2}-\kappa$ and upper bound $\calJ_2(\calH_l)$ with 
    \begin{align*}
        \calJ_2(\calH_l)
        &= \int_{[0,2\rmv_l]}\sqrt{ \log N(\calH_l,d_2,\rho)}\,\rmd\rho\\
        & \leq c_3\times \sqrt{N}2^{2l}c_1R^2\times 2^{-l(1-\frac{\gamma_2d}{2(\alpha+\kappa)})} c_1^{-\frac{1}{2}(1+\frac{d\gamma_2}{2(\alpha+\kappa)})} \times  R^{-1+\frac{d\gamma_2}{2(\alpha+\kappa)}}\delta^{-\frac{d(\gamma_2+1)}{2(\alpha+\kappa)}} \times \sqrt{\sbar^{-2}}^{-1+\frac{d}{2(\alpha+\kappa)}}\times\sqrt{N}^{-1}\\
        & \leq c_3\times  \sqrt{N}2^{2l}c_1R^2\times c_1^{-\frac{1}{2}(1-\frac{d\gamma_2}{2(\alpha+\kappa)})}.
    \end{align*}
    In particular, the entropy integral can be upper bounded by multiples of $\sqrt{N}2^{2l}c_1R^2$ as small as desired by choosing $c_1$ sufficiently large.
    With similar (tedious) computations, we obtain overall
    \begin{equation*}
        \calJ_2(\calH_{l}) \leq c(c_1,\alpha,\gamma_2,\kappa,d,d_W,\ConstLiptwo,\ConstLipinfty)\times \sqrt{N}2^{2l}c_1R^2
    \end{equation*}
    as well as
    \begin{equation*}
        \calJ_\infty(\calH_{l}) \leq c(c_1,\alpha,\gamma_\infty,\gamma_B,\eta,\ConstLiptwo,\ConstLipinfty,\Const{\calGtilde,B},d,d_W,M)\times N2^{2l}c_1R^2,
    \end{equation*}
    where both constants can be made as small as desired by choosing
    $c_1\in(0,\infty)$ sufficiently large.
    Plugging in these computations and choosing $x=\Bar{c}N2^{2l}R^2$, we obtain,
    \begin{align*}
        \sqrt{\frac{1}{N}\sum_{i\nset{N}}\frac{\sigma_i^2}{s_i^4}}\left(\calJ_2(\calH_{l})+\rmv_l\sqrt{x}\right)
        & \leq c(c_1,\texttt{S})\times 2^{2l}c_1R^2\sqrt{N},
    \end{align*}
    as well as
    \begin{equation*}
        \frac{\rmB}{s_0^2\sqrt{N}}\left(\calJ_\infty(\calH_{l})+2\Const{\calG,B} x\right)
        \leq c(c_1,\texttt{S})\times 2^{2l}c_1R^2\sqrt{N},
    \end{equation*}
    where again both constants can be made as small as desired by choosing $c_1\in(0,\infty)$ sufficiently large.
    Altogether, we obtain
    \begin{equation*}
        L\left(\sqrt{\frac{1}{N}\sum_{i\nset{N}}\frac{\sigma_i^2}{s_i^4}}\left(\calJ_2(\calH_{l})+\rmv_l\sqrt{x}\right)+\frac{\rmB}{s_0^2\sqrt{N}}\left(\calJ_\infty(\calH_{l})+2\Const{\calG,B} x\right)\right)
        \leq \sqrt{N}\frac{2^{2l}c_1R^2}{384}.
    \end{equation*}
    for $c_1$ sufficiently large, depending on $c_1,\texttt{S}$. All together, we have shown
    \begin{equation*}
        \IP\left(\sup_{\theta\in\Thetadiamond_l}|T_{N,1}(\theta)|\geq \sqrt{N}\frac{2^{2l}c_1R^2}{384}\right)\leq 3\exp(-\Bar{c}N2^{2l}R^2).
    \end{equation*}
    By the same arguments, with a change of constants, we can similarly show that
    \begin{equation*}
        \IP\left(\sup_{\theta\in\Thetadiamond_l}|T_{N,2}(\theta)-\IE_{\thetatrue}^N[T_{N,2}(\theta)]|\geq \sqrt{N}\frac{2^{2l}c_1R^2}{192}\right)
        \leq 3\exp(-\Bar{c}N2^{2l}R^2).
    \end{equation*}
    Thus, overall, we obtain
    \begin{equation*}
        \sum_{l\in\IN}\IP\left(\sup_{\theta\in\Thetadiamond_l}|\widetilde{\calW}_N(\theta)|\geq \frac{2^{2l}c_1^2R^2}{48}\right)
        \leq 6 \sum_{l\in\IN}\exp(-\Bar{c}N2^{2l}R^2)
        \leq \frac{3}{\ln(2)} \exp(-\Bar{c}NR^2),
    \end{equation*}
    where we have used \cref{lem:ExpoSum} with $\rma = 0$, $\rmb = \Bar{c}NR^2$, $\rmB = 2$ and $\rmc=2$. This shows the claim. 
\end{proof}

%% file: Sections/_app_PDE.tex
\section{Analysis and PDE Theory}\label{sec:app:PDE}

In the following we will make regulary usage of \textit{Young's Inequality for products} that is 
\begin{equation}\label{eq:Young}
    \forall\epsilon\in(0,\infty),\;\forall a,b\in[0,\infty):\;ab \leq \frac{a^2}{2\epsilon}+\frac{\epsilon b^2}{2}.
\end{equation}

We note that for any $s\in[0,\infty)$, $(\Hdot^s(\torus))^* \simeq \Hdot^{-s}(\torus)$ is the topological dual w.r.t the $\IL^2$-pairing. In particular, we have for $u,v\in C^\infty(\IT^d)$
    \begin{equation}\label{eq:CauchySchwarz}
        \abs{\sprod{u}{v}{\IL^2(\torus)}} \leq \norm{u}_{\Hdot^s(\torus)}\cdot\norm{v}_{\Hdot^{-s}(\torus)}.
    \end{equation}
    Further, for $s_1,s_2,s_3\in \IR$, we have 
    \begin{equation}\label{eq:PartialIntegration}
        \sprod{(-\Delta_T)^{s_1}u}{(-\Delta_T)^{s_2}v}{\Hdot^{s_3}(\torus)} = \sprod{u}{(-\Delta_T v)^{s_1+s_2+s_3}}{\IL^2(\torus)}
    \end{equation}
    as well as 
    \begin{equation}\label{eq:Deltanorm}
        \norm{(-\Delta_T)^{s_1}}_{\Hdot^{s_2}(\torus)} = \norm{u}_{\Hdot^{2s_1+s_2}}.
    \end{equation}
    For further details we refer to Section A in \cite{Konen_Nickl_2025}.

\subsection{Reaction Diffusion Equation}

In this section, we derive the analytical properties the solution of the non-linear reaction diffusion equation needed for the misspecification setting in \cref{sec:examples}. To that end, we recall the framework presented in \cite{Nickl_2025_BvMRDE}. Let $d\in\IN$ and $\calM = \IT^d$ the $d$-dimensional torus. Fix some time horizont $T\in(0,\infty)$.
We are interested in periodic solutions $u$ to the parabolic PDE
\begin{equation}\label{eq:ReactionDiffusion}
    \begin{dcases}
        \frac{\partial}{\partial t}u_\theta(t,x) - \Delta u_\theta(t,x)& = f(u_\theta(t,x))\text{ on $\IT^d\times(0,T]$}\\
        \hfill u_\theta(0,x) &= \theta(x)\text{ on $\IT^d$},
    \end{dcases}
\end{equation}
where $\Delta = \sum_{i\nset{d}}\frac{\partial^2}{\partial x_i^2}$ denotes the spatial Laplacian, $f\in C_c^\infty(\IR)$ is a reaction term, and $\theta\in H^1(\IT^d)$ is a initial condition. A \textit{weak} solution to \cref{eq:ReactionDiffusion} is a map $u\in\IL^2([0,T],H^1(\IT^d))\cap C^0([0,T],\IL^2(\IT^d))$, such that $\frac{\partial}{\partial t}u\in\IL^2([0,T],H^{-1}(\IT^d))$ satisfies
\begin{equation*}
    \begin{dcases}
        \sprod{\frac{\partial}{\partial t}u_\theta(t,\cdot)}{v}{\IL^2(\IT^d)} = \sprod{\nabla  u_\theta(t,\cdot)}{\nabla  v}{\IL^2(\IT^d)} &= \sprod{f(u_\theta(t,\cdot))}{v}{\IL^2(\IT^d)}\\
        \hfill 
        u_\theta(\cdot,0) &= \theta
    \end{dcases}
\end{equation*}
for all $v\in H^1(\IT^d)$ and a.e. $t\in(0,T]$. $u$ is called \textit{strong} solution to \cref{eq:ReactionDiffusion}, if in addition 
\begin{equation*}
    u\in\IL^2((0,T],H^2(\IT^d))\cap C^0((0,T],H^1(\IT^d))\text{ and }\frac{\partial}{\partial t}u\in\IL^2([0,T],\IL^2(\IT^d)),
\end{equation*}
in which case then $u$ solves \cref{eq:ReactionDiffusion} as an equation in $\IL^2([0,T],\IL^2(\IT^d))$.
In Theorem 6 in Section 3.1.1. in \cite{Nickl_2025_BvMRDE} it is shown that under these conditions there exists always a strong solution $u=u_\theta^f$ to the reaction-diffusion equation \cref{eq:ReactionDiffusion} that is unique in $C^0([0,T],\IL^2(\IT^d))$.\\

For the proofs of \cref{sec:examples}, where we misspecifiy the reaction term $f$, we need some analytical properties of the map $f\mapsto u_\theta^f$ to verify the assumptions formulated in \cref{cond:misspec-model}. To that end, we need to refine Proposition 5 in \cite{Nickl_2025_BvMRDE}.

\begin{lemma}\label{lemma:prop5new}
    For $a\in\IN_0$, let $\theta\in H^a(\IT^d)\cap H^1(\IT^d)$ and $f\in C^\infty_c(\IR)$ fixed. Let $u = u_\theta$ be a solution to \cref{eq:ReactionDiffusion}. Then, there exists a constant $c=c(f,T,d,a)$, such that 
    \begin{equation*}
        \sup_{t\in[0,T]}\norm{u_\theta(t,\cdot)}_{H^a(\IT^d)}^2 + \int_0^T\norm{u_\theta(t,\cdot)}_{H^{a+1}(\IT^d)}^2 \rmd t \leq c\cdot\left(1+\norm{\theta}_{H^a(\IT^d)}^{2(a-1)!}\right),
    \end{equation*}
    where we define $(-1)! := 1$.
\end{lemma}

\begin{remark}
    In \cite{Nickl_2025_BvMRDE} the claim was proven under the additional assumption that $\norm{\theta}_{H^a(\IT^d)}\leq B$ for some fixed and known constant $B\in(0,\infty)$, such that the right-hand side of the equation above is bounded by a constant  depending additionally on $B$.
\end{remark}

\begin{proof}[Proof of \cref{lemma:prop5new}]

From the proof of Proposition 4 in \cite{Nickl_2025_BvMRDE}, it suffices to assume that the solution $u_\theta$ is sufficiently regular as the general statement then follows from a Galerkin approximation argument.
\begin{itemize}
    \item Step I - Preliminary identities and estimates: Let $a\in\IN_0$. Using the fact that 
    \begin{equation*}
        \forall t\in[0,T]:\,\norm{\nabla  u_\theta(t,\cdot)}_{h^a}^2 = - \sprod{\Delta u_\theta(t,\cdot)}{u_\theta(t,\cdot)}{h^a}
    \end{equation*}
    and differentiating the squared $h^a$-norm yields
    \begin{equation}\label{eq:NewPropo5BaiscEquation1}
        \forall t\in[0,T]:\,\frac{1}{2}\frac{\rmd}{\rmd t}\norm{u_\theta(t,\cdot)}_{h^a}^2 + \norm{\nabla  u_\theta(t,\cdot)}_{h^a}^2 = \sprod{f(u_\theta(t,\cdot))}{u_\theta(t,\cdot)}{h^a}.
    \end{equation}
    For $a = 0$, we observe
    \begin{equation*}
        \forall t\in[0,T]:\,\sprod{f(u_\theta(t,\cdot))}{u_\theta(t,\cdot)}{\IL^2(\IT^d)} = \int_{\IT^d} f(u_\theta(t,x))u\mathds{1}_{\left\{|u_\theta(t,x)|\leq K\right\}}\mathrm{d}x \leq C_f  < \infty
    \end{equation*}
    for some $K\in(0,\infty)$ as $f$ is compactly supported. Thus, \cref{eq:NewPropo5BaiscEquation1} implies
    \begin{equation*}
        \forall t\in[0,T]:\; \norm{u_\theta(t,\cdot)}_{\IL^2(\torus)}^2 + 2\int_0^t\norm{\nabla  u_\theta(s,\cdot)}_{\IL^2(\torus)}^2\rmd s \leq 2 C_f t + \norm{\theta}_{\IL^2(\torus)}^2.
    \end{equation*}
    From this, elementray computations lead to 
    \begin{equation*}
        \sup_{t\in[0,T]} \norm{u_\theta(t,\cdot)}_{\IL^2(\torus)}^2 + 2\int_0^T\norm{ u_\theta(t,\cdot)}_{H^1(\torus)}^2\rmd t \leq c(f,T)\times\br{1+\norm{\theta}_{\IL^2(\torus)}^2}.
    \end{equation*}
    This already shows the claim for $a = 0$. For $a\in\IN$, going back to \cref{eq:NewPropo5BaiscEquation1},
    an application of the Cauchy-Schwarz inequality further yields
    \begin{align*}
        \forall t\in[0,T]:\,\sprod{f(u_\theta(t,\cdot))}{u_\theta(t,\cdot)}{h^a}&\leq \norm{f(u_\theta(t,\cdot))}_{h^{a-1}}\cdot\norm{u_\theta(t,\cdot)}_{h^{a+1}}.
    \end{align*}
    Using Young's inequality with $\epsilon = 2$ yields
    \begin{align*}
        \norm{f(u_\theta(t,\cdot))}_{h^{a-1}}\cdot\norm{u_\theta(t,\cdot)}_{h^{a+1}}
        &\leq \norm{f(u_\theta(t,\cdot))}_{h^{a-1}}^2 + \frac{1}{4}\norm{u_\theta(t,\cdot)}_{h^{a+1}}^2\\
        &\simeq \norm{f(u_\theta(t,\cdot))}_{H^{a-1}}^2\cdot\frac{1}{4}\br{\norm{\nabla  u_\theta(t,\cdot)}_{h^a}^2+\norm{u_\theta(t,\cdot)}_{\IL^2(\IT^d)}^2},
    \end{align*}
    where we have used $\norm{\cdot}_{H^{a+1}(\torus)}^2\simeq \norm{\cdot}_{\IL^2(\torus)}^2+\norm{\nabla \cdot}_{h^a}^2$.
    Thus, we have 
    \begin{equation}\label{eq:NewPropo5BaiscEquation2}
        \forall t\in[0,T]:\,\frac{1}{2}\frac{\rmd}{\rmd t}\norm{u_\theta(t,\cdot)}_{h^a}^2 + \frac{3}{4}\norm{\nabla  u_\theta(t,\cdot)}_{h^a}^2 \lesssim \norm{f(u_\theta(t,\cdot))}_{H^{a-1}}^2 + \frac{1}{4}\norm{u_\theta(t,\cdot)}_{\IL^2(\IT^d)}^2.
    \end{equation}

    \item Step II - Induction over $a$: We will prove the general claim by an induction argument.
    To that end, we start with the induction beginning by proving the claim for $a = 1$ and $a = 2$ explicitly. 
    \begin{itemize}
        \item \textbf{(IB)} $a = 1$: We start with \cref{eq:NewPropo5BaiscEquation2} and obtain
        \begin{equation*}
            \forall t\in[0,T]:\,\frac{1}{2}\frac{\rmd}{\rmd t}\norm{u_\theta(t,\cdot)}_{h^1}^2 + \frac{3}{4}\norm{\nabla  u_\theta(t,\cdot)}_{h^1}^2 \lesssim \norm{f(u_\theta(t,\cdot))}_{\IL^2(\torus)}^2 + \frac{1}{4}\norm{u_\theta(t,\cdot)}_{\IL^2(\IT^d)}^2.
        \end{equation*}
        Observing that 
        \begin{equation*}
            \forall t\in[0,T]:\,\norm{f(u_\theta(t,\cdot))}_{\IL^2(\torus)}^2  \leq \norm{f}_{\infty}^2
        \end{equation*}
        as well as using the bound derived for $a = 0$ leads to 
        \begin{equation*}
             \forall t\in[0,T]:\,\frac{1}{2}\frac{\rmd}{\rmd t}\norm{u_\theta(t,\cdot)}_{h^1}^2 + \frac{3}{4}\norm{\nabla  u_\theta(t,\cdot)}_{h^1}^2 \lesssim c(f,T)\times\br{1+\norm{\theta}_{\IL^2(\torus)}^2}.
        \end{equation*}
        Integration yields
        \begin{equation*}
         \sup_{t\in[0,T]}\norm{u_\theta(t,\cdot)}_{h^1}^2 + \frac{3}{2}\int_{0}^T\norm{\nabla  u_\theta(t,\cdot)}_{h^1}^2\rmd t \leq c(f,t)\br{1+\norm{\theta}_{H^1(\torus)}^2},
        \end{equation*}
        which shows the claim.
        \item \textbf{(IB)} $a = 2$: We start with \cref{eq:NewPropo5BaiscEquation2} and obtain
        \begin{equation*}
             \forall t\in[0,T]:\, \frac{1}{2}\frac{\rmd}{\rmd t}\norm{u_\theta(t,\cdot)}_{h^2}^2 + \frac{3}{4}\norm{\nabla  u_\theta(t,\cdot)}_{h^2}^2 \lesssim \norm{f(u_\theta(t,\cdot))}_{h^{1}}^2 + \frac{1}{4}\norm{u_\theta(t,\cdot)}_{\IL^2(\IT^d)}^2.
        \end{equation*}
        Note that 
        \begin{align*}
             \norm{f(u_\theta(t,\cdot))}_{h^{1}}^2 & \simeq \norm{f(u_\theta(t,\cdot))}_{\IL^2(\torus)}^2 + \max_{\abs{\boldbeta} = 1}\norm{\rmD^{\boldbeta}f(u_\theta(t,\cdot))}_{\IL^2(\torus)}^2\\
             & \leq  \norm{f}_\infty^2 + \norm{f'}_\infty^2\cdot\norm{u_\theta(t,\cdot)}_{H^1(\torus)}^2
              \leq c(f,f',T)\times\br{1+\norm{\theta}_{H^1(\torus)}^2},
        \end{align*}
        where we have used the bound derived for $a=1$ in the last step.
        Thus, we obtain for all $t\in[0,T]$
         \begin{equation*}
             \frac{1}{2}\frac{\rmd}{\rmd t}\norm{u_\theta(t,\cdot)}_{h^2}^2 + \frac{3}{4}\norm{\nabla  u_\theta(t,\cdot)}_{h^2}^2 \lesssim c(f,f',T)\times\br{1+\norm{\theta}_{H^1(\torus)}^2} + \frac{1}{4}\norm{u_\theta(t,\cdot)}_{\IL^2(\IT^d)}^2,
        \end{equation*}
        which - after integration - shows the claim for $a=2$.

        \item \textbf{(IH)} For $a \in\IN$ fixed, there exists a finite constant $c=c(f,T)\in(0,\infty)$, such that
        \begin{equation*}
            \sup_{t\in[0,T]}\norm{u_\theta(t,\cdot)}_{H^a(\IT^d)}^2 + \frac{3}{2}\int_0^T\norm{u_\theta(t,\cdot)}_{H^{a+1}(\IT^d)}^2\rmd t \leq c(f,T)\times\left(1+\norm{\theta}_{H^a(\IT^d)}^{2(a-1)!}\right).
        \end{equation*}
        \item \textbf{(IS)} $a \Rightarrow a+1$: Under \textbf{(IH)}, we consider the case $a +1$. We start with \cref{eq:NewPropo5BaiscEquation2} and obtain
        \begin{equation*}
             \forall t\in[0,T]:\,\frac{1}{2}\frac{\rmd}{\rmd t}\norm{u_\theta(t,\cdot)}_{h^{a+1}}^2 + \frac{3}{4}\norm{\nabla  u}_{h^{a+1}}^2 \lesssim \norm{f(u_\theta(t,\cdot))}_{h^{a}}^2 + \frac{1}{4}\norm{u_\theta(t,\cdot)}_{\IL^2(\IT^d)}^2.
        \end{equation*}
        With \cref{thm:Composition} we obtain 
        \begin{equation*}
             \forall t\in[0,T]:\,\norm{f(u(t,\cdot))}_{h^{a}}^2 \lesssim 1+\norm{u(t,\cdot)}_{H^a(\IT^d)}^{2a}.
        \end{equation*}
        Then \textbf{(IH)} yields 
        \begin{equation*}
            \forall t\in[0,T]:\, \norm{u(t,\cdot)}_{H^a(\IT^d)}^{2a} \lesssim 1+\norm{\theta}_{H^a(\IT^d)}^{2a(a-1)!} =1+\norm{\theta}_{H^a(\IT^d)}^{2a!}
        \end{equation*}
        Thus, 
        \begin{equation*}
            \norm{f(u)}_{h^{a}}^2 \lesssim1+\norm{\theta}_{H^a(\IT^d)}^{2a!}.
        \end{equation*}
        Integrating the resulting inequality as before then yields the desired claim.
    \end{itemize}
\end{itemize}
\end{proof}

\begin{lemma}\label{lem:RDE_f_Lipschitz}
    Let $d\nset{3}$. Let $\theta\in H^2(\IT^d)$. Denote by $u_{\theta}^{f_1}$ and $u_{\theta}^{f_1}$ the unique solutions to \cref{eq:ReactionDiffusion} with different reaction terms  $f_1,f_2\in C_c^\infty(\IR)$. Then there exists a finite constant $c=c(f_1,f_2,T)\in(0,\infty)$, such that
    \begin{equation*}
        \sup_{t\in[0,T]}\norm{u_{\theta}^{f_1}(t,\cdot)-u_{\theta}^{f_2}(t,\cdot)}_{H^2(\torus)}^2 \leq c(f_1,f_2,T,d)\cdot\br{1+\norm{\theta}_{H^2(\torus)}^2}\cdot\norm{f_1-f_2}_{C^1(\IR)}^2.
    \end{equation*}
\end{lemma}

\begin{proof}[Proof of \cref{lem:RDE_f_Lipschitz}]
    Let $\theta\in H^2(\torus)$ be arbitrary but fixed. We show the claim successively by deriving a $\IL^2$-bound first, a $H^1$-bound second and finally the desired claim. 
    We define the difference $w\coloneqq u_{\theta}^{f_1}-u_{\theta}^{f_2}$, which solves the PDE 
    \begin{equation*}
        \begin{dcases}
            \partial_t w(t,x) -\Delta w(t,x) &= f_1( u_{\theta}^{f_1}(t,x)) - f_2( u_{\theta}^{f_2}(t,x)),\quad \text{on $(0,T]\times \torus$}\\
            \hfill w(0,x) & = 0,\quad\hfill \text{on $\torus$}\\
        \end{dcases}
    \end{equation*}
    Analogously as in the proof of \cref{lemma:prop5new}, for all $a\in\IN_0$, we then have the basic equality 
    \begin{equation}\label{eq:RDE:basic}
        \frac{1}{2}\frac{\rmd}{\rmd t}\norm{w(t,\cdot)}_{h^a}^2 +\norm{\nabla  w(t,\cdot)}_{h^a}^2 = \sprod{f_1( u_{\theta}^{f_1}(t,\cdot)) - f_2( u_{\theta}^{f_2}(t,\cdot))}{w(t,\cdot)}{h^a}.
    \end{equation}
    In the subsequent, we will upper bound the right-hand-side of the equation accordingly.
    \begin{itemize}
        \item $a = 0$: For all $t\in[0,T]$, we can write
        \begin{align*}
            f_1( u_{\theta}^{f_1}(t,\cdot)) - f_2( u_{\theta}^{f_2}(t,\cdot)) = f_1( u_{\theta}^{f_1}(t,\cdot)) -  f_2( u_{\theta}^{f_1}(t,\cdot)) +  f_2( u_{\theta}^{f_1}(t,\cdot))- f_2( u_{\theta}^{f_2}(t,\cdot)),
        \end{align*}
        such that
        \begin{equation*}
             \forall t\in[0,T]:\; \abs{f_1( u_{\theta}^{f_1}(t,\cdot)) - f_2( u_{\theta}^{f_2}(t,\cdot))} \leq \norm{f_1-f_2}_{\infty} + \norm{f_2'}_\infty\cdot \abs{w(t,\cdot)}.
        \end{equation*}
        Thus, by Cauchy-Schwarz inequality
        \begin{align*}
            &\abs{\sprod{f_1( u_{\theta}^{f_1}(t,\cdot)) - f_2( u_{\theta}^{f_2}(t,\cdot))}{w(t,\cdot)}{\IL^2(\torus)}}\\
            &\hspace{2cm}\leq \norm{f_1( u_{\theta}^{f_1}(t,\cdot)) - f_2( u_{\theta}^{f_2}(t,\cdot))}{\IL^2(\torus)}\cdot \norm{w(t,\cdot)}{\IL^2(\torus)}\\
            &\hspace{2.5cm}\leq \norm{f_1-f_2}_{\infty} \cdot \norm{w(t,\cdot)}{\IL^2(\torus)} + \norm{f_2'}_\infty\norm{w(t,\cdot)}_{\IL^2(\torus)}^2\\
            &\hspace{3cm}\leq \frac{1}{2}\norm{f_1-f_2}_{\infty}^2 + \br{\frac{1}{2}+\norm{f_2'}_\infty}\norm{w(t,\cdot)}{\IL^2(\torus)}^2.
        \end{align*}
        Thus, the basic inequality yields
        \begin{align*}
            \frac{\rmd}{\rmd t}\norm{w(t,\cdot)}_{\IL^2(\torus)}^2 +2\norm{\nabla  w(t,\cdot)}_{\IL^2(\torus)}^2 & \leq \norm{f_1-f_2}_{\infty}^2 + \br{1+2\norm{f_2'}_\infty}\norm{w(t,\cdot)}{\IL^2(\torus)}^2.
        \end{align*}
        Integrating the last inequality and applying Grönwall's inequality yields 
        \begin{equation*}
            \sup_{t\in[0,T]}\norm{w(t,\cdot)}_{\IL^2(\torus)}^2 \leq c(T,f_1,f_2)\times\norm{f_1-f_2}_{\infty}^2.
        \end{equation*}
    
        \item $a = 1$: Applying Cauchy-Schwarz inequality, the right hand side of \cref{eq:RDE:basic} can be for all $t\in[0,T]$ upper bounded by 
        \begin{align*}
            &\sprod{f_1( u_{\theta}^{f_1}(t,\cdot)) - f_2( u_{\theta}^{f_2}(t,\cdot))}{w(t,\cdot)}{h^1} \leq \norm{f_1( u_{\theta}^{f_1}(t,\cdot)) - f_2( u_{\theta}^{f_2}(t,\cdot))}_{\IL^2(\torus)}\cdot\norm{w(t,\cdot)}_{h^2}\\
            &\leq \norm{f_1( u_{\theta}^{f_1}(t,\cdot)) - f_2( u_{\theta}^{f_2}(t,\cdot))}_{\IL^2(\torus)}^2 + \frac{1}{4}\norm{w(t,\cdot)}_{\IL^2(\torus)}^2 + \frac{1}{4}\norm{\nabla  w(t,\cdot)}_{h^1}^2\\
            & \lesssim \norm{f_1-f_2}_\infty^2+\br{\norm{f_2'}_\infty^2+\frac{1}{4}}\norm{w(t,\cdot)}_{\IL^2(\torus)}^2+\frac{1}{4}\norm{\nabla  w(t,\cdot)}_{h^1}^2.
        \end{align*}
        Integrating \cref{eq:RDE:basic} after absorbing the remaining term on the right-hand side, we have for all $t\in[0,T]$
        \begin{equation*}
        \norm{w(t,\cdot)}_{h^1}^2 +\frac{3}{2}\int_0^t\norm{\nabla w(s,\cdot)}_{h^1}^2\rmd s \lesssim \norm{f_1-f_2}_\infty^2+\br{\norm{f_2'}_\infty^2+\frac{1}{4}}\int_0^t\norm{w(s,\cdot)}_{h^1}^2\rmd s.
        \end{equation*}
        Using the bound derived for $a=0$ yields
        \begin{equation*}
            \sup_{t\in[0,T]}\norm{w(t,\cdot)}_{h^1}^2 \leq c(T,f_1,f_2)\times\norm{f_1-f_2}_\infty^2.
        \end{equation*}
        
        \item $a = 2$: Again, by applying Cauchy-Schwarz inequality to the right-hand side of \cref{eq:RDE:basic} yields 
        \begin{align*}
            \forall t\in[0,T]:\;&\sprod{f_1( u_{\theta}^{f_1}(t,\cdot)) - f_2( u_{\theta}^{f_2}(t,\cdot))}{w(t,\cdot)}{h^2}\\ 
            &\hspace{1cm}\leq \norm{f_1( u_{\theta}^{f_1}(t,\cdot)) - f_2( u_{\theta}^{f_2}(t,\cdot))}_{h^1}^2+ \frac{1}{4}\norm{w(t,\cdot)}_{\IL^2(\torus)}^2 + \frac{1}{4}\norm{\nabla  w(t,\cdot)}_{h^2}^2.
        \end{align*}
        As before, we can use the bounds derived for $a=0$ to get an $\IL^2$-estimate for $w$ for the second term. The third expression is absorbed by the left-hand-side of \cref{eq:RDE:basic}. It remains to upper bound the first expression. First observe
        \begin{align*}
            \forall t\in[0,T]:\;&\norm{f_1( u_{\theta}^{f_1}(t,\cdot)) - f_2( u_{\theta}^{f_2}(t,\cdot))}_{h^1}^2\\
            &\hspace{1cm}\leq \norm{f_1( u_{\theta}^{f_1}(t,\cdot))-f_1( u_{\theta}^{f_2}(t,\cdot))}_{h^1}^2
            +\norm{f_1( u_{\theta}^{f_2}(t,\cdot)) - f_2( u_{\theta}^{f_2}(t,\cdot))}_{h^1}^2\\
            & \hspace{1cm}\eqcolon E_1(t)+E_2(t).
        \end{align*}
        For $E_1$ we have for all $t\in[0,T]$
        \begin{align*}
             E_1(t) 
             &\lesssim  \norm{f_1( u_{\theta}^{f_1}(t,\cdot))-f_1( u_{\theta}^{f_2}(t,\cdot))}_{\IL^2(\torus)}^2+\max_{|\boldbeta|=1}\norm{\rmD^{\boldbeta} f_1( u_{\theta}^{f_1}(t,\cdot))-\rmD^{\boldbeta} f_1( u_{\theta}^{f_2}(t,\cdot))}_{\IL^2(\torus)}^2\\
             &\leq \norm{f_1''}_\infty^2\cdot\norm{u_\theta^{f_1}}_{C^1(\torus)}^2\cdot\norm{w(t,\cdot)}_{\IL^2(\torus)}^2 + \norm{f_1'}_\infty^2\cdot\norm{w(t,\cdot)}_{h^1}^2.
        \end{align*}
         As $d\leq3$, we have the embedding $H^3(\torus)\hookrightarrow C^1(\torus)$ and hence 
        \begin{equation*}
            \forall t\in[0,T]:\; E_1(t) \lesssim \br{1+\norm{u_{\theta}^{f_1}(t,\cdot)}_{H^3(\torus)}^2}\cdot\norm{f_1-f_2}_{\infty}^2,
        \end{equation*}
        where in the last step we also used the bounds derived for $a\in\brK{0,1}$.
        Similarly, for $E_2$ we derive for all $t\in[0,T]$
        \begin{align*}
            E_2(t) 
            &\lesssim \norm{f_1( u_{\theta}^{f_2}(t,\cdot)) - f_2( u_{\theta}^{f_2}(t,\cdot))}_{\IL^2(\torus)}^2 \\
            &\hspace{2cm}+\max_{|\boldbeta|=1}\norm{\rmD^{\boldbeta} f_1( u_{\theta}^{f_2}(t,\cdot)) - \rmD^{\boldbeta} f_2( u_{\theta}^{f_2}(t,\cdot))}_{\IL^2(\torus)}^2\\
            & \leq \norm{f_1-f_2}_\infty^2  +\max_{|\boldbeta|=1}\norm{\rmD^{\boldbeta} u_{\theta}^{f_2}(t,\cdot)\cdot \br{ f_1'( u_{\theta}^{f_2}(t,\cdot)) - f_2'( u_{\theta}^{f_2}(t,\cdot))}}_{\IL^2(\torus)}^2\\
            &\lesssim \norm{f_1-f_2}_\infty^2 + \norm{u_{\theta}^{f_2}(t,\cdot)}_{C^1(\torus)}^2\cdot\norm{f_1'-f_2'}_\infty^2.
        \end{align*}
        As $d\leq3$, we have the embedding $H^3(\torus)\hookrightarrow C^1(\torus)$ and hence 
        \begin{equation*}
            \forall t\in[0,T]:\; E_2(t) \lesssim \br{1+\norm{u_{\theta}^{f_2}(t,\cdot)}_{H^3(\torus)}^2}\cdot\norm{f_1-f_2}_{C^1(\IR)}^2.
        \end{equation*}
        Overall, the right-hand side of \cref{eq:RDE:basic} is then for all $t\in[0,T]$ upper bounded by 
        \begin{align*}
            \frac{1}{2}\frac{\rmd}{\rmd t}\norm{w(t,\cdot)}_{h^2}^2 +\frac{3}{4}\norm{\nabla  w(t,\cdot)}_{h^2}^2 
            & \leq c(f_1,f_2,T)\br{1+\norm{u_{\theta}^{f_2}(t,\cdot)}_{H^3(\torus)}^2}\cdot\norm{f_1-f_2}_{C^1(\IR)}^2.
        \end{align*}
        Integrating everything and using \cref{lemma:prop5new} with $a =2$, we have 
        \begin{equation*}
            \norm{w(t,\cdot)}_{h^2}^2 +\frac{3}{2}\int_0^t\norm{\nabla  w(s,\cdot)}_{h^2}^2\rmd s \leq c(f_1,f_2,T,d)\cdot\br{1+\norm{\theta}_{H^2(\torus)}^2}\cdot\norm{f_1-f_2}_{C^1(\IR)}^2
        \end{equation*}
        for all $t\in[0,T]$, which shows the claim.
    \end{itemize}
\end{proof}

\subsection{2D-Navier-Stokes Equation}

In the following, we derive regularity estimates for the solution of the 2D-Navier-Stokes equation, see \cref{eq:navier-stokes}. To that end recall the setting introduced in \cref{sec:examples}. As is common in the literature for the Navier-Stokes equation, we study the \textit{projected} equation, that is given the Leray operator $P$ (see \cref{eq:leray-operator}), we are interested in the solution $u:[0,T]\to\Hddiamond$ solving 
\begin{equation}\label{eq:2DNSE}
    \begin{dcases}
        \frac{\rmd}{\rmd t}u + \nu Au + B[u,u] &= f\\
        \hfill u_\theta(0) &= \theta
    \end{dcases},
\end{equation}
where $A\coloneqq -P\Delta$ denotes the Stokes operator, and 
\begin{equation*}
    B[u,v]\coloneqq P\br{\br{u\cdot\nabla }v}.
\end{equation*}
We call $u$ a strong solution of \cref{eq:2DNSE}, if the equations holds in $\IL^2([0,T],\Hddiamond)$. Based on the theory provided in \cite{NicklTiti_2024} it is shown in \cite{Konen_Nickl_2025} that for every $a\in\IN$, $\nu\in(0,\infty)$, $\theta\in\Hddiamond^a$, and $f\in\IL^2([0,T], \Hddiamond^{a-1})$ there exists a strong solution $u = u_\theta$ of \cref{eq:2DNSE} that satisfies
\begin{equation*}
    u_\theta\in C^0\br{[0,T],\Hddiamond^a}\cap\IL^2\br{[0,T],\Hddiamond^{a+1}}.
\end{equation*}
In the following, we derive an regularity estimate for differences of solutions of \cref{eq:2DNSE} with different viscosity and forcing terms, which is used in \cref{thm:example-model-rde}.

\begin{lemma}\label{lem:NS:Misspec:regularity}
    Let $\theta\in \Hddiamond^2$, such that $\norm{\theta}_{\Hdot^2(\IT^2)}\leq B$ for some $B\in(0,\infty)$, $\nu_1,\nu_2\in(0,\infty)$, and $f_1,f_2\in\IL^2\br{[0,T],\Hddiamond^1}$. Denote by $u_\theta^{\nu_i,f_i}$, $i\nset{2}$, the solutions to \cref{eq:2DNSE} with viscosity $\nu_i$ and external forcing $f_i$, respectively. We then have 
    \begin{equation*}
        \sup_{t\in[0,T]}\norm{u_\theta^{\nu_1,f_1}(t)-u_\theta^{\nu_2,f_2}(t)}_{\Hdot^2(\IT^2)}^2\leq c(\nu_1,\nu_2,f_1,f_2,T,B)\times \br{\abs{\nu_1-\nu_2}^2+\norm{f_1-f_2}_{\IL^2\br{[0,T],\Hdot^1(\IT^2)}}^2}.
    \end{equation*}
\end{lemma}

\begin{proof}[Proof of \cref{lem:NS:Misspec:regularity}]
    In this proof, we shorten notation by writing $u_i\coloneqq u_\theta^{\nu_i,f_i}$ for $i\nset{2}$ and further, we write $\Hdot^2\coloneqq\Hdot^2(\IT^2)$. First observe, that $w\coloneqq u_1-u_2$ satisfies the equation 
    \begin{equation*}
        \frac{\rmd}{\rmd t} w + \nu_1\Delta w +\br{\nu_1-\nu_2}\Delta u_2 + B[w,u_1] - B[u_2,w] = f_1 - f_2
    \end{equation*}
    on $[0,T]\times\IT^2$ with initial condition $w(0,\cdot) = 0$ on $\IT^2$. In the following, the proof is understood in terms of a Galerkin approximation, such that $w$ can be viewed of being smooth in the sense that $w\in C^\infty\cap\Hddiamond$, see \cite[Proposition A.6 and Proposition B.1]{Konen_Nickl_2025} for example. Thus, taking the $\Hdot^2$-inner product with $w$ yields
    \begin{align*}
        \frac{1}{2}\frac{\rmd}{\rmd t}\norm{w(t)}_{\Hdot^2}^2 + \nu_1\norm{w(t)}_{\Hdot^3}^2
        &= - \br{\nu_1-\nu_2}\sprod{\Delta u_2(t)}{w(t)}{\Hdot^2} - \sprod{B[w(t),u_1(t)]}{w(t)}{\Hdot^2}\\
        &\hspace{2cm}- \sprod{B[u_2(t),w(t)]}{w(t)}{\Hdot^2} +  \sprod{f_1(t)-f_2(t)}{w(t)}{\Hdot^2}\\
        & \eqcolon E_1(t) + E_2(t) + E_3(t) + E_4(t).
    \end{align*}
    In the subsequent, we upper bound the terms $E_1,\dots,E_4$.
    \begin{enumerate}
        \item $E_1(t)$: Utilising \cref{eq:PartialIntegration}, \cref{eq:CauchySchwarz} as well as \cref{eq:Deltanorm}, we obtain for all $t\in[0,T]$
        \begin{align*}
            \abs{E_1(t)} &= \abs{\nu_1-\nu_2}\cdot\abs{\sprod{u_2(t)}{(-\Delta)^3w(t)}{\Hdot^0}}
            \leq \abs{\nu_1-\nu_2}\cdot \norm{u_2(t)}_{\Hdot^3}\cdot\norm{(-\Delta)^3w(t)}_{\Hdot^{-3}}\\
            & \leq \abs{\nu_1-\nu_2}\cdot\norm{u_2(t)}_{\Hdot^3}\cdot\norm{w(t)}_{\Hdot^{3}}
            \leq c(\nu_1)\cdot\abs{\nu_1-\nu_2}^2\cdot\norm{u_2(t)}_{\Hdot^3}^2 + \frac{1}{6}\nu_1\norm{w(t)}_{\Hdot^{3}}^2
        \end{align*}

         \item $E_2(t)+E_3(t)$: Applying \cite[Proposition A.3]{Konen_Nickl_2025} with $a=2$ twice, we obtain for all $t\in[0,T]$
            \begin{align*}
                \abs{E_2(t)+E_3(t)} &\leq \abs{\sprod{B[w(t),u_1(t)]}{w(t)}{\Hdot^2}} + \abs{\sprod{B[u_2(t),w(t)]}{w(t)}{\Hdot^2}}\\
                & \lesssim \br{\norm{u_1(t)}_{\Hdot^2}+\norm{u_2(t)}_{\Hdot^2}}\cdot \norm{w(t)}_{\Hdot^2}\cdot \norm{w(t)}_{\Hdot^3}\\
                & \leq c(\nu_1)\cdot \br{\norm{u_1(t)}_{\Hdot^2}^2+\norm{u_2(t)}_{\Hdot^2}^2}\cdot \norm{w(t)}_{\Hdot^2}^2 + \frac{1}{6}\nu_1\norm{w(t)}_{\Hdot^3}^2.
            \end{align*}

          \item $E_4(t)$: Applying \cref{eq:CauchySchwarz} as well as Young's Inequality, we obtain for all $t\in[0,T]$
          \begin{equation*}
              \abs{E_4(t)} \leq \norm{f_1(t)-f_2(t)}_{\Hdot^1}\cdot\norm{w(t)}_{\Hdot^3} \leq c(\nu_1)\cdot \norm{f_1(t)-f_2(t)}_{\Hdot^1}^2 + \frac{1}{6}\nu_1\norm{w(t)}_{\Hdot^3}^2.
          \end{equation*}
    \end{enumerate}
    Absorbing the three $\Hdot^3$-terms to the left-hand side, we obtain
    \begin{align*}
       \frac{1}{2}\frac{\rmd}{\rmd t}\norm{w(t)}_{\Hdot^2}^2 + \frac{1}{2}\nu_1\norm{w(t)}_{\Hdot^3}^2 &\leq c(\nu_1)\cdot\abs{\nu_1-\nu_2}^2\cdot\norm{u_2(t)}_{\Hdot^3}^2\\
        &\hspace{1cm}+ c(\nu_1)\cdot \br{\norm{u_1(t)}_{\Hdot^2}^2+\norm{u_2(t)}_{\Hdot^2}^2}\cdot \norm{w(t)}_{\Hdot^2}^2\\
        &\hspace{1cm}+ c(\nu_1)\cdot\norm{f_1(t)-f_2(t)}_{\Hdot^1}^2 
    \end{align*}
    Integrating over $[0,t]$ for any $t\in[0,T]$ and applying \cite[Proposition A.6]{Konen_Nickl_2025} with $a=2$ yields
    \begin{align*}
        \norm{w(t)}_{\Hdot^2}^2 &\leq c(\nu_1,\nu_2,T,f_2,B)\cdot\br{\abs{\nu_1-\nu_2}^2+\norm{f_1-f_2}_{\IL^2([0,T],\Hdot^1)}^2}\\
        &\hspace{2cm}+ c(\nu_1)\int_{0}^t\br{\norm{u_1(s)}_{\Hdot^2}^2+\norm{u_2(s)}_{\Hdot^2}^2}\cdot \norm{w(t)}_{\Hdot^2}^2\rmd s.
    \end{align*}
    Thus, an application of Grönwall's inequality as well as \cite[Proposition A.6]{Konen_Nickl_2025} with $a = 2$ yields 
    \begin{equation*}
    \sup_{t\in[0,T]}\norm{w(t)}_{\Hdot^2}^2 \leq c(\nu_1,\nu_2,T,f_1,f_2,B)\times \br{\abs{\nu_1-\nu_2}^2+\norm{f_1-f_2}_{\IL^2([0,T],\Hdot^1)}^2},
    \end{equation*}
    which shows the claim.

\end{proof}

\subsection{Oseen approximation}

\begin{proof}[Proof of \cref{prop:oseen-satisfies-condition}]
        The proof follows similar arguments as the theory developed in \cite[Appendix A and B]{Konen_Nickl_2025}. 
        Firstly, we look at a general functional equation as
        \begin{equation}\label{oseengeneraleq}
            \frac{\rmd}{\rmd t}U +\nu \Delta U + B[U,v_1]+B[v_2,U] = f,\quad U(0) = \xi.
        \end{equation}
        Denote for $a\in\IZ$
        \begin{equation*}
            a^*\coloneqq 
            \begin{dcases}
                \abs{a}+1,&\quad \text{if $\abs{a}\leq 1$,}\\
                \abs{a},&\quad \text{if $\abs{a}\geq 2$.}
            \end{dcases}
        \end{equation*}
        Following the proof of \cite[Proposition B.1]{Konen_Nickl_2025}, if $\xi\in\Hddiamond^a$, $f\in\IL^2\br{[0,T],\Hddiamond^{a-1}}$, $\nu>0$ and $v_1,v_2\in\IL^2\br{[0,T],\Hddiamond^{a^*}}$, 
        there exists a unique solution $U\colon [0,T]\times\IT^2\to\IR^2$ of \cref{oseengeneraleq}, and we have 
        \begin{equation}\label{eq:osseengeneralregularoity}
            U\in C^{0}([0,T],\Hddiamond^a)\cap\IL^2([0,T],\Hddiamond^{a+1}),\quad\frac{\rmd U}{\rmd t}\in \IL^2([0,T],\Hddiamond^{a-1}).
        \end{equation}
        In particular, 
        \begin{equation}\label{eq:oseen:generalinequality}
            \sup_{t\in[0,T]}\norm{U(t)}_{\Hdot^a}^2 + \nu\int_0^T\norm{U(t)}_{\Hdot^{a+1}}^2\rmd t \leq c\times\br{\norm{\xi}_{\Hdot^a}^2+\norm{f}_{\IL^2\br{[0,T],\Hddiamond^{a-1}}}^2}
        \end{equation}
        with some constant $c=c\br{\nu,T,f,a,\norm{v_1}_{\IL^2([0,T],\Hdot^{a^*})},\norm{v_2}_{\IL^2([0,T],\Hdot^{a^*})}}>0$.
        Thus, looking at the Oseen-type iteration \cref{eq:oseen-iteration}, we deduce that if $u^0\in\IL^2\br{[0,T],\Hddiamond^{a^*}}$,
        for each $l\in\IN$, there exists a iterated solution $u^l$ that satisfies \cref{eq:osseengeneralregularoity}.
        Now let $L=L_N\in\IN$ chosen as by hypothesis, such that
        \begin{equation}\label{eq:oseen:MMI:hypothesis}
            \forall\theta\in\Hddiamond^2:\;\sup_{t\in[0,T]}\norm{u_\theta^{L}(t)}_{\Hdot^2(\IT^2)} \leq \Const{Oseen,B}\times\br{1+\norm{\theta}_{\Hdot^2}^2}.
        \end{equation}
        \begin{equation*}
            \forall r>0:\; \sup_{\theta\in\Hddiamond^2(r)}\sup_{t\in[0,T]}\norm{u_\theta^{L}(t)-u_\theta^{L-1}(t)}_{\Hdot^2(\IT^2)} \leq \Const{\mathrm{model}}(r)\times\delta_N^2,
        \end{equation*}
        We are now verifying that the associated forward map $\calGtilde(\theta)\coloneqq u_\theta^L$ satisfies the conditions of \cref{cond:misspec-model}.
        \begin{itemize}
            \item Proof of \cref{cond:MM:I}: Follows by hypothesis \cref{eq:oseen:MMI:hypothesis}.
            \item Proof of \cref{cond:MM:III}:  Let us write $w^L := u^L-u$.
                    Then, $w^L$ solves
                    \begin{equation}\label{eq:oseen-difference}
                        \frac{\rmd }{\rmd t}w^{L}
                        - \nu \Delta w^{L}
                        + B[u^{L-1}, u^{L}] - B[u,u]
                        = 0,\quad w^L(0) = 0
                    \end{equation}
                    By bilinearity,
                    \begin{align*}
                        B[u^{L-1}, u^{L}] - B[u,u] 
                    &= B[u^{L-1}-u^L, u^L] + B[u^L, u^L] - B[u,u]\\ 
                    &= B[u^{L-1}-u^L, u^L] + B[w^L, u^L] + B[u,u^L] - B[u,u]\\ 
                    &=B[u^{L-1} -u^{L}, u^L] + B[w^L,u^L] + B[u, w^L].
                    \end{align*}
                    Thus, applying \cref{eq:oseen:generalinequality} and \cref{eq:oseen:MMI:hypothesis} with $U = w^L$, $v_1=u^L$, $v_2 = u$, $f = -B[u^{L-1} -u^{L}, u^L]$, $\xi = 0$ and $a=2$, we obtain for all $r>0$
                    \begin{align*}
                         \sup_{\theta\in\Hdot^2(r)}\sup_{t\in[0,T]}\norm{w^L(t)}_{\Hdot^a}^2 &\leq c\times \sup_{\theta\in\Hdot^2(r)}\norm{B[u^{L-1} -u^{L}, u^L]}_{\IL^2\br{[0,T],\Hddiamond^{1}}}^2\\
                            & = c\times\int_0^T\norm{B[u^L-u^{L-1}, u^L](t)}_{\dot H^1}^2\rmd t\\
                            & \leq  c\times \sup_{\theta\in\Hdot^2(r)}\int_0^T\norm{u^L(t)-u^{L-1}(t)}_{\dot H^2}^2\cdot\norm{u^L(t)}_{\Hdot^2}^2\rmd t\\
                            &\leq c'\times \delta_N^4
                    \end{align*}
                    with $c'=c'(\nu,T,f,r)>0$, which shows the claims.
    \end{itemize}

\end{proof}

%% file: Sections/app_miscellaneous.tex
\section{Miscellaneous}

\subsection{An inequality for Sobolev spaces}



\begin{lemma}[Composition with smooth and compactly supported functions]\label{thm:Composition}
    Let $f\in C_c^\infty(\IR)$ and assume that $u\in H^m(\IT^d)$ for some integer $m \geq\frac{d}{2}$. We then have 
    \begin{equation*}
        \norm{f(u)}_{H^m(\IT^d)} \leq c\cdot(1+\norm{u}_{H^m(\IT^d)}^m),
    \end{equation*}
    where $c=c(m,d)\in(0,\infty)$.
\end{lemma}

The proof of \cref{thm:Composition} follows exactly the lines of Lemma 29 in \cite{Nickl_vdGeer_Wang_2020} by utilizing an analogous version of Nierenberg's inequality on the torus, which can be found for instance in Theorem 3.70 in \cite{Aubin_1982}.


\subsection{A chaining lemma for non i.i.d.\ data}

\begin{lemma}\label{lem:chaining}
Let $\Theta$ be a countable set. Consider the family
\begin{equation*}
    \scrH := \{ h_\theta : \calZ \to V \mid \theta \in \Theta \}
\end{equation*}
of $V$-valued functions defined on a probability space $(\calZ,\scrZ,\IP^Z)$.
Assume that there exist finite constants $\rmv,\rmU \in \pRz$ such that
\begin{equation}\label{eq:chaining_C1}
    \sup_{\theta\in\Theta}\IE[\Vnorm{h_\theta(Z)}^2] \leq \rmv^2,
    \qquad
    \sup_{\theta\in\Theta}\norm{h_\theta}_\infty \leq \rmU,
    \tag{C1}
\end{equation}
where $Z\sim\IP^Z$. Define the entropy integrals $\calJ_2(\scrH)$ and $\calJ_\infty(\scrH)$ by
\begin{equation*}
    \calJ_2(\scrH)
    := \int_{[0,2\rmv]} \sqrt{\log N(\scrH,d_2,\rho)}\,\mathrm{d}\rho,
    \qquad
    \calJ_\infty(\scrH)
    := \int_{[0,2\rmU]} \log N(\scrH,d_\infty,\rho)\,\mathrm{d}\rho,
\end{equation*}
with respect to the (pseudo)-metrics
\begin{equation*}
    d_2(\theta_1,\theta_2)
    := \sqrt{\IE[\Vnorm{h_{\theta_1}(Z)-h_{\theta_2}(Z)}^2]},
    \qquad
    d_\infty(\theta_1,\theta_2)
    := \norm{h_{\theta_1}-h_{\theta_2}}_\infty .
\end{equation*}

\begin{enumerate}
\item
Let $N\in\IN$. Let $\varepsilon_1,\dots,\varepsilon_N$ be independent $V$-valued random variables satisfying \cref{ass:errorBernstein}.
Let $Z_1,\dots,Z_N$ be i.i.d.\ copies of $Z\sim\IP^Z$, independent of $\varepsilon_1,\dots,\varepsilon_N$, and let $a_1,\dots,a_N$ be real numbers such that $ \max_{i\in\nset{N}} \abs{a_i} \le a_\infty \in (0,\infty)$.
Define the empirical process
\begin{equation}\label{eq:emp1}
    \forall \theta\in\Theta:\quad
    T_{N,1}(\theta)
    := \frac{1}{\sqrt{N}}
    \sum_{i\in\nset{N}} a_i \Vsprod{\varepsilon_i}{h_\theta(Z_i)}.
    \tag{EP1}
\end{equation}
Then there exists a universal constant $M\in(0,\infty)$ such that for all $x\geq 1$,
\begin{equation*}
    \IP\!\left(
        \sup_{\theta\in\Theta} |T_{N,1}(\theta)|
        \ge
        M\!\left(
            \sqrt{\overline{(a\sigma)_N}^2}
            \bigl(\calJ_2(\scrH)+\rmv\sqrt{x}\bigr)
            + \frac{a_\infty\rmB}{\sqrt{N}}
            \bigl(\calJ_\infty(\scrH)+\rmU x\bigr)
        \right)
    \right)
    \le 3\exp(-x),
\end{equation*}
where
\[
    \overline{(a\sigma)_N}^2 := \frac{1}{N}\sum_{i\in\nset{N}} a_i^2 \sigma_i^2 .
\]

\item
Let $V=\IR$.
Let $a_1,\dots,a_N$ be real numbers such that $\max_{i\in\nset{N}} \abs{a_i} \le a_\infty \in (0,\infty)$.
Let $Z_1,\dots,Z_N$ be i.i.d.\ copies of $Z\sim\IP^Z$ and define
\begin{equation}\label{eq:emp2}
    \forall \theta\in\Theta:\quad
    T_{N,2}(\theta)
    := \frac{1}{\sqrt{N}}
    \sum_{i\in\nset{N}} a_i
    \bigl(h_\theta(Z_i)-\IE[h_\theta(Z_i)]\bigr).
    \tag{EP2}
\end{equation}
Then there exists a universal constant $M\in(0,\infty)$ such that for all $x\geq 1$,
\begin{equation*}
    \IP\!\left(
        \sup_{\theta\in\Theta} |T_{N,2}(\theta)|
        \ge
        M\!\left(
            \Bar{a}_N
            \bigl(\calJ_2(\scrH)+\rmv\sqrt{x}\bigr)
            + \frac{a_\infty}{\sqrt{N}}
            \bigl(\calJ_\infty(\scrH)+\rmU x\bigr)
        \right)
    \right)
    \le 3\exp(-x).
\end{equation*}
\end{enumerate}
\end{lemma}

\begin{proof}[Proof of \cref{lem:chaining}]
For both empirical processes we will apply Theorem~3.5 in \cite{Dirksen_2015}. To this end, we verify that the processes
$(T_{N,1}(\theta))_{\theta\in\Theta}$ and $(T_{N,2}(\theta))_{\theta\in\Theta}$
satisfy Condition~3.8 therein, that is, they exhibit mixed sub-Gaussian--exponential tails. We begin with the multiplier process $(T_{N,1}(\theta))_{\theta\in\Theta}$.
Fix arbitrary $\theta_1,\theta_2\in\Theta$, $\lambda\in\IR$, and $i\in\nset{N}$.
By Fubini's theorem,
\begin{multline*}
\IE\brE{\exp(\lambda a_i\Vsprod{\varepsilon_i}{h_{\theta_1}(Z_i)-h_{\theta_2}(Z_i)})}
= \IE\left[\sum_{k\in\IN_0}\frac{\lambda^k}{k!}a_i^k
\Vsprod{\varepsilon_i}{h_{\theta_1}(Z_i)-h_{\theta_2}(Z_i)}^k\right] \\
\le 1
+ \lambda a_i\IE\brE{\Vsprod{\varepsilon_i}{h_{\theta_1}(Z_i)-h_{\theta_2}(Z_i)}}
+ \sum_{k\in\IN_{\ge 2}}\frac{|\lambda|^k}{k!}|a_i|^k
\IE\brE{\IE\brE{|\Vsprod{\varepsilon_i}{h_{\theta_1}(Z_i)-h_{\theta_2}(Z_i)}|^k \mid Z_i}} .
\end{multline*}
Using independence of $Z_i$ and $\varepsilon_i$ and applying \cref{ass:errorBernstein}, the preceding display is bounded by
\begin{align*}
1
+ \frac{\sigma_i^2}{2}
\sum_{k\in\IN_{\ge 2}}
\frac{|\lambda|^k}{k!}
|a_i|^k
k!\rmB^{k-2}
\IE\brE{\Vnorm{h_{\theta_1}(Z_i)-h_{\theta_2}(Z_i)}^{k}} .
\end{align*}
Invoking \cref{eq:chaining_C1}, we further bound this by
\begin{align*}
1
+ \frac{\sigma_i^2}{2}
\lambda^2 |a_i|^2
\IE\brE{\Vnorm{h_{\theta_1}(Z_i)-h_{\theta_2}(Z_i)}^2}
\sum_{k\in\IN_{\ge 2}}
|\lambda|^{k-2}
|a_i|^{k-2}
\rmB^{k-2}
\norm{h_{\theta_1}-h_{\theta_2}}_\infty^{k-2}.
\end{align*}
The remaining sum is a geometric series. Hence, for
$|\lambda| < (a_\infty\rmB d_\infty(\theta_1,\theta_2))^{-1}$,
\begin{align*}
\IE\brE{\exp(\lambda a_i\Vsprod{\varepsilon_i}{h_{\theta_1}(Z_i)-h_{\theta_2}(Z_i)})}
\le
\exp\!\left(
\frac{\lambda^2 a_i^2 \sigma_i^2 d_2^2(\theta_1,\theta_2)}
{2 - 2|\lambda| a_\infty\rmB d_\infty(\theta_1,\theta_2)}
\right),
\end{align*}
where we used that $1+x\le \exp(x)$ for all $x\in\IR$.

Similarly, for
$|\lambda| < \frac{\sqrt{N}}{a_\infty\rmB d_\infty(\theta_1,\theta_2)}$,
independence yields
\begin{multline*}
\IE\left[\exp(\lambda(T_{N,1}(\theta_1)-T_{N,1}(\theta_2)))\right]
= \prod_{i\in\nset{N}}
\IE\left[\exp\left(\frac{\lambda}{\sqrt{N}}a_i
\Vsprod{\varepsilon_i}{h_{\theta_1}(Z_i)-h_{\theta_2}(Z_i)}\right)\right] \\
\le
\exp\!\left(
\frac{\lambda^2 \overline{(a\sigma)_N^2} d_2^2(\theta_1,\theta_2)}
{2 - 2N^{-1/2}|\lambda|a_\infty\rmB d_\infty(\theta_1,\theta_2)}
\right).
\end{multline*}
Applying the exponential Chebyshev inequality, for any $x\in\IR_{\ge 0}$,
\begin{align*}
\IP\left(|T_{N,1}(\theta_1)-T_{N,1}(\theta_2)| \ge x\right)
\le
2\exp\!\left(
\frac{\lambda^2 \overline{(a\sigma)_N^2} d_2^2(\theta_1,\theta_2)}
{2 - 2N^{-1/2}|\lambda|a_\infty\rmB d_\infty(\theta_1,\theta_2)}
-\lambda x
\right).
\end{align*}
Minimizing the right-hand side with respect to $\lambda$, as in the proof of Theorem~3.1.8 in \cite{Gine_Nickl_2021}, yields
\begin{align*}
\IP\left(|T_{N,1}(\theta_1)-T_{N,1}(\theta_2)| \ge x\right)
\le
2\exp\!\left(
-\frac{x^2}{
2\overline{(a\sigma)_N^2}d_2^2(\theta_1,\theta_2)
+ a_\infty\frac{\rmB}{\sqrt{N}} d_\infty(\theta_1,\theta_2)x}
\right).
\end{align*}
Consequently, for any $\theta_1,\theta_2\in\Theta$,
\begin{equation*}
\IP\!\left(
|T_{N,1}(\theta_1)-T_{N,1}(\theta_2)|
\ge
2\sqrt{\overline{(a\sigma)_N^2}}\,d_2(\theta_1,\theta_2)\sqrt{x}
+ \frac{2a_\infty\rmB}{\sqrt{N}}d_\infty(\theta_1,\theta_2)x
\right)
\le 2e^{-x}.
\end{equation*}
Hence, $(T_{N,1}(\theta))_{\theta\in\Theta}$ satisfies Condition~3.8 of \cite{Dirksen_2015}
with metrics
$\Bar d_1 := \frac{2a_\infty\rmB}{\sqrt{N}}d_\infty$
and
$\Bar d_2 := 2\sqrt{\overline{(a\sigma)_N^2}}d_2$. By Theorem~3.5 in \cite{Dirksen_2015}, there exist universal constants $c,C\in(0,\infty)$ such that for any $\theta^\dagger\in\Theta$ and $x\geq 1$,
\begin{multline*}
\IP\!\left(
\sup_{\theta\in\Theta}
|T_{N,1}(\theta)-T_{N,1}(\theta^\dagger)|
>
C\bigl(\gamma_2(\scrH,\Bar d_2)+\gamma_1(\scrH,\Bar d_1)\bigr)
+ c\bigl(\sqrt{x}\Delta_{\Bar d_2}(\scrH)+x\Delta_{\Bar d_1}(\scrH)\bigr)
\right)
\le e^{-x}.
\end{multline*}
Moreover, the diameters satisfy
\begin{equation*}
\Delta_{\Bar d_1}(\scrH)
= \frac{2a_\infty\rmB}{\sqrt{N}}
\sup_{\theta_1,\theta_2\in\Theta} d_\infty(\theta_1,\theta_2)
\le \frac{4a_\infty\rmB\rmU}{\sqrt{N}},
\end{equation*}
and
\begin{equation*}
\Delta_{\Bar d_2}(\scrH)
= 2\sqrt{\overline{(a\sigma)_N^2}}
\sup_{\theta_1,\theta_2\in\Theta} d_2(\theta_1,\theta_2)
\le 4\sqrt{\overline{(a\sigma)_N^2}}\rmv .
\end{equation*}
Furthermore, up to universal constants,
\begin{align*}
\gamma_1(\scrH,\Bar d_1)
&\lesssim \int_{(0,\infty)}\log N(\scrH,\Bar d_1,\rho)\,\mathrm{d}\rho 
= \frac{2a_\infty\rmB}{\sqrt{N}}\calJ_\infty(\scrH).
\end{align*}
Analogously,
\begin{align*}
\gamma_2(\scrH,\Bar d_2)
&\lesssim \int_{(0,\infty)}\log N(\scrH,\Bar d_2,\rho)^{1/2}\,\mathrm{d}\rho 
= 2\sqrt{\overline{(a\sigma)_N^2}}\calJ_2(\scrH).
\end{align*}
Thus, for some universal $M\in(0,\infty)$,
\begin{multline*}
\IP\!\left(
\sup_{\theta\in\Theta}
|T_{N,1}(\theta)-T_{N,1}(\theta^\dagger)|
>
M\!\left(
\sqrt{\overline{(a\sigma)_N^2}}(\calJ_2(\scrH)+\rmv\sqrt{x})
+ \frac{a_\infty\rmB}{\sqrt{N}}(\calJ_\infty(\scrH)+\rmU x)
\right)
\right)
\le e^{-x}.
\end{multline*}
Now, for $x\geq 1$ and $\tau(x)$ specified below,
\begin{align*}
\IP\left(\sup_{\theta\in\Theta}|T_{N,1}(\theta)| > 2\tau(x)\right)
&\le
\IP\left(\sup_{\theta\in\Theta}|T_{N,1}(\theta)-T_{N,1}(\theta^\dagger)| > \tau(x)\right)
+
\IP\left(|T_{N,1}(\theta^\dagger)| > \tau(x)\right).
\end{align*}
Moreover, for all $k\in\IN$,
\begin{align*}
\frac{1}{N}\sum_{i\in\nset{N}}\IE\brE{|a_i\Vsprod{h_{\theta^\dagger}(Z_i)}{\varepsilon_i}|^k}
&\le \frac{1}{2}k!\,\overline{(a\sigma)_N^2}\,\rmv^2\,(a_\infty\rmB\rmU)^{k-2}.
\end{align*}
Therefore, Bernstein's inequality (Lemma~2.2.10 in \cite{Wellner_2023}) implies that for all $x\in(0,\infty)$,
\begin{equation*}
\IP\left(\left|\sum_{i\in\nset{N}}a_i\Vsprod{h_{\theta^\dagger}(Z_i)}{\varepsilon_i}\right|>x\right)
\le
\exp\!\left(
-\frac{x^2}{2N\overline{(a\sigma)_N^2}\rmv^2+2a_\infty\rmB\rmU x}
\right),
\end{equation*}
and hence
\begin{equation*}
\IP\left(
|T_{N,1}(\theta^\dagger)|
>
2\sqrt{\overline{(a\sigma)_N^2}}\rmv\sqrt{x}
+\frac{4}{\sqrt{N}}a_\infty\rmB\rmU x
\right)
\le 2e^{-x}.
\end{equation*}
The claim follows by taking $M\ge 2$ and defining, for $x\geq 1$,
\begin{equation*}
\tau(x)
:= M\!\left(
\sqrt{\overline{(a\sigma)_N^2}}(\calJ_2(\scrH)+\rmv\sqrt{x})
+ \frac{a_\infty\rmB}{\sqrt{N}}(\calJ_\infty(\scrH)+\rmU x)
\right),
\end{equation*}
which proves~(i). Part~(ii) follows by analogous computations and is omitted.
\end{proof}

\subsection{An inequality for infinite series}

\begin{lemma}\label{lem:ExpoSum}
    Let $\rmB>1$. Let $\rma\in\pRzz$, and $\rmb,\rmc\in\pRz$, such that $\rma < \rmb\rmc$. It holds,
    \begin{equation}\label{eq:ExpoSum}
        \sum_{l\in\IN}\rmB^{\rma l}\exp\br{-\rmb\cdot\rmB^{\rmc l}} \leq \frac{\rmb^{-\frac{\rma}{\rmc}}}{\rmc\ln(\rmB)}\Gamma\br{2+\frac{\rma}{\rmc}}\cdot \exp\br{-\tau^\star\rmb},
    \end{equation}
    where $\tau^\star\coloneqq\Gamma\br{2+\frac{\rma}{\rmc}}^{-\frac{\rmc}{\rmc+\rma}}$ and $\Gamma$ denotes the Gamma function.
\end{lemma}

\begin{proof}[Proof of \cref{lem:ExpoSum}]
    \mbox{}
    We follow the proof strategy of Lemma A.2 in \cite{KutriScheichl_2024}. We start with defining for $x\in\pRz$ the function $ g(x)\coloneqq \rmB^{\rma x}\exp\br{-\rmb\cdot\rmB^{\rmc x}}$.
    Observing that 
    \begin{equation*}
       g'(x) = -\ln(\rmB)\cdot
       \left(\rmB^{\rmc x} \rmb\rmc - \rma\right)\cdot g(x),
    \end{equation*}
    we see directly that $g$ is monotonically decreasing, if $\rma = 0$. If $\rma>0$, we have
    \begin{equation*}
        g'(x_o) = 0 \quad\text{ if and only if }\quad
        x_o = \frac{\ln\br{\frac{\rma}{\rmc\rmb}}}{\rmc\ln(\rmB)}.
    \end{equation*}
    As $\rma <\rmb\rmc$ by assumption, we have $x_o < 0$ as well as $\left(\rmB^{\rmc x} \rmb\rmc - \rma\right)>0$ for all $x>x_o$,
    such that $g$ is monotonically decreasing on $[0,\infty)$. Thus, independent of the choice of $\rma\in\pRzz$, we have  
    \begin{align*}
        \sum_{l\in\IN}g(l)
        &\leq \int_{0}^\infty g(x)\rmd x 
        = \frac{1}{\rmc\ln(\rmB)}
        \int_\rmb^\infty
        \br{\frac{x}{\rmb}}^{\frac{\rma}{\rmc}}
        \frac{\rmb}{x}\exp(-x)\rmd x 
        = \frac{\rmb^{1-\frac{\rma}{\rmc}}}{\rmc\ln(\rmB)}
        \int_{b}^\infty x^{\frac{\rma}{\rmc}-1}\exp\br{-x}\rmd x,
    \end{align*}
    where we used the integral transformation
    $\varphi: x\mapsto\frac{\ln\br{\frac{x}{\rmb}}}{\rmc\ln\br{\rmB}}$. In case of $\rma = 0$, the above display is upper bounded by 
    \begin{equation*}
        \frac{1}{\rmc\ln(\rmB)}\int_{b}^\infty \exp\br{-x}\rmd x
        \leq \frac{1}{\rmc\ln(\rmB)}\exp(-\rmb),
    \end{equation*}
    which shows the claim. If $\rma\neq 0$, applying the integral transformation
    $\varphi_p\colon\pRz\ni x\mapsto x^p\in\pRz$, $p\in\pRz$, the last display is upper bounded by 
    \begin{equation*}
        \frac{\rmb^{-\frac{\rma}{\rmc}}}{\rmc\ln(\rmB)}
        \int_{b}^\infty x^{\frac{\rma}{\rmc}}\exp\br{-x}\rmd x
        = \frac{p\rmb^{-\frac{\rma}{\rmc}}}{\rmc\ln(\rmB)}
        \int_{\rmb^{\frac{1}{p}}}^\infty
        x^{p\frac{\rma}{\rmc}+p-1}\exp(-x^p)\rmd x.
    \end{equation*}
    Choosing $p=\frac{\rmc}{\rmc+\rma}< 1$, the last display reads
    \begin{equation*}
        \frac{p\rmb^{-\frac{\rma}{\rmc}}}{\rmc\ln(\rmB)}
        \int_{\rmb^{\frac{1}{p}}}^\infty \exp(-x^p)\rmd x.
    \end{equation*}
    For any $0\leq \tau\leq \min\brK{1,\Gamma(1+1/p)^{-p}}$, we can apply the results of \cite{Alzer_1997} and upper bound the last display by
    \begin{equation*}
         \frac{p\rmb^{-\frac{\rma}{\rmc}}}{\rmc\ln(\rmB)}
         \Gamma\br{1+1/p}\cdot
         \br{1-\br{1-\exp\br{-\tau\rmb}}^{\frac{1}{p}}}.
    \end{equation*}
    By an application of Bernoulli's inequality, the last term is upper bounded by 
    \begin{equation*}
        \frac{\rmb^{-\frac{\rma}{\rmc}}}{\rmc\ln(\rmB)}
        \Gamma\br{2+\frac{\rma}{\rmc}}\cdot
        \exp\br{-\tau^\star\rmb}
    \end{equation*}
    with $\tau^\star\coloneqq\Gamma\br{2+\frac{\rma}{\rmc}}^{-\frac{\rmc}{\rmc+\rma}}\leq1$.
\end{proof}

%% file: misspec.bib
@article{Dirksen_2015,
  author = {Dirksen, Sjoerd},
  title = {Tail bounds via generic chaining},
  journal = {Electronic Journal of Probability},
  volume = {20},
  pages = {1--29},
  year = {2015},
  doi = {10.1214/EJP.v20-3760},
  url = {https://doi.org/10.1214/EJP.v20-3760}
}

@book{Gine_Nickl_2021,
  author = {Gine, Evarist and Nickl, Richard},
  title = {Mathematical Foundations of Infinite-Dimensional Statistical Models},
  series = {Cambridge Series in Statistical and Probabilistic Mathematics},
  publisher = {Cambridge University Press},
  year = {2021},
  address = {Cambridge}
}

@book{Wellner_2023,
  author = {van der Vaart, Aad W. and Wellner, Jon A.},
  title = {Weak Convergence and Empirical Processes},
  subtitle = {With Applications to Statistics},
  edition = {2},
  publisher = {Springer},
  address = {Cham},
  year = {2023},
  series = {Springer Series in Statistics},
  isbn = {978-3-031-29040-4},
  doi = {10.1007/978-3-031-29040-4},
  url = {https://doi.org/10.1007/978-3-031-29040-4}
}

@book{Geer_2000,
  author = {{van de Geer}, Sara},
  title = {Applications of Empirical Process Theory},
  series = {Cambridge Series in Statistical and Probabilistic Mathematics},
  volume = {6},
  publisher = {Cambridge University Press},
  address = {Cambridge},
  year = {2000},
  isbn = {0-521-65002-X}
}

@book{Engletal_2000,
  title = {Regularization of {{Inverse Problems}}},
  author = {Engl, Heinz Werner and Hanke, Martin and Neubauer, A.},
  year = 2000,
  month = mar,
  publisher = {Springer Science \& Business Media},
  isbn = {978-0-7923-6140-4},
  langid = {english}
}

@article{GiordanoWang_2025,
  title = {Statistical Algorithms for Low-Frequency Diffusion Data: {{A PDE}} Approach},
  shorttitle = {Statistical Algorithms for Low-Frequency Diffusion Data},
  author = {Giordano, Matteo and Wang, Sven},
  year = 2025,
  month = jun,
  journal = {The Annals of Statistics},
  volume = {53},
  number = {3},
  pages = {1150--1175},
  publisher = {Institute of Mathematical Statistics},
  issn = {0090-5364, 2168-8966},
  doi = {10.1214/25-AOS2496},
  urldate = {2026-01-09}
}

@article{Haireretal_2014,
  title = {Spectral {{Gaps}} for a {{Metropolis}}–{{Hastings Algorithm}} in {{Infinite Dimensions}}},
  author = {Hairer, Martin and Stuart, Andrew M. and Vollmer, Sebastian J.},
  year = 2014,
  journal = {The Annals of Applied Probability},
  volume = {24},
  number = {6},
  eprint = {24520134},
  eprinttype = {jstor},
  pages = {2455--2490},
  publisher = {Institute of Mathematical Statistics},
  issn = {1050-5164},
  urldate = {2026-01-09}
}

@book{Triebel_1983,
  title = {Theory of {{Function Spaces}}},
  author = {Triebel, Hans},
  year = 1983,
  publisher = {Springer},
  address = {Basel},
  doi = {10.1007/978-3-0346-0416-1},
  urldate = {2025-11-29},
  copyright = {https://www.springernature.com/gp/researchers/text-and-data-mining},
  isbn = {978-3-0346-0416-1},
  langid = {english}
}

@article{Kekkonen_2022,
  title = {Consistency of {{Bayesian}} Inference with {{Gaussian}} Process Priors for a Parabolic Inverse Problem},
  author = {Kekkonen, Hanne},
  year = 2022,
  month = jan,
  journal = {Inverse Problems},
  volume = {38},
  number = {3},
  pages = {035002},
  publisher = {IOP Publishing},
  issn = {0266-5611},
  doi = {10.1088/1361-6420/ac4839},
  urldate = {2025-11-25},
  langid = {english}
}

@book{LionsMagenes_1972,
  title = {Non-{{Homogeneous Boundary Value Problems}} and {{Applications}}},
  author = {Lions, J. L. and Magenes, E.},
  year = 1972,
  publisher = {Springer},
  address = {Berlin, Heidelberg},
  doi = {10.1007/978-3-642-65161-8},
  urldate = {2026-01-12},
  copyright = {http://www.springer.com/tdm},
  isbn = {978-3-642-65161-8}
}

@article{Cotteretal_2013,
  title = {{{MCMC Methods}} for {{Functions}}: {{Modifying Old Algorithms}} to {{Make Them Faster}}},
  shorttitle = {{{MCMC Methods}} for {{Functions}}},
  author = {Cotter, S. L. and Roberts, G. O. and Stuart, A. M. and White, D.},
  year = 2013,
  month = aug,
  journal = {Statistical Science},
  volume = {28},
  number = {3},
  pages = {424--446},
  publisher = {Institute of Mathematical Statistics},
  issn = {0883-4237, 2168-8745},
  doi = {10.1214/13-STS421},
  urldate = {2026-01-08}
}

@book{Evans_2010,
  title = {Partial {{Differential Equations}}},
  author = {Evans, Lawrence C.},
  year = 2010,
  publisher = {American Mathematical Soc.},
  googlebooks = {Xnu0o\_EJrCQC},
  isbn = {978-0-8218-4974-3},
  langid = {english}
}

@book{Robinson_2001,
  title = {Infinite-{{Dimensional Dynamical Systems}}: {{An Introduction}} to {{Dissipative Parabolic PDEs}} and the {{Theory}} of {{Global Attractors}}},
  shorttitle = {Infinite-{{Dimensional Dynamical Systems}}},
  author = {Robinson, James C.},
  year = 2001,
  month = apr,
  publisher = {Cambridge University Press},
  googlebooks = {3e4h1j9WNlwC},
  isbn = {978-0-521-63204-1},
  langid = {english}
}

@book{Nickl_2023,
  author = {Nickl, Richard},
  title = {Bayesian Non-linear Statistical Inverse Problems},
  series = {Zurich Lectures in Advanced Mathematics},
  publisher = {EMS Press},
  address = {Berlin},
  year = {2023},
  isbn = {978-3-98547-053-2},
  doi = {10.4171/zlam/30},
  url = {https://doi.org/10.4171/zlam/30}
}

@book{Aubin_1982,
  title = {Nonlinear {{Analysis}} on {{Manifolds}}. {{Monge-Ampère Equations}}},
  author = {Aubin, Thierry},
  editor = {Artin, M. and Chern, S. S. and Doob, J. L. and Grothendieck, A. and Heinz, E. and Hirzebruch, F. and Hörmander, L. and Lane, S. Mac and Magnus, W. and Moore, C. C. and Moser, J. K. and Nagata, M. and Schmidt, W. and Scott, D. S. and Tits, J. and Van Der Waerden, B. L. and Berger, M. and Eckmann, B. and Varadhan, S. R. S.},
  year = 1982,
  series = {Grundlehren Der Mathematischen {{Wissenschaften}}},
  volume = {252},
  publisher = {Springer},
  address = {New York, NY},
  doi = {10.1007/978-1-4612-5734-9},
  urldate = {2026-01-30},
  copyright = {http://www.springer.com/tdm},
  isbn = {978-1-4612-5734-9},
  langid = {english}
}

@misc{KutriScheichl_2024,
  title = {Dirichlet-{{Neumann Averaging}}: {{The DNA}} of {{Efficient Gaussian Process Simulation}}},
  shorttitle = {Dirichlet-{{Neumann Averaging}}},
  author = {Kutri, Robert and Scheichl, Robert},
  year = 2024,
  month = dec,
  number = {arXiv:2412.07929},
  eprint = {2412.07929},
  primaryclass = {stat},
  publisher = {arXiv},
  doi = {10.48550/arXiv.2412.07929},
  urldate = {2026-01-22},
  archiveprefix = {arXiv}
}

@article{Andrieuetal_2010,
  title = {Particle {{Markov Chain Monte Carlo Methods}}},
  author = {Andrieu, Christophe and Doucet, Arnaud and Holenstein, Roman},
  year = 2010,
  month = jun,
  journal = {Journal of the Royal Statistical Society Series B: Statistical Methodology},
  volume = {72},
  number = {3},
  pages = {269--342},
  issn = {1369-7412},
  doi = {10.1111/j.1467-9868.2009.00736.x},
  urldate = {2026-01-22}
}

@article{ChristenFox_2005,
  title = {Markov {{Chain Monte Carlo Using}} an {{Approximation}}},
  author = {Christen, J. Andrés and Fox, Colin},
  year = 2005,
  journal = {Journal of Computational and Graphical Statistics},
  volume = {14},
  number = {4},
  eprint = {27594150},
  eprinttype = {jstor},
  pages = {795--810},
  publisher = {[American Statistical Association, Taylor \& Francis, Ltd., Institute of Mathematical Statistics, Interface Foundation of America]},
  issn = {1061-8600},
  urldate = {2026-01-22}
}

@article{AndrieuRoberts_2009,
  title = {The {{Pseudo-Marginal Approach}} for {{Efficient Monte Carlo Computations}}},
  author = {Andrieu, Christophe and Roberts, Gareth O.},
  year = 2009,
  journal = {The Annals of Statistics},
  volume = {37},
  number = {2},
  eprint = {30243645},
  eprinttype = {jstor},
  pages = {697--725},
  publisher = {Institute of Mathematical Statistics},
  issn = {0090-5364},
  urldate = {2026-01-22}
}

@misc{CastreNickl_2026,
  title = {On Gradient Stability in Nonlinear {{PDE}} Models and Inference in Interacting Particle Systems},
  author = {Castre, Aurélien and Nickl, Richard},
  year = 2026,
  month = jan,
  number = {arXiv:2601.10326},
  eprint = {2601.10326},
  primaryclass = {math},
  publisher = {arXiv},
  doi = {10.48550/arXiv.2601.10326},
  urldate = {2026-01-22},
  archiveprefix = {arXiv}
}

@book{Vaart_1998,
  title = {Asymptotic {{Statistics}}},
  author = {{van der Vaart}, Aad},
  year = 1998,
  series = {Cambridge {{Series}} in {{Statistical}} and {{Probabilistic Mathematics}}},
  publisher = {Cambridge University Press},
  address = {Cambridge},
  doi = {10.1017/CBO9780511802256},
  urldate = {2026-01-14},
  isbn = {978-0-521-78450-4}
}

@book{Geer_2000a,
  title = {Empirical {{Processes}} in {{M-Estimation}}},
  author = {{van de Geer}, Sara A.},
  year = 2000,
  month = jan,
  publisher = {Cambridge University Press},
  googlebooks = {2DYoMRz\_0YEC},
  isbn = {978-0-521-65002-1},
  langid = {english}
}

@article{Siebel_2025,
  author = {Siebel, Maximilian},
  title = {Convergence Rates for the Maximum A Posteriori Estimator in PDE-Regression Models with Random Design},
  journal = {SIAM/ASA Journal on Uncertainty Quantification},
  year = {2025},
  pages = {1862--1903},
  doi = {10.1137/25M1744526}
}

@article{Geer_2001,
  author = {{van de Geer}, Sara},
  title = {Least squares estimation with complexity penalties},
  journal = {Mathematical Methods of Statistics},
  volume = {10},
  number = {3},
  pages = {355--374},
  year = {2001}
}

@article{Nickl_Wang_2024,
  author = {Nickl, Richard and Wang, Sven},
  title = {On polynomial-time computation of high-dimensional posterior measures by Langevin-type algorithms},
  journal = {Journal of the European Mathematical Society},
  volume = {26},
  number = {3},
  pages = {1031--1112},
  year = {2024},
  doi = {10.4171/jems/1304}
}

@article{Bohr_Nickl_2024,
  author = {Bohr, Jan and Nickl, Richard},
  title = {On log-concave approximations of high-dimensional posterior measures and stability properties in non-linear inverse problems},
  journal = {Annales de l'Institut Henri Poincaré Probabilités et Statistiques},
  volume = {60},
  number = {4},
  pages = {2619--2667},
  year = {2024},
  doi = {10.1214/23-aihp1397}
}

@article{Giordano_Nickl_2020,
  author = {Giordano, Matteo and Nickl, Richard},
  title = {Consistency of Bayesian inference with Gaussian process priors in an elliptic inverse problem},
  journal = {Inverse Problems},
  volume = {36},
  number = {8},
  pages = {085001},
  year = {2020},
  doi = {10.1088/1361-6420/ab7d2a}
}

@misc{Nickl_2025_BvMRDE,
  author = {Nickl, Richard},
  title = {Bernstein-von Mises Theorems for Time Evolution Equations},
  year = {2024},
  eprint = {2407.14781},
  archivePrefix = {arXiv},
  primaryClass = {math.ST},
  url = {https://arxiv.org/abs/2407.14781}
}

@article{Nickl_vdGeer_Wang_2020,
  author = {Nickl, Richard and {van de Geer}, Sara and Wang, Sven},
  title = {Convergence Rates for Penalized Least Squares Estimators in PDE Constrained Regression Problems},
  journal = {SIAM/ASA Journal on Uncertainty Quantification},
  volume = {8},
  number = {1},
  pages = {374--413},
  year = {2020},
  doi = {10.1137/18M1236137}
}

@article{NicklTiti_2024,
  author = {Nickl, Richard and Titi, Edriss S.},
  title = {On Posterior Consistency of Data Assimilation with Gaussian Process Priors: The 2D Navier--Stokes Equations},
  journal = {Annals of Statistics},
  volume = {52},
  number = {4},
  pages = {1825--1844},
  year = {2024},
  doi = {10.1214/24-AOS2427}
}

@misc{Konen_Nickl_2025,
  author = {Konen, Dimiri and Nickl, Richard},
  title = {Data Assimilation with the 2D Navier-Stokes Equations: Optimal Gaussian Asymptotics for the Posterior Measure},
  year = {2025},
  eprint = {2507.18279},
  archivePrefix = {arXiv},
  primaryClass = {math},
  doi = {10.48550/arXiv.2507.18279},
  url = {https://arxiv.org/abs/2507.18279}
}

@article{Alzer_1997,
  author = {Alzer, Horst},
  title = {On Some Inequalities for the Incomplete Gamma Function},
  journal = {Mathematics of Computation},
  volume = {66},
  number = {218},
  pages = {771--778},
  year = {1997}
}

@article {Norets_2015,
    AUTHOR = {Norets, Andriy},
     TITLE = {Bayesian regression with nonparametric heteroskedasticity},
   JOURNAL = {J. Econometrics},
  FJOURNAL = {Journal of Econometrics},
    VOLUME = {185},
      YEAR = {2015},
    NUMBER = {2},
     PAGES = {409--419},
      ISSN = {0304-4076,1872-6895},
   MRCLASS = {62F15 (62F12 62G05 62J05)},
  MRNUMBER = {3311830},
MRREVIEWER = {Todd\ Alan\ Kuffner},
       DOI = {10.1016/j.jeconom.2014.12.006},
       URL = {https://doi.org/10.1016/j.jeconom.2014.12.006},
}

@article {kleijn_2012,
    AUTHOR = {Kleijn, B. J. K. and {van der Vaart}, Aad},
     TITLE = {The {B}ernstein-{V}on-{M}ises theorem under misspecification},
   JOURNAL = {Electron. J. Stat.},
  FJOURNAL = {Electronic Journal of Statistics},
    VOLUME = {6},
      YEAR = {2012},
     PAGES = {354--381},
      ISSN = {1935-7524},
   MRCLASS = {62F15 (62F12 62F25)},
  MRNUMBER = {2988412},
MRREVIEWER = {Gang\ Shen},
       DOI = {10.1214/12-EJS675},
       URL = {https://doi.org/10.1214/12-EJS675},
}

@book{GhosalvanderVaart_2017,
  title = {Fundamentals of {{Nonparametric Bayesian Inference}}},
  author = {Ghosal, Subhashis and {van der Vaart}, Aad},
  year = 2017,
  series = {Cambridge {{Series}} in {{Statistical}} and {{Probabilistic Mathematics}}},
  publisher = {Cambridge University Press},
  address = {Cambridge},
  doi = {10.1017/9781139029834},
  urldate = {2026-01-09},
  isbn = {978-0-521-87826-5}
}

@article {stuart_2010,
    author = {Stuart, A. M.},
     TITLE = {Inverse problems: a {B}ayesian perspective},
   JOURNAL = {Acta Numer.},
  FJOURNAL = {Acta Numerica},
    VOLUME = {19},
      YEAR = {2010},
     PAGES = {451--559},
      ISSN = {0962-4929,1474-0508},
   MRCLASS = {65J22 (35R25 35R30 62C10 65J20)},
  MRNUMBER = {2652785},
MRREVIEWER = {Ruben\ D.\ Spies},
       DOI = {10.1017/S0962492910000061},
       URL = {https://doi.org/10.1017/S0962492910000061},
}

@article {kleijn_2006,
    AUTHOR = {Kleijn, B. J. K. and {van der Vaart}, Aad},
     TITLE = {Misspecification in infinite-dimensional {B}ayesian
              statistics},
   JOURNAL = {Ann. Statist.},
  FJOURNAL = {The Annals of Statistics},
    VOLUME = {34},
      YEAR = {2006},
    NUMBER = {2},
     PAGES = {837--877},
      ISSN = {0090-5364,2168-8966},
   MRCLASS = {62G07 (62F05 62G20)},
  MRNUMBER = {2283395},
       DOI = {10.1214/009053606000000029},
       URL = {https://doi.org/10.1214/009053606000000029},
}

@article {white_1982,
    AUTHOR = {White, Halbert},
     TITLE = {Maximum likelihood estimation of misspecified models},
   JOURNAL = {Econometrica},
  FJOURNAL = {Econometrica. Journal of the Econometric Society},
    VOLUME = {50},
      YEAR = {1982},
    NUMBER = {1},
     PAGES = {1--25},
      ISSN = {0012-9682,1468-0262},
   MRCLASS = {62P20},
  MRNUMBER = {640163},
MRREVIEWER = {D.\ E. A. Giles},
       DOI = {10.2307/1912526},
       URL = {https://doi.org/10.2307/1912526},
}

@article {berk_1966,
    AUTHOR = {Berk, Robert H.},
     TITLE = {Limiting behavior of posterior distributions when the model is
              incorrect},
   JOURNAL = {Ann. Math. Statist.},
  FJOURNAL = {Annals of Mathematical Statistics},
    VOLUME = {37},
      YEAR = {1966},
     PAGES = {51--58; correction, ibid. 745--746},
      ISSN = {0003-4851},
   MRCLASS = {62.35},
  MRNUMBER = {189176},
MRREVIEWER = {J.\ Sacks},
       DOI = {10.1214/aoms/1177699477},
       URL = {https://doi.org/10.1214/aoms/1177699477},
}

@book{Temam_1997,
  title = {Infinite-{{Dimensional Dynamical Systems}} in {{Mechanics}} and {{Physics}}},
  author = {Temam, Roger},
  editor = {Marsden, J. E. and Sirovich, L. and John, F.},
  year = 1997,
  series = {Applied {{Mathematical Sciences}}},
  volume = {68},
  publisher = {Springer},
  address = {New York, NY},
  doi = {10.1007/978-1-4612-0645-3},
  urldate = {2026-01-08},
  copyright = {http://www.springer.com/tdm},
  isbn = {978-1-4612-0645-3}
}

@book{Strogatz_2018,
  title = {Nonlinear {{Dynamics}} and {{Chaos}}: {{With Applications}} to {{Physics}}, {{Biology}}, {{Chemistry}}, and {{Engineering}}},
  shorttitle = {Nonlinear {{Dynamics}} and {{Chaos}}},
  author = {Strogatz, Steven H.},
  year = 2018,
  month = may,
  edition = {2},
  publisher = {CRC Press},
  address = {Boca Raton},
  doi = {10.1201/9780429492563},
  isbn = {978-0-429-49256-3}
}

@book{ConstantinFoias_1989,
  title = {Navier-{{Stokes Equations}}},
  author = {Constantin, Peter and Foias, Ciprian},
  year = 1989,
  month = jul,
  series = {Chicago {{Lectures}} in {{Mathematics}}},
  publisher = {University of Chicago Press},
  address = {Chicago, IL},
  urldate = {2026-01-08},
  isbn = {978-0-226-11549-8},
  langid = {english}
}

@article{Arridgeetal_2019,
  title = {Solving Inverse Problems Using Data-Driven Models},
  author = {Arridge, Simon and Maass, Peter and Öktem, Ozan and Schönlieb, Carola-Bibiane},
  year = 2019,
  month = may,
  journal = {Acta Numerica},
  volume = {28},
  pages = {1--174},
  issn = {0962-4929, 1474-0508},
  doi = {10.1017/S0962492919000059},
  urldate = {2026-01-08},
  langid = {english}
}

@book{Kaltenbacheretal_2008,
  title = {Iterative {{Regularization Methods}} for {{Nonlinear Ill-Posed Problems}}},
  author = {Kaltenbacher, Barbara and Neubauer, Andreas and Scherzer, Otmar},
  year = 2008,
  month = sep,
  publisher = {De Gruyter},
  doi = {10.1515/9783110208276},
  urldate = {2026-01-08},
  copyright = {De Gruyter expressly reserves the right to use all content for commercial text and data mining within the meaning of Section 44b of the German Copyright Act.},
  isbn = {978-3-11-020827-6},
  langid = {english}
}

@book {Girault_Raviart_1986,
    AUTHOR = {Girault, Vivette and Raviart, Pierre-Arnaud},
     TITLE = {Finite element methods for {N}avier-{S}tokes equations},
    SERIES = {Springer Series in Computational Mathematics},
    VOLUME = {5},
      NOTE = {Theory and algorithms},
 PUBLISHER = {Springer-Verlag, Berlin},
      YEAR = {1986},
     PAGES = {x+374},
      ISBN = {3-540-15796-4},
   MRCLASS = {65N30 (65-02 76-08)},
  MRNUMBER = {851383},
MRREVIEWER = {Max\ D.\ Gunzburger},
       DOI = {10.1007/978-3-642-61623-5},
       URL = {https://doi.org/10.1007/978-3-642-61623-5},
}

@article {reiss_2008,
    AUTHOR = {Rei\ss, Markus},
     TITLE = {Asymptotic equivalence for nonparametric regression with
              multivariate and random design},
   JOURNAL = {Ann. Statist.},
  FJOURNAL = {The Annals of Statistics},
    VOLUME = {36},
      YEAR = {2008},
    NUMBER = {4},
     PAGES = {1957--1982},
      ISSN = {0090-5364,2168-8966},
   MRCLASS = {62G08 (62B15 62G20)},
  MRNUMBER = {2435461},
       DOI = {10.1214/07-AOS525},
       URL = {https://doi.org/10.1214/07-AOS525},
}

@article {vollmer_2013,
    AUTHOR = {Vollmer, Sebastian J.},
     TITLE = {Posterior consistency for {B}ayesian inverse problems through
              stability and regression results},
   JOURNAL = {Inverse Problems},
  FJOURNAL = {Inverse Problems. An International Journal on the Theory and
              Practice of Inverse Problems, Inverse Methods and Computerized
              Inversion of Data},
    VOLUME = {29},
      YEAR = {2013},
    NUMBER = {12},
     PAGES = {125011, 32},
      ISSN = {0266-5611,1361-6420},
   MRCLASS = {35R30 (62F15 62G08 65J22 65N21)},
  MRNUMBER = {3141858},
MRREVIEWER = {Vyacheslav\ F.\ Gubarev},
       DOI = {10.1088/0266-5611/29/12/125011},
       URL = {https://doi.org/10.1088/0266-5611/29/12/125011},
}

@article {Huillier_2023,
    AUTHOR = {L'Huillier, Alice and Travis, Luke and Castillo, Isma\"el and
              Ray, Kolyan},
     TITLE = {Semiparametric inference using fractional posteriors},
   JOURNAL = {J. Mach. Learn. Res.},
  FJOURNAL = {Journal of Machine Learning Research (JMLR)},
    VOLUME = {24},
      YEAR = {2023},
     PAGES = {Paper No. [389], 61},
      ISSN = {1532-4435,1533-7928},
   MRCLASS = {62G08 (62G20)},
  MRNUMBER = {4720845},
MRREVIEWER = {Matteo\ Giordano},
}

@incollection {Grunwald_2012,
    AUTHOR = {Gr\"unwald, Peter},
     TITLE = {The safe {B}ayesian: learning the learning rate via the
              mixability gap},
 BOOKTITLE = {Algorithmic learning theory},
    SERIES = {Lecture Notes in Comput. Sci.},
    VOLUME = {7568},
     PAGES = {169--183},
 PUBLISHER = {Springer, Heidelberg},
      YEAR = {2012},
      ISBN = {978-3-642-34106-9},
   MRCLASS = {62F15 (62H30 68T05)},
  MRNUMBER = {3042889},
       DOI = {10.1007/978-3-642-34106-9\_16},
       URL = {https://doi.org/10.1007/978-3-642-34106-9_16},
}

@article {miller_2019,
    AUTHOR = {Miller, Jeffrey W. and Dunson, David B.},
     TITLE = {Robust {B}ayesian inference via coarsening},
   JOURNAL = {J. Amer. Statist. Assoc.},
  FJOURNAL = {Journal of the American Statistical Association},
    VOLUME = {114},
      YEAR = {2019},
    NUMBER = {527},
     PAGES = {1113--1125},
      ISSN = {0162-1459,1537-274X},
   MRCLASS = {62H30 (62F15)},
  MRNUMBER = {4011766},
       DOI = {10.1080/01621459.2018.1469995},
       URL = {https://doi.org/10.1080/01621459.2018.1469995},
}

@article {Bhattacharya_2019,
    AUTHOR = {Bhattacharya, Anirban and Pati, Debdeep and Yang, Yun},
     TITLE = {Bayesian fractional posteriors},
   JOURNAL = {Ann. Statist.},
  FJOURNAL = {The Annals of Statistics},
    VOLUME = {47},
      YEAR = {2019},
    NUMBER = {1},
     PAGES = {39--66},
      ISSN = {0090-5364,2168-8966},
   MRCLASS = {62G07 (62F15 62G20)},
  MRNUMBER = {3909926},
MRREVIEWER = {Ao\ Yuan},
       DOI = {10.1214/18-AOS1712},
       URL = {https://doi.org/10.1214/18-AOS1712},
}

@book {Batchelor_1999,
    AUTHOR = {Batchelor, G. K.},
     TITLE = {An introduction to fluid dynamics},
    SERIES = {Cambridge Mathematical Library},
   EDITION = {paperback},
 PUBLISHER = {Cambridge University Press, Cambridge},
      YEAR = {1999},
     PAGES = {xviii+615},
      ISBN = {0-521-66396-2},
   MRCLASS = {76-01},
  MRNUMBER = {1744638},
}

@article {gine_2011,
    AUTHOR = {Gin\'e, Evarist and Nickl, Richard},
     TITLE = {Rates of contraction for posterior distributions in
              {$L^r$}-metrics, {$1\leq r\leq\infty$}},
   JOURNAL = {Ann. Statist.},
  FJOURNAL = {The Annals of Statistics},
    VOLUME = {39},
      YEAR = {2011},
    NUMBER = {6},
     PAGES = {2883--2911},
      ISSN = {0090-5364,2168-8966},
   MRCLASS = {62G20 (62F15 62G07 62G08)},
  MRNUMBER = {3012395},
MRREVIEWER = {Thomas\ R.\ Boucher},
       DOI = {10.1214/11-AOS924},
       URL = {https://doi.org/10.1214/11-AOS924},
}

@book {Kaipio_Somersalo_2005,
    AUTHOR = {Kaipio, Jari and Somersalo, Erkki},
     TITLE = {Statistical and computational inverse problems},
    SERIES = {Applied Mathematical Sciences},
    VOLUME = {160},
 PUBLISHER = {Springer-Verlag, New York},
      YEAR = {2005},
     PAGES = {xvi+339},
      ISBN = {0-387-22073-9},
   MRCLASS = {65-01 (60G60 62-01 62G05 62H12 65J22 65L09 65N21)},
  MRNUMBER = {2102218},
}
